\documentclass[12pt]{article}
\usepackage[margin=1.1in]{geometry}
\usepackage[english]{babel}
\usepackage[numbers]{natbib}
\usepackage{amsmath}
\usepackage{amssymb}
\usepackage{amsthm}
\usepackage{amscd}
\usepackage{bbm}
\usepackage{enumerate}
\usepackage{esvect}
\usepackage[inline]{asymptote}
\usepackage{hyperref}
\usepackage{cleveref}
\usepackage{pgf,tikz}
\usepackage{csquotes}
\usetikzlibrary{arrows}
\usepackage{lipsum}
\usepackage{dsfont}

\usepackage{authblk}

\newtheorem{theorem}{Theorem}[section]
\newtheorem{lemma}[theorem]{Lemma}

\newtheorem{corollary}[theorem]{Corollary}
\newtheorem{claim}[theorem]{Claim}

\theoremstyle{definition}
\newtheorem{definition}[theorem]{Definition}

\theoremstyle{remark}
\newtheorem{remark}[theorem]{Remark}

\crefname{subsection}{Subsection}{Subsections}

\crefname{claim}{Claim}{Claims}

\newcommand{\pref}[1]{(\ref{#1})}

\newcommand{\E}{\mathbb{E}}
\newcommand{\Var}{\text{Var}}
\newcommand{\rhg}{\mathcal{H}}
\newcommand{\pg}{\Psi}

\DeclareMathOperator*{\argmax}{\text{argmax}}
\DeclareMathOperator*{\argmin}{\text{argmin}}

\renewcommand{\b}{\big}
\newcommand{\Bc}{\mathcal{B}_C}

\newcounter{condition}
\newcommand{\condition}[1]{
\unskip\refstepcounter{condition}($\ast$\arabic{condition})\label{#1}}

\newcommand{\refcondition}[1]{
\unskip($\ast$\ref{#1})\unskip
}

\newcommand{\1}{\mathds{1}}
\newcommand{\ind}[1]{\1\{#1\}}

\newcommand{\defeq}{\triangleq}

\newcommand{\proj}{\mathrm{Proj}}
\newcommand{\projw}{\mathrm{Proj}_W}
\newcommand{\Cli}{\mathrm{Cli}}

\allowdisplaybreaks

\title{Partial and Exact Recovery of a Random Hypergraph from its Graph Projection}

\author{Guy Bresler\thanks{Department of EECS at MIT. Email: guy@mit.edu.}, Chenghao Guo\thanks{Department of EECS at MIT. Email: chenghao@mit.edu.}, Yury Polyanskiy\thanks{Department of EECS at MIT. Email: yp@mit.edu.}, Andrew Yao\thanks{Department of EECS at MIT. Email: ajyao@mit.edu.}}

\date{\today}

\begin{document}
\maketitle

\begin{abstract}
Consider a $d$-uniform random hypergraph on $n$ vertices in which hyperedges are included iid 
so that the average degree is $n^\delta$. The projection of a hypergraph is a graph on the same $n$ vertices where an edge connects two vertices if and only if they belong to some hyperedge. The goal is to reconstruct the hypergraph given its projection. 
An earlier work of~\citep{reconstruct} showed that exact recovery for $d=3$ is possible if and only if $\delta < 2/5$.
This work completely resolves the question for all values of $d$ for both exact and partial recovery and for both cases of whether multiplicity information about each edge is available or not. In addition, we show that the reconstruction fidelity undergoes an all-or-nothing transition at a threshold. In particular, this resolves all conjectures from~\citep{reconstruct}.
\end{abstract}

\section{Introduction}
\label{sec:intro}

The overall goal of community detection is recovering clusters of vertices given observations of noisy associations between them. The most popular version is that of pairwise associations~\citep{newman2006modularity}: in friendship graphs, social networks, co-purchase data, etc., an edge between two vertices appears more often when vertices belong to the same latent cluster. However, oftentimes the interactions are of higher order. For instance, in scientific communities, the data comes in the form of co-authorship~\citep{newman2004coauthorship}, which is a form of multi-vertex association relation. A statistical model for this kind of data, known as the hypergraph stochastic block model (HSBM), has been actively investigated recently, starting from~\citep{ghoshdastidar2014consistency,angelini2015spectral}.

This paper focuses on a particular subproblem in this general research area: recovery of the multi-vertex association data (i.e. a hypergraph) from the pairwise association data (i.e. a graph). As an example, consider the case of a large group of people working together on multiple projects (project membership becomes the multi-association data, a hypergraph), but due to data collection limitations we only observe the existence of email exchanges between pairs of co-workers \citep{klimt2004introducing}, which for simplicity we further simplify into a graph with an edge between any pair of people who exchanged an email. Clearly, the mapping from a hypergraph to a graph generally loses information. The question we study in this work is how well we can infer the original hypergraph given only the \textit{projected graph}, i.e. a graph that connects any pair of vertices that belong to some hyperedge. 

The problem of recovery from the projected graph has been studied in the last few years, both theoretically and empirically. \cite{graphstohypergraphs} assumes access to a sample of a hypergraph from the underlying distribution and proposes a scoring method to select hyperedges based on their similarity to the sampled hypergraph. Another algorithmic approach is to sample from the posterior distribution given the projected graph \cite{young2021hypergraph,lizotte2023hypergraph}. In \cite{relationallearning} a large foundational model is used to recover a weighted hypergraph from a sample of its hyperedges, where each hyperedge is assigned a probability proportional to its weight.

Theoretically, interest in reconstructing a hypergraph from its projection arose following the work of \cite{hypersbm}, which proposed to solve the HSBM community detection task by leveraging well-known graph algorithms. This motivated~\cite{reconstruct} to investigate the amount of information loss introduced by projecting the hypergraph data to graph data. Somewhat surprisingly, Subsection 1.2.1 \textit{ibid} demonstrated no loss, thus validating this general approach to HSBM. 

Note that as the hypergraph becomes more dense (contains more hyperedges), the recovery task becomes harder. While in HSBM the hypergraph is so sparse that the projection can be ``undone'' without error, a 
more general question of exact recovery threshold was only answered in~\cite{reconstruct} for $d=3$-uniform hypergraphs. In this work, we not only settle the value of the threshold for all $d$, but also establish the value of the (generally higher) threshold for partial recovery of the hypergraph, in which the algorithm is allowed to miss a fraction of the hyperedges. 

In addition to finding the value of the threshold, we also establish a so-called ``all-or-nothing'' phenomena for our problem.
It was previously demonstrated in \cite{graphmatch}, which considers exact and partial recovery for graph matching and exhibits
that at a threshold the best algorithm transitions from recovering almost all of the assignment to almost none of the assignment. Even earlier, this was demonstrated in the problem of sparse linear regression \cite{sparselinear}. Similarly, we demonstrate that for partial recovery, the fraction of hyperedges recovered transitions from $1-o_n(1)$ to $o_n(1)$ while for exact recovery, the probability of success transitions from $1-o_n(1)$ to $o_n(1)$. 

To understand what \textit{threshold} we are talking about, let us specify the random hypergraph model that we consider. It is a natural extension of the Erd\H{o}s-R\'enyi random graph model $G(n, p)$. Suppose that $d\geq 3$ and $n$ is a positive integer. In this paper we often denote the set of $d$-uniform hypergraphs with vertex set $[n]$ as $\{0, 1\}^{\binom{[n]}{d}}$. Letting $\delta\in (-\infty, d-1)$ and taking $$p\propto (1+o_n(1))n^{-d+1+\delta}\,,$$ the random $d$-uniform hypergraph $\mathcal{H}$ is sampled from the product measure $\text{Ber}(p)^{\otimes \binom{[n]}{d}}$. So the recovery threshold corresponds to the largest value of $\delta$ for which recovery is possible. See Table~\ref{tab:result-compare} for a summary of previous and new results.

For partial recovery, the efficient algorithm that outputs the set of cliques in the projection of $\mathcal{H}$ achieves a partial recovery loss of $o_n(1)$ when partial recovery is possible, see \Cref{subsec:maxcover}. For exact recovery, the MAP algorithm is efficient with high probability when exact recovery is possible, see \cite[Theorem 10]{reconstruct}.

Notice that the hypergraph reconstruction problem is a particular example of a planted constraint satisfaction problem (CSP): the planted (latent) assignment is the hypergraph and the set of constraints is defined by the observed edges in the projection. The problem of reconstruction given the weighted projection is also a linear inverse problem, since we can describe the problem as $AH=W$, where $A\in\{0,1\}^{\binom{n}{2}\times\binom{n}{d}}$ is a fixed (non-random) matrix encoding the incidence relation between edges and hyperedges, $H\in\{0,1\}^{\binom{n}{d}\times 1}$ is the iid Bernoulli planted hypergraph, and $W\in \mathbb{Z}_{\geq 0}^{\binom{n}{2}\times 1}$ is the observed weighted projection. (For the unweighted projection, we have $W\in\{0, 1\}^{\binom{n}{2}\times 1}$ and the problem is a generalized linear inverse problem: $\min(AH, 1) = W$.) Thus, our work can be seen as deriving sharp thresholds for exact and partial recovery of these planted CSPs, contributing to a long line of work of \cite{quietsolutions, planted_csp, glm, amp} and many others. 

\subsection{Main results}
\label{subsec:main}

We consider two recovery objectives given the projection of a random hypergraph $H$. In \textit{exact recovery} the algorithm is required to precisely recover the hypergraph. The metric is the probability of correct recovery. In \textit{partial recovery}, the algorithm is required to produce a hypergraph $\widehat{\mathcal{H}}$ with a small \textit{partial recovery loss}, which equals the ratio of incorrect hyperedges (either missing or spurious) to the total number of true hyperedges. See \Cref{subsubsec:partial} for more precise definitions.

\begin{theorem}[Partial Recovery]
\label{thm:mainpartial}
If $\delta<\frac{d-1}{d+1}$ then the partial recovery loss is $o_n(1)$ and if $\delta>\frac{d-1}{d+1}$ then the partial recovery loss is $1-o_n(1)$.
\end{theorem}
This theorem statement follows by combining \Cref{corr:dense}, \Cref{thm:epthreshold}, and \Cref{thm:apimpossible}. We prove \Cref{thm:mainpartialweighted} in \Cref{subsec:weightedpartial} using a similar framework and we prove \Cref{thm:mainexact,thm:mainexactweighted} in \Cref{sec:exactproof}.

Next, we also study the problem of recovering from a \textit{weighted projection graph}, which counts the multiplicity of each edge inside true hyperedges thus making the recovery problem easier. However, as the next result shows, this extra information is not able to shift the partial recovery threshold.

\begin{theorem}[Partial Recovery for Weighted Projection]
\label{thm:mainpartialweighted}
If $\delta<\frac{d-1}{d+1}$ then the weighted partial recovery loss is $o_n(1)$ and if $\delta>\frac{d-1}{d+1}$ then the weighted partial recovery loss is $1-o_n(1)$.
\end{theorem}

\begin{theorem}[Exact Recovery]
\label{thm:mainexact}
Suppose $3\leq d\leq 5$. If $\delta<\frac{2d-4}{2d-1}$, then the probability of exact recovery is $1-o_n(1)$. Conversely, if $\delta>\frac{2d-4}{2d-1}$, then the probability of exact recovery is $o_n(1)$. 

Suppose $d\geq 5$. If $\delta<\frac{d-1}{d+1}$, then the probability of exact recovery is $1-o_n(1)$. Conversely, if $\delta>\frac{d-1}{d+1}$, then the probability of exact recovery is $o_n(1)$.
\end{theorem}

In particular, \Cref{thm:mainexact} verifies the conjecture from {\cite[Appendix C]{reconstruct}} that the exact recovery threshold is $\frac{2d-4}{2d-1}$ for $d=4, 5$.

Going to recovery from a weighted projected graph, surprisingly we discover that unlike in partial recovery, the side information is able to move the threshold, but only for $d=3,4$.

\begin{theorem}[Exact Recovery for Weighted Projection]
\label{thm:mainexactweighted}
Suppose $d\geq 3$. If $\delta<\frac{d-1}{d+1}$, then the probability of weighted exact recovery is $1-o_n(1)$. If $\delta>\frac{d-1}{d+1}$, then the probability of weighted exact recovery is $o_n(1)$.
\end{theorem}

We remark that for the unweighted projection when $d=3,4$, there are two separate (both all-or-nothing) thresholds of $\frac{d-1}{d+1}$ and $\frac{2d-4}{2d-1}$ for partial recovery and exact recovery, respectively; when $d\geq 5$, both thresholds are at $\frac{d-1}{d+1}$. Interestingly, for the weighted projection, there is a single threshold of $\frac{d-1}{d+1}$ for both exact and partial recovery. The reason for this is that the bottleneck for exact recovery, which is called an ambiguous graph, in the two models are different. We will discuss the definition of ambiguous graphs in \Cref{subsec:ambiguous}.

The contributions of this paper as compared to previously established exact recovery thresholds for $d\geq 4$ are summarized in  \Cref{tab:result-compare}.
\begin{table}[!ht]
    \centering
\begin{tabular}{ |c|c|c|c| } 
 \hline
 $d$ & Previous work (exact recovery) & New (exact recovery) & New  (partial recovery)\\ 
\hline 
$3$ & $2/5$ & $2/5$ & $1/2$ \\
 $4$ & $[1/2, 4/7]$ & $4/7$ & $3/5$ \\ 
 
 $5$ & $[1/2, 2/3]$ & $ 2/3$ & $2/3$ \\
 
 $\geq 6$ & $[\frac{d-3}{d}, \frac{d^2-d-2}{d^2-d+2}]$ & $\frac{d-1}{d+1}$ & $\frac{d-1}{d+1}$ \\
 \hline
\end{tabular}
    \caption{Recovery thresholds (critical value of $\delta$): comparison of known results and this work. Note that exact recovery threshold for weighted projection graph coincides with the third column (partial recovery).}
    \label{tab:result-compare}
\end{table}

\subsection{Notation}
\label{subsec:notation}

For a hypergraph $H=(V,E)$ let the \textit{projection} of $H$ be the graph with vertex set $V$ and edge set equal to the set of $\{i,j\}$ such that $i,j\in V$, $i\not=j$, and there exists $h\in E$ such that $\{i,j\}\subset h$; we denote the projection of $H$ by $\proj(H)$. Let the \textit{weighted projection} of $H$ be the weighted graph with vertex set $V$ and edge weight of $\{i,j\}$ equal to the number of $h\in H$ that contain $\{i,j\}$ for all $i,j\in V$, $i\not=j$; we denote the weighted projection of $H$ by $\projw(H)$.

Suppose $d\geq 2$ and $H$ is a $d$-uniform hypergraph. Let $E(H)$ and $V(H)$ denote the sets of edges and vertices of $H$, respectively and let $e(H)=|E(H)|$ and $v(H)=|V(H)|$. Furthermore let $\alpha(H)=\frac{e(H)}{v(H)}$ and $m(H)=\max_{K\leq H} \alpha(H)$, where the maximum is over subgraphs $K$ of $H$.

Let $\mathcal{G}$ (resp. $\mathcal{G}_W$) be the set of projections (resp. weighted projections) of some $d$-uniform hypergraph. For $G\in\mathcal{G}$ (resp. $\mathcal{G}_W$), the hypergraph $H\in\{0,1\}^{\binom{[n]}{d}}$ is a \textit{preimage} of $G$ if $\proj(H)=G$ (resp. $\projw(H)=G$). 

\subsection{Random graph model and recovery criteria}

We now introduce the random hypergraph model from \cite{reconstruct} that we study. Suppose $\delta\in(-\infty, d-1)$ and $c\in (0, \infty)$. Unless otherwise stated assume that $p=(c+o_n(1))n^{-d+1+\delta}$. Furthermore, we assume that $\delta<1$ in all sections of the paper except for \Cref{sec:intro} and \Cref{subsec:dense}. Suppose $\mathcal{H}\subset\{0, 1\}^{\binom{[n]}{d}}$ the random variable such that each element of $\binom{[n]}{d}$ is a hyperedge of $\mathcal{H}$ independently with probability $p$. Furthermore we denote by $\mathcal{H}_c\in\{0,1\}^{\binom{[n]}{d}}$ the random $d$-uniform hypergraph with edge set equal to the set of $d$-cliques in $\proj(\mathcal{H})$. We observe $\proj(\mathcal{H})$ and use this projected graph to make statistical inferences about $\mathcal{H}$. We also sometimes use the random variable $\pg$ to denote $\proj(\mathcal{H})$, for brevity.

\begin{remark}
The regime of $p$ we consider differs by a constant factor from the regime $p=n^{-d+1+\delta}$ that \cite{reconstruct} considers because we hope to be more precise. Despite this, many of the results from the paper \cite{reconstruct} are still true and provable with the same methods, so we do not repeat the proofs after citing results from that paper.
\end{remark}

We sometimes refer to elements of $E(\mathcal{H}_c)\backslash E(\mathcal{H})$ as \textit{fake hyperedges}. Let $q\defeq\Pr[[d]\in E(\mathcal{H}_c)]$; note that $q$ is the probability that any element of $\binom{[n]}{d}$ is an edge of $\mathcal{H}_c$ by symmetry, so $q$ is the hyperedge density of $\mathcal{H}_c$.

\subsubsection{Partial recovery}
\label{subsubsec:partial}
Suppose $\mathcal{B}: \mathcal{G}\rightarrow \{0,1\}^{\binom{[n]}{d}}$ is a partial recovery algorithm for the unweighted projection. We define the loss of $\mathcal{B}$ to be
\[
\ell(\mathcal{B}) \defeq \frac{\E[|\mathcal{B}(\proj(\mathcal{H}))\Delta E(\mathcal{H})|]}{p\binom{n}{d}},
\]
where $\Delta$ denotes the symmetric difference. Note that $p\binom{n}{d} = \E_\mathcal{H}[|E(\mathcal{H})|]$, so $\ell$ corresponds to the fraction of edges in $H$ that are predicted incorrectly. Hence, the optimal unweighted partial recovery algorithm is $\mathcal{B}^*:\mathcal{G}\rightarrow\{0,1\}^{\binom{[n]}{d}}$ where 
\[
\mathcal{B}^*(G)=\left\{h\in\binom{[n]}{d}: \Pr[h\in\mathcal{H}|\proj(\mathcal{H})=G]\geq \frac{1}{2}\right\}.
\]

For a partial recovery algorithm $\mathcal{B}_W: \mathcal{G}_W \rightarrow \{0,1\}^{\binom{[n]}{d}}$ for the weighted projection, the loss of $\mathcal{B}_W$ is 
\[
\ell_W(\mathcal{B}_W) \defeq \frac{\mathbb{E}[|\mathcal{B}_W(\projw(\mathcal{H}))\Delta E(\mathcal{H})|]}{p\binom{n}{d}}.
\]
The optimal weighted partial recovery algorithm $\mathcal{B}_W^*$ is defined analogously to $\mathcal{B}^*$. 

\begin{definition}
\label{def:partialloss}
The \textit{partial recovery loss} is $\ell(\mathcal{B}^*)$. If $\ell(\mathcal{B}^*) = o_n(1)$, then \textit{almost exact recovery} is possible.

The \textit{weighted partial recovery loss} is $\ell_W(\mathcal{B}_W^*)$. If $\ell_W(\mathcal{B}_W^*)=o_n(1)$, then \textit{almost exact weighted recovery} is possible.
\end{definition}

Partial recovery is considered in other problems such as the HSBM~\cite{weakrecovery_hsbm} and almost exact recovery is also considered in graph matching~\cite{graphmatch} and the HSBM~\cite{optimalexact}.

\subsubsection{Exact recovery}
\label{subsubsec:exact}

Suppose $\mathcal{A}: \mathcal{G}\rightarrow\{0,1\}^{\binom{[n]}{d}}$ is an exact recovery algorithm for the unweighted projection. Its probability of error is $\Pr[\mathcal{A}(\proj(\mathcal{H}))\not=\mathcal{H}]$. As discussed in \cite[Section 2.3]{reconstruct}, the algorithm minimizing the probability of error is the MAP algorithm $\mathcal{A}^*:\mathcal{G}\rightarrow\{0,1\}^{\binom{[n]}{d}}$ where (assuming that $p< \frac{1}{2}$, which is true if $n$ is sufficiently large)
\[
\mathcal{A}^*(G)\in\argmax_{H:\,\proj(H)=G}\Pr[\mathcal{H}=H] = \argmin_{H:\,\proj(H)=G}e(H).
\]

Similarly, for an exact recovery algorithm $\mathcal{A}_W:\mathcal{G}_W\rightarrow\{0,1\}^{\binom{[n]}{d}}$ for the weighted projection, its probability of error is $\Pr[\mathcal{A}_W(\projw(\mathcal{H}))\not=\mathcal{H}]$. The algorithim minimizing the probability error is the MAP algorithm $\mathcal{A}^*_W: \mathcal{G}_W\rightarrow\{0,1\}^{\binom{[n]}{d}}$, which is analogous to $\mathcal{A}^*$. 

\begin{definition}
\label{def:probabilityexact}
The \textit{probability of exact recovery} is $\Pr[\mathcal{A}^*(\proj(\mathcal{H}))=\mathcal{H}]$. The \textit{weighted probability of exact recovery} is $\Pr[\mathcal{A}^*_W(\projw(\mathcal{H}))=\mathcal{H}]$.
\end{definition}

\textbf{Acknowledgments.} 
 This material is based upon work partially supported by the National Science
    Foundation under Grant No CCF-2131115, and by the MIT-IBM Watson AI Lab.

\section{Overview of main ideas}
\subsection{Weighted Versus Unweighted Projection}
We remark that recovery from the weighted projection is always easier than from the unweighted projection, for  both partial and exact recovery. This is because we can always obtain the vanilla projection from the weighted one by forgetting the weights. Therefore, all algorithms for unweighted projections also apply to the weighted case. In particular:

\begin{lemma}
\label{lemma:fundamental}
$\ell(\mathcal{B}^*) \geq \ell_W(\mathcal{B}_W^*)$ and $\Pr[\mathcal{A}^*(\proj(\mathcal{H}))=\mathcal{H}] \leq \Pr[\mathcal{A}^*_W(\proj_W(\mathcal{H}))=\mathcal{H}]$.
\end{lemma}

To introduce the main ideas, we will focus on unweighted projections for simplicity. The impossibility of partial recovery is stronger for weighted projections than for unweighted projections, but we defer the discussion of weighted projections to \Cref{sec:partial}, since the proof techniques are similar.

\subsection{Maximum Clique Cover Algorithm}
\label{subsec:maxcover}

Let us focus on partial recovery first.
We know that every hypergraph in $\rhg$ corresponds to a $d$-clique in $\pg$, where $\pg=\text{Proj}(\mathcal{H})$ by definition. A naive algorithm would be to output the hypergraph $\mathcal{H}_c$, which consists of all $d$-cliques in $\pg$. We call this algorithm \emph{Maximum Clique Cover} and denote it by $\Bc$. However, since only a subset of cliques correspond to hyperedges, $\Bc$ is not optimal; the loss of $\Bc$ is the following:
\begin{align*}
    \ell(\Bc) &= \frac{\E\b[\sum_{T\subset [n],\,|T|=d}\ind{\binom{T}{2}\subset \pg} - |\rhg|\b]}{p\binom{n}{d}} = \frac{\Pr\b(\binom{[d]}{2}\subset E(\pg)\b)}{p}-1.
\end{align*}

As a reminder, $q= \Pr\b(\binom{[d]}{2}\subset E(\pg)\b)$ is the hyperedge density of $\mathcal{H}_c$. For the proof of the following result, see \Cref{subsec:fakeedge}. The main idea is that a hyperedge of $\mathcal{H}_c$ is either a hyperedge of $\mathcal{H}$, which occurs with probability $p$, or each of its edges is contained in a hyperedge of $\mathcal{H}$. Then, we can approximate the probability of a hyperedge appearing in $\mathcal{H}_c$ with the sum of $p$ and the product of the probabilities of each of its edges being contained in a hyperedge of $\mathcal{H}$.

\begin{theorem}
\label{thm:fakeedge}
\[
q=(1+o_n(1))(p+\left(\frac{cn^{\delta-1}}{(d-2)!}\right)^{\binom{d}{2}}).
\]
\end{theorem}

\begin{corollary}
\label{corr:asyqp}
\[
q =
\begin{cases}
(1+o_n(1))p & \text{if $\delta<\frac{d-1}{d+1}$}, \\
p+\Theta_n(p) & \text{if $\delta=\frac{d-1}{d+1}$},  \\
\omega_n(p) &\text{if $\delta>\frac{d-1}{d+1}$}.
\end{cases}
\]
\end{corollary}

The following theorem tells us that $\mathcal{B}_C$ achieves a partial recovery loss of $o_n(1)$ whenever $\delta$ is below $\frac{d-1}{d+1}$, because $\mathcal{B}_C$ falsely recovers $\frac{q}{p}-1 = o_n(1)$ of the hyperedges of $\mathcal{H}$ on average. 

\begin{theorem}
\label{thm:epthreshold}
Almost exact recovery is possible if $\delta<\frac{d-1}{d+1}$.
\end{theorem}

This simple algorithm actually achieves the optimal threshold for almost exact recovery and partial recovery, as we will show in the next section.

\subsection{Relating partial recovery loss to preimage overlap}

We discuss some fundamental equalities for the partial recovery loss introduced in \Cref{subsubsec:partial}. First we have that for $G\in\mathcal{G}$,
\begin{equation}
\label{eq:optloss}
\E_{\mathcal{H}|\pg=G}\b[|\mathcal{B}^*(G)\Delta\mathcal{H}|\b] = \sum_{h\in\binom{[n]}{d}} \min\b\{\Pr[h\in\mathcal{H}|\pg=G], \Pr[h\notin\mathcal{H}|\pg=G]\b\}.
\end{equation}
Furthermore, $\ell(\mathcal{B}^*) \leq 1$ since the trivial algorithm that always outputs the empty set has a loss of $1$.

The partial recovery loss can be captured by the expected size of the intersection of two random hypergraphs with the same projection: Let $\rhg'$ be another random hypergraph that is independently and identically distributed as $\rhg$ after conditioning on the projected graph; that is, when $\proj(\mathcal{H}) = G$ for some $G\in\mathcal{G}$, $\mathcal{H}'$ is sampled independently of $\mathcal{H}$ from the distribution $p_{\mathcal{H}|\proj(\mathcal{H})=G}$. Because $\mathcal{H}$ and $\mathcal{H}'$ are independent after conditioning on $\proj(\mathcal{H})$, we have the Markov chain $\mathcal{H}\rightarrow \proj(\mathcal{H})=\proj(\mathcal{H}')\rightarrow \mathcal{H}'$.

\begin{lemma}
\label{lemma:exactpartialequiv} The partial recovery loss can be expressed in terms of the overlap:
\begin{enumerate}
\item[(a)] $\ell(\mathcal{B}^*) \geq 1-\frac{\E_{\rhg,\rhg'} [E(\rhg)\cap E(\rhg')]}{p\binom{n}{d}} \geq \ell(\mathcal{B}^*)/2$. If
$
\E_{\rhg,\rhg'}[E(\rhg)\cap E(\rhg')] = o_n\left(p\binom{n}{d}\right)$
then the partial recovery loss is $1-o_n(1)$.
\item[(b)] $\E_{\rhg,\rhg'} [E(\rhg)\cap E(\rhg')]=\binom{n}{d}\sum_{G\in\mathcal{G}} \Pr[\pg=G]\Pr[[d]\in \mathcal{H}|\pg=G]^2$.
\end{enumerate}
\end{lemma}

\begin{proof}
The proof of (b) is straightforward. For (a), first observe that using \pref{eq:optloss} gives that
\begin{equation}
\label{eq:loss_equiv}
\begin{split}
p\binom{n}{d}\ell(\mathcal{B}^*) & = \E_{G\sim \pg}[\E_{\mathcal{H}|\pg=G}[|\mathcal{B}^*(G)\Delta\mathcal{H}|]] \\
& = \sum_{G\in\mathcal{G}} \sum_{h\in\binom{[n]}{d}} \Pr[\pg = G]\min(\Pr[h\in\mathcal{H}|\pg=G], \Pr[h\notin\mathcal{H}|\pg=G]) \\
& = \sum_{h\in\binom{[n]}{d}}\sum_{G\in\mathcal{G}} \Pr[\pg = G]\min(\Pr[h\in\mathcal{H}|\pg=G], \Pr[h\notin\mathcal{H}|\pg=G]) \\ 
& = \binom{n}{d} \sum_{G\in\mathcal{G}} \Pr[\pg = G]\min(\Pr[[d]\in\mathcal{H}|\pg=G], \Pr[[d]\notin\mathcal{H}|\pg=G]) \\
& =  \binom{n}{d} \sum_{G\in\mathcal{G}} \Pr[\pg = G]\left(\frac{1}{2} - \Big|\Pr[[d]\in\mathcal{H}|\pg=G] - \frac{1}{2}\Big|\right).
\end{split}
\end{equation}
Using this then gives that
\begin{align*}
& 1\leq \frac{\sum_{G\in\mathcal{G}} \Pr[\pg = G]\left(\frac{1}{2} - \left|\Pr[[d]\in\mathcal{H}|\pg=G] - \frac{1}{2}\right|\right)}{\sum_{G\in\mathcal{G}} \Pr[\pg = G]\left(\frac{1}{4} - \left(\Pr[[d]\in\mathcal{H}|\pg=G] - \frac{1}{2}\right)^2\right)} \leq 2 \\
\Leftrightarrow & 1\leq \frac{\sum_{G\in\mathcal{G}} \Pr[\pg = G]\left(\frac{1}{2} - \left|\Pr[[d]\in\mathcal{H}|\pg=G] - \frac{1}{2}\right|\right)}{p - \sum_{G\in\mathcal{G}} \Pr[\pg=G]\Pr[[d]\in\mathcal{H}|\pg=G]^2} \leq 2 \\
\Leftrightarrow \,& \ell(\mathcal{B}^*) \ge  
1-\frac{\E_{\rhg,\rhg'} [E(\rhg)\cap E(\rhg')]}{p\binom{n}{d}} \ge \ell(\mathcal{B}^*)/2.
\end{align*}
The remaining result of (a) is straightforward to prove.
\end{proof}

\subsection{Information-theoretic argument proving weaker impossibility of partial recovery}

To gain some intuition on why partial recovery is impossible, we can look at the question from an information-theoretic view.
Since we are recovering $\mathcal{H}$ after observing $\pg$, if $H(\pg)$ is significantly smaller than $H(\mathcal{H})$, we will not be able to recover much information about $\mathcal{H}$. This is the case when $\delta>\frac{d-1}{d+1}$ because $\mathcal{H}$ is sufficiently dense, see \Cref{lemma:condentropy}. A weaker version of the impossibility result in \Cref{thm:mainpartial} can thus be obtained.

\begin{lemma}
\label{lemma:weakpartial}
If $\delta>\frac{d-1}{d+1}$, then the partial recovery loss is $\Omega_n(1)$.
\end{lemma}
\begin{proof}
From \Cref{lemma:condentropy}, if $\delta>\frac{d-1}{d+1}$ then $H(\pg)=(1-\Omega_n(1))H(\mathcal{H})$. Afterwards, $H(\mathcal{H}|\pg)=H(\mathcal{H}, \pg) - H(\pg) = H(\mathcal{H})-H(\pg)$, so $H(\mathcal{H}|\pg)=\Omega_n(H(\mathcal{H}))$. 

Let $P_e$ be the probability that $\mathcal{B}^*$ predicts $\1\{[d]\in\mathcal{H}\}$ incorrectly given $\pg$. Then, $\ell(\mathcal{B}^*) = \frac{P_e}{p}$. Applying Fano's lemma gives that
\begin{align*}
& \binom{n}{d}H_B(P_e) \geq \binom{n}{d}H(\1\{[d]\in \mathcal{H}\}|\pg) = \sum_{h\in \binom{[n]}{d}} H(\1\{h\in\mathcal{H}\}|\pg) \\
& \geq H(\mathcal{H}|\pg) = \Omega_n(H(\mathcal{H})) = \binom{n}{d}\Omega_n(H_B(p))
\Rightarrow P_e = \Omega_n(p) \Rightarrow \ell(\mathcal{B}^*) = \frac{P_e}{p} = \Omega_n(1).
\end{align*}
\end{proof}

However, we cannot extend this argument to prove that $\ell(\mathcal{B}^*)=1-o_n(1)$ if $\delta>\frac{d-1}{d+1}$. For this to be the case, we must show that $P_e=(1-o_n(1))p$ since $\ell(\mathcal{B}^*) = \frac{P_e}{p}$; then, using the previous chain of inequalities, we require that $H(\mathcal{H}|\pg) = (1-o_n(1))H(\mathcal{H})$. We will see in \Cref{thm:entropycompare} or \Cref{lemma:entropyorder} that $H(\pg) = \Omega_n(H(\mathcal{H}))$, i.e.  $H(\mathcal{H}|\pg) = (1-\Omega_n(1))H(\mathcal{H})$. So we cannot extend the argument. In the next section, we explain how we circumvent this problem.

\subsection{Impossibility of partial recovery}
\label{subsec:impossibility}

The goal of this subsection is to prove that partial recovery is impossible when $\delta>\frac{d-1}{d+1}$, which is stated in the following result. 

\begin{theorem}
\label{thm:apimpossible} Suppose $\delta>\frac{d-1}{d+1}$. Then the partial recovery loss is $1-o_n(1)$.
\end{theorem}

To prove the theorem, we use \Cref{lemma:exactpartialequiv}, which means that we need to prove that $\E[|E(\mathcal{H})\cap E(\mathcal{H}')|] = o_n\left(p\binom{n}{d}\right)$. One of the main ideas is that we can restrict the set of $\mathcal{H}$ and $\mathcal{H}'$ we compute the expectation over by using expressions such as $\E[|E(\mathcal{H})\cap E(\mathcal{H}')|\1\{\mathcal{H}\in \mathcal{A}\}\1\{\mathcal{H}'\in\mathcal{A}'\}]$ for sets $\mathcal{A}$ and $\mathcal{A}'$ that contain $\mathcal{H}$ and $\mathcal{H}'$, respectively, with high probability. We use the following lemma for this approach.

\begin{lemma}
\label{lemma:highprob}
Suppose $\mathcal{U}_n\subset\{0,1\}^{\binom{[n]}{d}}$, $n\geq 1$ satisfy $\Pr[\mathcal{H}\in\mathcal{U}_n]=o_n(1)$. Then
\[
\E_{\rhg,\rhg'}[|E(\rhg)\cap E(\rhg')|\ind{\mathcal{H}\in\mathcal{U}_n}] = o_n\left(p\binom{n}{d}\right).
\]
\end{lemma}

\begin{proof}
The Cauchy-Schwarz inequality gives that
\begin{align*}
& \sum_{\substack{H \in\mathcal{U}_n,\,H'\in\{0,1\}^{\binom{[n]}{d}}, \\ \proj(H)=\proj(H')}} \frac{\Pr[\mathcal{H}=H]\Pr[\mathcal{H}=H']}{\Pr[\proj(\mathcal{H})=\proj(H)]} |E(H)\cap E(H')|\leq\sum_{H\in\mathcal{U}_n}\Pr[\mathcal{H}=H] e(H) \\
& \leq\left(\sum_{H\in\mathcal{U}_n} \Pr[\mathcal{H}=H]\right)^{\frac{1}{2}}\left(\sum_{H\in\{0,1\}^{\binom{[n]}{d}}} \Pr[\mathcal{H}=H]e(H)^2\right)^{\frac{1}{2}} = o_n\left(p\binom{n}{d}\right).
\end{align*}
\end{proof}

First observe that $e(\mathcal{H})$ is concentrated around its mean $p\binom{n}{d}$ and \Cref{corr:varnumclique} gives that $e(\mathcal{H}_c)$ is concentrated around its mean $q\binom{n}{d}$. Suppose $\epsilon=o_n(1)$ satisfies $e(\mathcal{H})\in [(1-\epsilon)p\binom{n}{d}, (1+\epsilon)p\binom{n}{d}]$ and $e(\mathcal{H}_c)\in [(1-\epsilon) q\binom{n}{d}, (1+\epsilon)q\binom{n}{d}]$ with probability $1-o_n(1)$. Let $\mathcal{Z}$ be the set of $H\in\{0,1\}^{\binom{[n]}{d}}$ such that $e(H)\in [(1-\epsilon)p\binom{n}{d}, (1+\epsilon)p\binom{n}{d}]$ and $e(\mathcal{H}_c)\in[(1-\epsilon) q\binom{n}{d}, (1+\epsilon)q\binom{n}{d}]$. From \Cref{lemma:highprob}, 
\begin{equation}
\label{eq:intersectequiv}
\E[|E(\mathcal{H})\cap E(\mathcal{H}')|] = \E[|E(\mathcal{H})\cap E(\mathcal{H}')|\1\{\mathcal{H},\mathcal{H}'\in\mathcal{Z}\}] + o_n\left(p\binom{n}{d}\right),
\end{equation}
because $\Pr[\mathcal{H}\notin\mathcal{Z}] = \Pr[\mathcal{H}'\notin\mathcal{Z}]=o_n(1)$. This means we can focus only on the case where both $e(\mathcal{H})$ and $e(\mathcal{H}_c)$ are concentrated.

The concentration of $e(\mathcal{H})$ and $e(\mathcal{H}_c)$ is important for the proof of the key result \Cref{thm:highintersect}. Essentially, it allows us to count instances of $\mathcal{H}$ and $\mathcal{H}'$ with a particular set of overlapping hyperedges. See \Cref{subsec:highintersectproof} for the proof of the result, which we state next.

\begin{theorem}
\label{thm:highintersect}
Suppose $\delta>\frac{d-1}{d+1}$ and $M\in (0, 1)$. Then
\[
\Pr\left(\rhg,\rhg'\in \mathcal{Z}, |E(\rhg)\cap E(\rhg')|\ge Mp\binom{n}{d}\right) = o_n(1).
\]
\end{theorem}

\begin{proof}[Proof of {\Cref{thm:apimpossible}}]
From \Cref{lemma:exactpartialequiv} and \pref{eq:intersectequiv}, it suffices to prove that 
\[
\E[|E(\mathcal{H})\cap E(\mathcal{H}')|\1\{\mathcal{H},\mathcal{H}'\in\mathcal{Z}\}] = o_n\left(p\binom{n}{d}\right).
\]
Observe that
\begin{align*}
& \E[|E(\mathcal{H})\cap E(\mathcal{H}')|\1\{\mathcal{H},\mathcal{H}'\in\mathcal{Z}\}] \\
&= \E\left[|E(\mathcal{H})\cap E(\mathcal{H}')|\1\{\mathcal{H},\mathcal{H}'\in\mathcal{Z},\,|E(\mathcal{H})\cap E(\mathcal{H}')| < Mp\binom{n}{d}\}\right] \\
& + \E\left[|E(\mathcal{H})\cap E(\mathcal{H}')|\1\{\mathcal{H},\mathcal{H}'\in\mathcal{Z},\,|E(\mathcal{H})\cap E(\mathcal{H}')| \geq Mp\binom{n}{d}\}\right] \\
& < Mp\binom{n}{d} + \Pr\left(\rhg,\rhg'\in \mathcal{Z}, |E(\rhg)\cap E(\rhg')|\ge Mp\binom{n}{d}\right)(1+\epsilon)p\binom{n}{d},
\end{align*}
since $\mathcal{H}\in\mathcal{Z}\Rightarrow e(\mathcal{H}) \leq (1+\epsilon)p\binom{n}{d}$. Using \Cref{thm:highintersect} and $\epsilon=o_n(1)$ then gives that 
\[
\E[|E(\mathcal{H})\cap E(\mathcal{H}')|\1\{\mathcal{H},\mathcal{H}'\in\mathcal{Z}\}] \leq (M+o_n(1))p\binom{n}{d}.
\]
Considering the limit as $M\rightarrow 0$ completes the proof.
\end{proof}

\subsection{Ambiguous graphs and exact recovery}
\label{subsec:ambiguous}

We now find the threshold for exact recovery. First, since the impossibility of partial recovery implies the impossibility of exact recovery, based on \Cref{thm:apimpossible}, exact recovery is impossible if $\delta>\frac{d-1}{d+1}$. Next, we explain how the exact recovery threshold is determined by the threshold for the existence of ambiguous graphs and the notion of two-connected components, which are both introduced in the paper \cite{reconstruct}. We also define an important ambiguous graph which the paper discovers, see \Cref{def:ambiguous}.

\begin{definition}
For $G\in\mathcal{G}$, we refer to a preimage $H$ of $G$ that minimizes $e(H)$ as a \textit{minimal preimage} of $G$. The graph $G$ is \textit{ambiguous} if it has at least two distinct minimal preimages. Also, the graph $G$ is \textit{weighted-ambiguous} if it has at least two distinct minimal preimages with respect to the weighted projection.
\end{definition}

Now we explain why the presence of ambiguous graphs implies the impossibility of exact recovery. Suppose that $H\in \{0,1\}^{\binom{[n]}{d}}$ and $G=\proj(H)$ is ambiguous. Then, there is a preimage $H'$ of $G$ different from $H$ with at most the same number of hyperedges. We have that $\Pr[\mathcal{H}=H] = p^{e(H)}(1-p)^{\binom{n}{d}-e(H)} \leq p^{e(H')}(1-p)^{\binom{n}{d}-e(H')} = \Pr[\mathcal{H}=H']$, since $p=o_n(1)<\frac{1}{2}$ when $n$ is sufficiently large; then, $\Pr[\mathcal{H}=H|\proj(\mathcal{H}) = G] \leq \Pr[\mathcal{H}=H'|\proj(\mathcal{H}) = G]$. So, whenever $\proj(\mathcal{H}) = G$, we will not be able to recover $\mathcal{H}$ with probability at least $\frac{1}{2}$. As discussed in \Cref{lemma:ambiguous}, the converse is also true: the absence of ambiguous graphs implies exact recovery. We will show the following ambiguous graph is the key bottleneck in exact recovery. 

\begin{definition}
\label{def:ambiguous}
Suppose $H=(V, E)$ is a $d$-uniform hypergraph such that:
\begin{itemize}
\item $V= \{u, w\}\sqcup \{v_i: 1\leq i\leq d-1\} \bigsqcup_{1\leq i\leq d-1} S_i^u\bigsqcup_{1\leq i\leq d-1} S_i^w$, where $|S_i^u|=|S_i^w|=d-2$ for $1\leq i\leq d-1$ (note the usage of disjoint unions).
\item $E$ consists of $\{u\}\cup \{v_i: 1\leq i\leq d-1\}$, $\{u, v_i\}\cup S_i^u$ for $1\leq i\leq d-1$, and $\{w, v_i\}\cup S_i^w$ for $1\leq i\leq d-1$.
\end{itemize}
Then, define $G_{a,d}:=\proj(H)$.
\end{definition}

See \Cref{fig:Ga3} for an illustration when $d=3.$ Observe that $G_{a,d}$ is an ambiguous graph, see \cite[Lemma 27]{reconstruct}, which explains that if $H'$ is $H$ with the hyperedge $\{u\}\cup\{v_i: 1\leq i\leq d-1\}$ replaced with $\{w\}\cup\{v_i: 1\leq i\leq d-1\}$, then $H$ and $H'$ are two distinct minimal preimages of $G_{a,d}$.

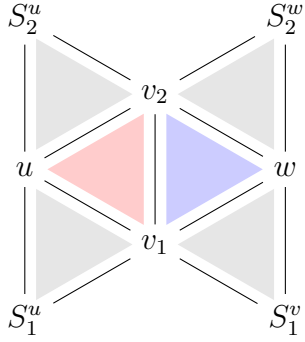
\begin{figure}[!ht]
\begin{center}
\begin{tikzpicture}[x=1cm, y=1cm]
\node (A) at (0,0) {$u$};
\node (B) at (3^0.5, -1) {$v_1$};
\node (C) at (2*3^0.5, 0) {$w$};
\node (D) at  (3^0.5, 1) {$v_2$};
\node (E) at (0, 2) {$S_2^u$};
\node (F) at (0, -2) {$S_1^u$};
\node (G) at (2*3^0.5, -2) {$S_1^v$};
\node (H) at (2*3^0.5, 2) {$S_2^w$};
\draw (A) -> (B);
\draw (A) -> (D);
\draw (A) -> (E);
\draw (A) -> (F);
\draw (B) -> (C);
\draw (B) -> (D); 
\draw (B) -> (F);
\draw (B) -> (G);
\draw (C) -> (D);
\draw (C) -> (H);
\draw (C) -> (G);
\draw (D) -> (E);
\draw (D) -> (H);
\fill[fill=red!20] ([shift=({0.3,0})] A.center)--([shift=({-0.15, 0.15*3^0.5})] B.center)--([shift=({-0.15, -0.15*3^0.5})] D.center) -- cycle;
\fill[fill=blue!20] ([shift=({-0.3,0})] C.center)--([shift=({0.15, 0.15*3^0.5})] B.center)--([shift=({0.15, -0.15*3^0.5})] D.center) -- cycle;
\fill[fill = gray!20] ([shift=({0.15, 0.15*3^0.5})] A.center)--([shift=({-0.3, 0})] D.center)--([shift=({0.15, -0.15*3^0.5})] E.center) -- cycle;
\fill[fill = gray!20] ([shift=({0.15, -0.15*3^0.5})] A.center)--([shift=({-0.3, 0})] B.center)--([shift=({0.15, 0.15*3^0.5})] F.center) -- cycle;
\fill[fill = gray!20] ([shift=({-0.15, 0.15*3^0.5})] C.center)--([shift=({0.3, 0})] D.center)--([shift=({-0.15, -0.15*3^0.5})] H.center) -- cycle;
\fill[fill = gray!20] ([shift=({-0.15, -0.15*3^0.5})] C.center)--([shift=({0.3, 0})] B.center)--([shift=({-0.15, 0.15*3^0.5})] G.center) -- cycle;
\end{tikzpicture}\vspace{-7mm}
\end{center}
\caption{\small This figure displays $G_{a,3}$. One minimal preimage consists of the gray and red triangles while the other consists of the gray and blue triangles. For $d>3$ the red and blue triangles are replaced with cliques of size $d$, the sets $S_i^u$, $S_i^w$ change to $(d-2)$-cliques and their number changes from 2 to $d-1$.}\label{fig:Ga3}
\end{figure}

We state the key lemma for exact recovery after introducing the notion of \textit{two-connected components}.

\begin{definition}
Suppose $H$ is a $d$-uniform hypergraph. The \textit{two-connected components} of $H$ are the connected components of the graph with vertex set $E(H)$ and edge set $\{\{h_1, h_2\}: h_1, h_2\in E(H),\, |h_1\cap h_2|\geq 2\}$.
\end{definition}

\begin{lemma}[{\cite[Lemma 19]{reconstruct}}]
\label{lemma:ambiguous}
Suppose $\delta<\frac{d-1}{d+1}$. If for all finite ambiguous graphs $G_a$,
\[
\Pr[\Cli(G_a) \text{ is a two-connected component of $\mathcal{H}_c$}] = o_n(1),
\]
then $\Pr[\mathcal{A}^*(\pg)=\mathcal{H}] = 1 - o_n(1)$.
\end{lemma}

In order to analyze the probability of an ambiguous graph being a two-connected component, we analyze the probabilities of the graph's minimal preimages. The following lemma is useful for this purpose; $m(\cdot)$ is defined in the notation section.

\begin{lemma}[{\cite[Lemma 17]{reconstruct}}]
\label{lemma:threshold}
Suppose $K$ is a fixed $d$-uniform hypergraph. Then,
\[
\Pr[K\subset \mathcal{H}] = \begin{cases}
o_n(1) & \text{if } \delta < d-1-\frac{1}{m(K)},
\\ 1-o_n(1) & \text{if } \delta > d-1-\frac{1}{m(K)}, 
\\ \Omega_n(1), & \text{if } \delta = d-1-\frac{1}{m(K)},
\end{cases}
\]
where $\Pr[K\subset\mathcal{H}]$ is the probability that $K$ appears as an induced subgraph of $\mathcal{H}$.
\end{lemma}

Observe that for the minimal preimage $H$ of $G_{a,d}$, we have that $d-1-\frac{1}{m(H)}=\frac{2d-4}{2d-1}$. So, $\frac{2d-4}{2d-1}$ is the threshold at which $G_{a,d}$ appears as an induced subgraph. In fact, we will prove that $G_{a,d}$ minimizes this threshold over ambiguous graphs when $d=3,4$, which is included in the following result.

\begin{theorem}
\label{thm:ambiguous_threshold}
Suppose $d\geq 3$ and $h$ is a preimage of an ambiguous graph. Then, $d-1-\frac{1}{m(h)}\geq \frac{2d-4}{2d-1}$.
\end{theorem}

\begin{proof}
The $d=3$ case is resolved in \cite[Appendix D]{reconstruct}. The result is proved in a slightly more general setting for $d=4$ and $d\geq 5$ in the sense that the projected graph does not need to be ambiguous in \Cref{thm:min4} and \Cref{thm:general}, respectively.
\end{proof}

Essentially, this result and the partial recovery threshold tells us that exact recovery is possible if and only if $\delta<\min(\frac{d-1}{d+1}, \frac{2d-4}{2d-1})$; this is exactly the contents of \Cref{thm:mainexact}. For more details of the proof as well as the argument for weighted exact recovery, see \Cref{sec:exactproof}.

\section{Future directions}
In this paper we essentially fully close the exact and partial recovery questions for projections of hypergraphs. There are a number of interesting directions for future exploration. A few of these are:
\begin{enumerate}
    \item A potential future direction is to study the same recovery questions but for non-uniform hypergraphs, i.e., when the hyperedges have differing sizes. A natural model is to include a hyperedge on each given subset $S\subset [n]$ independently with probability $p_{|S|}$ depending only on the size of $S$. For each $d\geq 2$, we might take $p_d \propto (1+o_n(1))n^{-d+1+\delta_d}$, where $\delta_d\in [-\infty, d-1)$. 
    \item A specific instance of the above non-uniform hyperedge problem arises in the context of \textit{random intersection graphs} (RIG). Each node $i\in [n]$ is associated with an independent random subsets $S_i$ of $[m]$ such that $j\in [m]$ is an element of $S_i$ independently with probability $p$. The RIG includes edge $\{u,v\}$ if and only if $S_u$ and $S_v$ are not disjoint. For $j\in [m]$, define $C_j$ to be the set of $i\in [n]$ such that $j\in S_i$, and note that the graph has a clique on each $C_j$. When can one recover the sets $C_j$, given the graph? \cite{randomintersect} finds a phase transition for the total variation convergence of an RIG to Erd\"os-R\'enyi. We expect a phase transition to occur for the feasibility of recovering the cliques $C_j$.

    \item One might aim to recover a hypergraph from a \emph{noisy} unweighted projection; for example, we could flip each of the edges of $\text{Proj}(\mathcal{H})$ with some fixed probability. We expect that noise will lower the partial recovery threshold for $\delta$. What is the new recovery threshold?
\end{enumerate}

\bibliography{references.bib}

\begin{appendix}

\section{Partial recovery results}
\label{sec:partial}

\subsection{Proof of \texorpdfstring{\Cref{thm:highintersect}}{}}

\label{subsec:highintersectproof}

As a reminder, $q=\Pr[\binom{[d]}{2} \subset \text{Proj}(\mathcal{H})]$ is the hyperedge density of $\mathcal{H}_c$.
\begin{lemma}
\label{lemma:squareprob}
Suppose $\delta>\frac{d-1}{d+1}$. Let $\mathcal{G}' = \text{Proj}(\mathcal{Z})$. Then
\[
\sum_{\substack{H, H'\in\mathcal{Z}, \\ \text{Proj}(H)=\text{Proj}(H')}} \frac{\Pr[\mathcal{H}=H]\Pr[\mathcal{H}=H']}{\Pr[\text{Proj}(\mathcal{H})=\text{Proj}(H)]^2}\leq |\mathcal{G}'| \leq \left(\frac{(1+o_n(1))e^{\binom{d}{2}}}{q}\right)^{(1+\epsilon)p\binom{n}{d}}.
\]
\end{lemma}
\begin{proof}
We have that
\[
\sum_{\substack{H, H'\in\mathcal{Z}, \\ \text{Proj}(H)=\text{Proj}(H')}} \frac{\Pr[\mathcal{H}=H]\Pr[\mathcal{H}=H']}{\Pr[\text{Proj}(\mathcal{H})=\text{Proj}(H)]^2} \leq \sum_{H\in\mathcal{Z}} \frac{\Pr[\mathcal{H}=H]}{\Pr[\text{Proj}(\mathcal{H})=\text{Proj}(H)]} \leq |\mathcal{G}'|.
\]
The maximum number of edges in some $G\in\mathcal{G}'$ is $(1+\epsilon)\binom{d}{2}p\binom{n}{d}$. Therefore,
\[
|\mathcal{G}'| \leq (1+\epsilon)\binom{d}{2}p\binom{n}{d}\binom{\binom{n}{2}}{(1+\epsilon)\binom{d}{2}p\binom{n}{d}},
\]
if $n$ is sufficiently large so that $(1+\epsilon)\binom{d}{2}p\binom{n}{d}< \frac{1}{2}\binom{n}{2}$. Stirling's approximation gives that $k!\geq\left(\frac{k}{e}\right)^k$. 

When $\delta>\frac{d-1}{d+1}$, the fraction of the hyperedges of $\mathcal{H}$ in the hyperedges of $\mathcal{H}_c$ is $o_n(1)$, so we have that $q=(1+o_n(1))\left(\frac{cn^{\delta-1}}{(d-2)!}\right)^{\binom{d}{2}} = (1+o_n(1))\left(\frac{\binom{d}{2}p\binom{n}{d}}{\binom{n}{2}}\right)^{\binom{d}{2}}$, see \Cref{thm:fakeedge}. This gives that
\[
\binom{\binom{n}{2}}{(1+\epsilon)\binom{d}{2}p\binom{n}{d}} \leq \left(\frac{e\binom{n}{2}}{(1+\epsilon)\binom{d}{2}p\binom{n}{d}}\right)^{(1+\epsilon)\binom{d}{2}p\binom{n}{d}} = \left(\frac{(1+o_n(1))e^{\binom{d}{2}}}{q}\right)^{(1+\epsilon)p\binom{n}{d}}.
\]
\end{proof}

\begin{proof}[Proof of \Cref{thm:highintersect}]
If $H'\in\mathcal{Z}$ then
\[
\Pr[\mathcal{H}=H'] \leq \left(\frac{p}{1-p}\right)^{(1-\epsilon)p\binom{n}{d}}(1-p)^{\binom{n}{d}}.
\]
Suppose $I\in [Mp\binom{n}{d}, (1+\epsilon)p\binom{n}{d}]$ is an integer. Next we upper bound the number of choices for $H'$ given $H$ and $|E(H)\cap E(H')|=I$. 

Suppose $H\in\mathcal{Z}$. The number of choices for $E(H)\cap E(H')$ is at most
\[
\binom{\lfloor(1+\epsilon)p\binom{n}{d}\rfloor}{I}.
\]
Since $E(H')$ is a subset of the $d$-cliques of $H$, the number of choices for $E(H')\backslash E(H)$ is at most
\[
\binom{\lfloor(1+\epsilon)q\binom{n}{d}\rfloor}{\lfloor (1+\epsilon)p\binom{n}{d}\rfloor - I}
\]
assuming that $n$ is sufficiently large so that $p\leq \frac{q}{2}$. Hence
\begin{align*}
& \sum_{\substack{H, H'\in\mathcal{Z}, \\ \text{Proj}(H)=\text{Proj}(H'), \\
|E(H)\cap E(H')|=I}} \Pr[\mathcal{H}=H]\Pr[\mathcal{H}=H'] \\ 
& \leq \sum_{H\in\mathcal{Z}} \Pr[\mathcal{H}=H]\left(\frac{p}{1-p}\right)^{(1-\epsilon)p\binom{n}{d}}(1-p)^{\binom{n}{d}}\binom{\lfloor(1+\epsilon)p\binom{n}{d}\rfloor}{I}\binom{\lfloor(1+\epsilon)q\binom{n}{d}\rfloor}{\lfloor (1+\epsilon)p\binom{n}{d}\rfloor - I} \\
& \leq \left(\frac{p}{1-p}\right)^{(1-\epsilon)p\binom{n}{d}}e^{-p\binom{n}{d}}\left(\frac{e(1+\epsilon)}{I/(p\binom{n}{d})}\right)^I\left(\frac{e(1+\epsilon)\frac{q}{p}}{1+\epsilon-I/(p\binom{n}{d})}\right)^{(1+\epsilon - I/(p\binom{n}{d}))p\binom{n}{d}} \\
& \leq (1+o_n(1)) p^{(I/(p\binom{n}{d})-2\epsilon)p\binom{n}{d}}q^{(1+\epsilon-I/(p\binom{n}{d}))p\binom{n}{d}} \\
& \quad \cdot \left(\frac{1+o_n(1)}{(I/(p\binom{n}{d}))^{I/(p\binom{n}{d})}(1+\epsilon - I/(p\binom{n}{d}))^{1 + \epsilon -I/(p\binom{n}{d})}}\right)^{p\binom{n}{d}}.
\end{align*}
Observe that the inequalities are true even for the edge-case $I=\lfloor (1+\epsilon)p\binom{n}{d}\rfloor$. Let \\$\alpha=\max_{m\in [M, 1)} \frac{2}{m^m(1-m)^{1-m}}$. Then if $n$ is sufficiently large,
\begin{equation}
\label{eq:sumprobpair}
\sum_{\substack{H, H'\in\mathcal{Z}, \\ \text{Proj}(H)=\text{Proj}(H'), \\
|E(H)\cap E(H')|=I}} \Pr[\mathcal{H}=H]\Pr[\mathcal{H}=H'] \leq (1+o_n(1))p^{(I/(p\binom{n}{d})-2\epsilon)p\binom{n}{d}}q^{(1+\epsilon-I/(p\binom{n}{d}))p\binom{n}{d}} \alpha^{p\binom{n}{d}}.
\end{equation}

Using Cauchy-Schwarz gives that
\begin{align*}
& \sum_{\substack{H, H'\in\mathcal{Z}, \\ \text{Proj}(H)=\text{Proj}(H'), \\
Mp\binom{n}{d} \leq |E(H)\cap E(H')|\leq(1+\epsilon)p\binom{n}{d}}} \frac{\Pr[\mathcal{H}=H]\Pr[\mathcal{H}=H']}{\Pr[\text{Proj}(\mathcal{H})=H]} \\
& \leq \left(\sum_{Mp\binom{n}{d}\leq I \leq (1+\epsilon)p\binom{n}{d} }\sum_{\substack{H, H'\in\mathcal{Z}, \\ \text{Proj}(H)=\text{Proj}(H'), \\
|E(H)\cap E(H')|=I}} \Pr[\mathcal{H}=H]\Pr[\mathcal{H}=H']\right)^{\frac{1}{2}} \\
& \quad \times \left(\sum_{\substack{H, H'\in\mathcal{Z}, \\ \text{Proj}(H)=\text{Proj}(H')}} \frac{\Pr[\mathcal{H}=H]\Pr[\mathcal{H}=H']}{\Pr[\text{Proj}(\mathcal{H})=\text{Proj}(H)]^2}\right)^{\frac{1}{2}}
\end{align*}
Afterwards, using \pref{eq:sumprobpair} and \Cref{lemma:squareprob} gives that
\begin{align*}
&\sum_{\substack{H, H'\in\mathcal{Z}, \\ \text{Proj}(H)=\text{Proj}(H'), \\
Mp\binom{n}{d} \leq |E(H)\cap E(H')|\leq(1+\epsilon)p\binom{n}{d}}} \frac{\Pr[\mathcal{H}=H]\Pr[\mathcal{H}=H']}{\Pr[\text{Proj}(\mathcal{H})=H]} \\ & \leq \Bigg(\sum_{Mp\binom{n}{d}\leq I \leq (1+\epsilon)p\binom{n}{d}}(1+o_n(1))p^{(I/(p\binom{n}{d})-2\epsilon)p\binom{n}{d}}q^{(1+\epsilon-I/(p\binom{n}{d}))p\binom{n}{d}} \alpha^{p\binom{n}{d}} \\ & \quad\cdot \left(\frac{(1+o_n(1))e^{\binom{d}{2}}}{q}\right)^{(1+\epsilon)p\binom{n}{d}}\Bigg)^{\frac{1}{2}} \\
& = \Bigg(\sum_{Mp\binom{n}{d}\leq I \leq (1+\epsilon)p\binom{n}{d}}\left((1+o_n(1))\left(\frac{p}{q}\right)^{I/(p\binom{n}{d})} p^{-2\epsilon} \alpha e^{\binom{d}{2}}\right)^{p\binom{n}{d}}\Bigg)^{\frac{1}{2}}.
\end{align*}
Observe that using \Cref{thm:fakeedge} gives that
\[
\frac{q}{p} \geq (1+o_n(1))\frac{c^{\binom{d}{2}-1}}{(d-2)!^{\binom{d}{2}}}n^{\binom{d}{2}(\delta-1)+d-1-\delta}.
\]
Because $\delta>\frac{d-1}{d+1}$, $\binom{d}{2}(\delta-1)+d-1-\delta>0$. Suppose $Mp\binom{n}{d}\leq I \leq (1+\epsilon)p\binom{n}{d}$. Then $I/(p\binom{n}{d})\geq M$ so
\[
(1+o_n(1))\left(\frac{p}{q}\right)^{I/(p\binom{n}{d})} p^{-2\epsilon} \alpha e^{\binom{d}{2}} = O_n(n^{-M\left(\binom{d}{2}(\delta-1)+d-1-\delta\right) + 2\epsilon(d-1-\delta)}).
\]
Particularly, if $n$ is sufficiently large then
\[
(1+o_n(1))\left(\frac{p}{q}\right)^{I/(p\binom{n}{d})} p^{-2\epsilon} \alpha e^{\binom{d}{2}} = O_n(n^{-\frac{M}{2}\left(\binom{d}{2}(\delta-1)+d-1-\delta\right)})
\]
since $\epsilon=o_n(1)$. We therefore have that 
\begin{align*}
& \sum_{\substack{H, H'\in\mathcal{Z}, \\ \text{Proj}(H)=\text{Proj}(H'), \\
Mp\binom{n}{d} \leq |E(H)\cap E(H')|\leq(1+\epsilon)p\binom{n}{d}}} \frac{\Pr[\mathcal{H}=H]\Pr[\mathcal{H}=H']}{\Pr[\text{Proj}(\mathcal{H})=H]} \\
& \leq \left((1+\epsilon)p\binom{n}{d}O_n(n^{-\frac{M}{2}\left(\binom{d}{2}(\delta-1)+d-1-\delta\right)})^{p\binom{n}{d}}\right)^{\frac{1}{2}} = o_n(1).
\end{align*}
\end{proof}

\subsection{Weighted partial recovery}
\label{subsec:weightedpartial}

For weighted partial recovery, we can prove a similar result as \Cref{thm:apimpossible} using a similar method.

\begin{theorem}
\label{thm:weightedapimpossible}
Suppose $\delta>\frac{d-1}{d+1}$. Then the weighted partial recovery loss is $1-o_n(1)$.
\end{theorem}

\begin{proof}
Suppose $k\geq 1$. For $i,j\in [n]$, $i\not=j$, the probability that $\{i,j\}$ is contained in at least $k$ hyperedges of $\mathcal{H}$ is
\[
(1-p)^{\binom{n-2}{d-2}}\sum_{m=k}^{\binom{n-2}{d-2}} \binom{\binom{n-2}{d-2}}{m} \left(\frac{p}{1-p}\right)^m \leq \sum_{m\geq k} \left(\binom{n-2}{d-2}p\right)^m = O_n(n^{k(\delta-1)}).
\]
Then, the expected number of $\{i,j\}$ that are contained in at least $k$ hyperedges of $\mathcal{H}$ is $\\O_n(n^{2+k(\delta-1)})$. By selecting $k$ to be sufficiently large, this expected value will be $o_n(1)$, so the probability that there exists an edge $\{i, j\}$ that is contained in at least $k$ hyperedges is $o_n(1)$. 

Suppose $k$ is sufficiently large. Then, set $\mathcal{Z}'$ to be the set of $H\in\mathcal{Z}$ such that each edge $\{i, j\}$ is contained in less than $k$ hyperedges. We have that $\Pr[\mathcal{H}\in\mathcal{Z}']=1-o_n(1)$, so we can use $\mathcal{Z}'$ in place of $\mathcal{Z}$ by \Cref{lemma:highprob}.

We can now essentially use the same proof as \Cref{thm:apimpossible}, of course after replacing $\text{Proj}$ with $\text{Proj}_W$ and $\mathcal{Z}$ with $\mathcal{Z}'$. The only significant step is to prove the analogue of \Cref{thm:highintersect}; the remaining steps are straightforward to verify. For this proof, we can follow the framework given in \Cref{subsec:highintersectproof}.

The most important step is to justify the analogue of \Cref{lemma:squareprob}. Letting $\mathcal{G}''=\text{Proj}_W(\mathcal{Z}')$, since each edge of $\text{Proj}(H)$ is contained in less than $k$ hyperedges and the number of edges in $\text{Proj}(H)$ is at most $(1+\epsilon)\binom{d}{2}p\binom{n}{d}$ for $H\in\mathcal{Z}'$, we have that
\[
|\mathcal{G}''|\leq (k-1)^{(1+\epsilon)\binom{d}{2}p\binom{n}{d}}|\mathcal{G}'|,
\]
so using \Cref{lemma:squareprob} gives that
\[
\sum_{\substack{H, H'\in\mathcal{Z}', \\ \text{Proj}_W(H)=\text{Proj}_W(H')}} \frac{\Pr[\mathcal{H}=H]\Pr[\mathcal{H}=H']}{\Pr[\text{Proj}_W(\mathcal{H})=\text{Proj}_W(H)]^2}\leq |\mathcal{G}''| \leq \left(\frac{(1+o_n(1))(ke)^{\binom{d}{2}}}{q}\right)^{(1+\epsilon)p\binom{n}{d}}.
\]
Furthermore, it is clear that
\[
\sum_{\substack{H, H'\in\mathcal{Z}', \\ \text{Proj}_W(H) =\text{Proj}_W(H'), \\
|E(H)\cap E(H')|=I}} \Pr[\mathcal{H}=H]\Pr[\mathcal{H}=H'] \leq \sum_{\substack{H, H'\in\mathcal{Z}, \\ \text{Proj}(H)=\text{Proj}(H'), \\
|E(H)\cap E(H')|=I}} \Pr[\mathcal{H}=H]\Pr[\mathcal{H}=H']
\]
for all $I\in [Mp\binom{n}{d}, (1+\epsilon)p\binom{n}{d}]$, so the analogue 
\[
\sum_{\substack{H, H'\in\mathcal{Z},\, \text{Proj}_W(H) \\=\text{Proj}_W(H'), \\
|E(H)\cap E(H')|=I}} \Pr[\mathcal{H}=H]\Pr[\mathcal{H}=H'] \leq (1+o_n(1))p^{(I/(p\binom{n}{d})-2\epsilon)p\binom{n}{d}}q^{(1+\epsilon-I/(p\binom{n}{d}))p\binom{n}{d}} \alpha^{p\binom{n}{d}}
\]
of \pref{eq:sumprobpair} is true. Afterwards, we can follow the same steps to prove the analogue of \Cref{thm:highintersect}.
\end{proof}

\begin{proof}[Proof of \texorpdfstring{\Cref{thm:mainpartialweighted}}{}]
Since the partial recovery loss is $o_n(1)$ for $\delta<\frac{d-1}{d+1}$ by \Cref{thm:mainpartial}, the same is true for the weighted partial recovery loss, see \Cref{lemma:fundamental}. Afterwards using \Cref{thm:weightedapimpossible} and \Cref{thm:denseweighted} finishes the proof. 
\end{proof}

\section{Additional partial recovery results}

\subsection{The dense regime \texorpdfstring{$\delta \geq 1$}{}}
\label{subsec:dense}

In this section, we consider when $1\leq \delta<d-1$. In any other parts of the paper other than \Cref{sec:intro}, it is assumed that $\delta<1$. The goal of this subsection is to prove that the partial recovery and weighted partial recovery losses are $1-o_n(1)$ in the regime $\delta\geq 1$, see \Cref{thm:mainpartial} and \Cref{thm:mainpartialweighted}. We first prove that the weighted partial recovery loss is $1-o_n(1)$, since the partial recovery loss being $1-o_n(1)$ would directly follow, see \Cref{lemma:fundamental}.

\begin{theorem}
\label{thm:denseweighted}
Suppose $1\leq \delta<d-1$. The weighted partial recovery loss is $1-o_n(1)$.
\end{theorem}

\begin{proof} 
Suppose the function $f: \mathcal{G}_W\rightarrow\{0, 1\}$ satisfies the condition that if $G\in \mathcal{G}$ then $f(G) = \1\{[d] \in \mathcal{B}^*(G)\}$. Using \pref{eq:loss_equiv} gives that
\begin{equation}
\label{eq:losserr}
\E[|\mathcal{B}^*(\text{Proj}_W(\mathcal{H}))\Delta\mathcal{H}|] = \binom{n}{d}\Pr_\mathcal{H}[f(\text{Proj}_W(\mathcal{H}))\not= \1\{[d]\in E(\mathcal{H})\}].
\end{equation}
From Fano's inequality,
\[
H(\1\{[d]\in E(\mathcal{H})\}|\text{Proj}_W(\mathcal{H})) \leq H_B(\Pr_\mathcal{H}[f(\text{Proj}_W(\mathcal{H}))\not= \1\{[d]\in E(\mathcal{H})\}]).
\]
Using this gives that
\begin{equation}
\label{eq:entropylower}
\begin{split}
\binom{n}{d}H_B(\Pr_\mathcal{H}[f(\text{Proj}_W(\mathcal{H}))\not= \1\{[d]\in E(\mathcal{H})\}]) & \geq \sum_{h\in \binom{[n]}{d}} H(\1\{h\in\mathcal{H}\})|\text{Proj}_W(\mathcal{H})) \\
& \geq H(\mathcal{H}|\text{Proj}_W(\mathcal{H})) = H(\mathcal{H})-H(\text{Proj}_W(\mathcal{H})).
\end{split}
\end{equation}

First suppose $\delta>1$. Then, for all $i,j\in [n]$, $i\not=j$, we have that
\[
H(|\{h\in E(\mathcal{H}): \{i,j\}\subset h\}|) \leq \log\left(\binom{n-2}{d-2}\right).
\]
Thus,
\[
H(\text{Proj}_W(\mathcal{H})) \leq \sum_{1\leq i<j\leq n} H(|\{h\in E(\mathcal{H}): \{i,j\}\subset h\}|) \leq \binom{n}{2}\log\left(\binom{n-2}{d-2}\right).
\]
Since $H(\mathcal{H})=\binom{n}{d}H_B(p) = \Omega_n(n^{1+\delta}\log(n))$, we then have that $H(\text{Proj}_W(\mathcal{H}))=o_n(H(\mathcal{H}))$. Hence, \pref{eq:entropylower} gives that $\Pr[f(\text{Proj}_W(\mathcal{H}))\not= \1\{[d]\in E(\mathcal{H})\}]=(1-o_n(1))p$ so using \pref{eq:losserr} gives that
\[
\ell(\mathcal{B}^*)=\frac{\E_\mathcal{H}[|\mathcal{B}^*(\pg)\Delta\mathcal{H}|]}{p\binom{n}{d}} = \frac{\Pr_\mathcal{H}[f(\pg)\not= \1\{[d]\in E(\mathcal{H})\}]}{p} = 1-o_n(1).
\]

Next suppose $\delta=1$. Suppose $i,j\in [n]$, $i\not=j$. We have that 
\[
|\{h\in E(\mathcal{H}): \{i,j\}\subset h\}| \sim \text{Binomial}\left(\binom{n-2}{d-2}, p\right)
\]
so 
\[
H(|\{h\in E(\mathcal{H}): \{i,j\}\subset h\}|) = O_n(1)
\]
since $|\{h\in E(\mathcal{H}): \{i,j\}\subset h\}|$ has mean $\binom{n-2}{d-2}p=O_n(1)$, see \cite[Exercise I.4]{itcoding}. Then,
\[
H(\text{Proj}_W(\mathcal{H})) \leq \sum_{1\leq i<j\leq n} H(|\{h\in E(\mathcal{H}): \{i,j\}\subset h\}|) = O_n(n^2).
\]
Since $H(\mathcal{H})=\Omega_n(n^2\log(n))$, $H(\text{Proj}_W(\mathcal{H}))=o_n(H(\mathcal{H}))$ and we conclude similarly as the case $\delta>1$.
\end{proof}

\begin{corollary}
\label{corr:dense}
Suppose $1\leq \delta < d-1$. The partial recovery loss is $1-o_n(1)$.
\end{corollary}

\begin{proof}
This follows from \Cref{lemma:fundamental} and \Cref{thm:denseweighted}.
\end{proof}

\subsection{Correlation inequality}

It is in fact possible to extend the results from \Cref{thm:apimpossible} from one hyperedge to multiple hyperedges. As explained in the proof of the following result, the case $k=1$ is equivalent to \Cref{thm:apimpossible}.

\begin{theorem}
\label{thm:corr}
Suppose $\delta>\frac{d-1}{d+1}$ and $k\geq 1$. Then
\[
\sup_{\substack{h_1, \ldots, h_k \in \binom{[n]}{d} \\ \text{distinct}}} \sum_{G\in\mathcal{G}} \Pr[h_1,\ldots, h_k \in E(\mathcal{H})|\text{Proj}(\mathcal{H})=G]^2\Pr[\text{Proj}(\mathcal{H})=G] = o_n(p^k).
\]
\end{theorem}

\begin{proof}
First we resolve the case $k=1$. From \Cref{lemma:exactpartialequiv}, it suffices to prove that $\E_{\rhg,\rhg'}[E(\rhg)\cap E(\rhg')] = o_n\left(p\binom{n}{d}\right)$, which we show in the proof of \Cref{thm:apimpossible}.

Next, assume that $k\geq 2$. Suppose $h_1,\ldots, h_k\in\binom{[n]}{d}$. Suppose $S_i$, $1\leq i\leq k$ are disjoint sets of $\lfloor \frac{n}{k} \rfloor -d$ vertices that are disjoint from $h_i$, $1\leq i\leq k$, assuming that $n\geq kd$. Let $T_i=S_i\cup h_i$ for $1\leq i\leq k$. Let $\mathcal{H}_i$ be $\mathcal{H}$ with vertex set restricted to $T_i$ for $1\leq i\leq k$. Observe that the $\mathcal{H}_i$ do not have any overlapping hyperedges. Furthermore, let $\mathcal{H}^C$ denote $\mathcal{H}\backslash \left(\bigcup_{i=1}^k \mathcal{H}_i\right)$, that is, $\mathcal{H}^C$ is $\mathcal{H}$ restricted to $\binom{[n]}{d}\backslash \left(\bigcup_{i=1}^k \binom{T_i}{d}\right)$.

Suppose $G\in\mathcal{G}$. Where $G_i$ is some graph with vertex set $T_i$ for $1\leq i\leq k$ and $G^C$ is some graph with vertex set $[n]$, we have that 
\[
\Pr[\text{Proj}(\mathcal{H})=G] = \sum_{\substack{G_i, 1\leq i\leq k,\,G^C, \\ \text{projection is } G}}\prod_{i=1}^k\Pr[\text{Proj}(\mathcal{H}_i)=G_i] \Pr[\text{Proj}(\mathcal{H}^C)=G^C]
\]
and
\begin{align*}
& \Pr[h_1,\ldots,h_k\in E(\mathcal{H}), \text{Proj}(\mathcal{H})=G] \\ 
& = \sum_{\substack{G_i, 1\leq i\leq k,\,G^C, \\ \text{projection is } G}} \prod_{i=1}^k\Pr[\text{Proj}(\mathcal{H}_i)=G_i, h_i\in E(\mathcal{H}_i)] \Pr[\text{Proj}(\mathcal{H}^C)=G^C].
\end{align*}
Using the Cauchy-Schwarz inequality gives that
\begin{align*}
& \Pr[h_1,\ldots, h_k \in E(\mathcal{H})|\text{Proj}(\mathcal{H})=G]^2\Pr[\text{Proj}(\mathcal{H})=G] \\
& = \frac{\left(\sum_{\substack{G_i, 1\leq i\leq k,\,G^C, \\ \text{projection is } G}} \prod_{i=1}^k\Pr[\text{Proj}(\mathcal{H}_i)=G_i, h_i\in E(\mathcal{H}_i)] \Pr[\text{Proj}(\mathcal{H}^C)=G^C]\right)^2}{\sum_{\substack{G_i, 1\leq i\leq k,\,G^C, \\ \text{projection is } G}} \prod_{i=1}^k\Pr[\text{Proj}(\mathcal{H}_i)=G_i, h_i\in E(\mathcal{H}_i)] \Pr[\text{Proj}(\mathcal{H}^C)=G^C]}
\\ & \leq \sum_{\substack{G_i, 1\leq i\leq k,\,G^C, \\ \text{projection is } G}} \prod_{i=1}^k\frac{\Pr[\text{Proj}(\mathcal{H}_i)=G_i, h_i\in E(\mathcal{H}_i)]^2}{\Pr[\text{Proj}(\mathcal{H}_i)=G_i]}\Pr[\text{Proj}(\mathcal{H}^C)=G^C].
\end{align*}
Afterwards summing over the $G$ gives that
\begin{align*}
& \sum_{G\in\mathcal{G}} \Pr[h_1,\ldots, h_k \in E(\mathcal{H})|\text{Proj}(\mathcal{H})=G]^2\Pr[\text{Proj}(\mathcal{H})=G] \\
& \leq \sum_{G\in\mathcal{G}}\sum_{\substack{G_i, 1\leq i\leq k,\,G^C, \\ \text{projection is } G}} \prod_{i=1}^k\frac{\Pr[\text{Proj}(\mathcal{H}_i)=G_i, h_i\in E(\mathcal{H}_i)]^2}{\Pr[\text{Proj}(\mathcal{H}_i)=G_i]}\Pr[\text{Proj}(\mathcal{H}^C)=G^C] \\
& = \sum_{\substack{G_i, 1\leq i\leq k,\,G^C}} \prod_{i=1}^k\frac{\Pr[\text{Proj}(\mathcal{H}_i)=G_i, h_i\in E(\mathcal{H}_i)]
^2}{\Pr[\text{Proj}(\mathcal{H}_i)=G_i]}\Pr[\text{Proj}(\mathcal{H}^C)=G^C] \\
& = \sum_{G_i, 1\leq i\leq k} \prod_{i=1}^k\frac{\Pr[\text{Proj}(\mathcal{H}_i)=G_i, h_i\in E(\mathcal{H}_i)]^2}{\Pr[\text{Proj}(\mathcal{H}_i)=G_i]} \\
& = \prod_{i=1}^k\sum_{G_i}\frac{\Pr[\text{Proj}(\mathcal{H}_i)=G_i, h_i\in E(\mathcal{H}_i)]^2}{\Pr[\text{Proj}(\mathcal{H}_i)=G_i]}.
\end{align*}

Observe that the $\mathcal{H}_i$ follow the same random model as $\mathcal{H}$ but with a different value of $p$. The $\mathcal{H}_i$ have $\lfloor \frac{n}{k}\rfloor$ vertices and each hyperedge appears with probability $(c+o_n(1))n^{-d+1+\delta}$. Thus, when the $\mathcal{H}_i$ have $n$ vertices, each hyperedge appears with probability $(c+o_n(1))k^{-d+1+\delta}n^{-d+1+\delta}$. From repeating the case $k=1$ for this random model,
\begin{align*}
& \sum_{G\in\mathcal{G}} \Pr[h_1,\ldots, h_k \in E(\mathcal{H})|\text{Proj}(\mathcal{H})=G]^2\Pr[\text{Proj}(\mathcal{H})=G] \\
& \leq \prod_{i=1}^k\sum_{G_i}\frac{\Pr[\text{Proj}(\mathcal{H}_i)=G_i, h_i\in E(\mathcal{H}_i)]^2}{\Pr[\text{Proj}(\mathcal{H}_i)=G_i]} = o_n(p^k).
\end{align*}
\end{proof}

\begin{remark}
The previous proof exhibits an advantage of considering the parameterization $ p=(c+o_n(1))n^{-d+1+\delta}$.
\end{remark}

\section{Structures of the Projected Graph}
\label{sec:proj}

\subsection{Projection covers and convex relaxations}
\label{subsec:projcover}

First we give two lemmas that generalize ideas from \cite[Proof of Lemma 39]{reconstruct}. We do not include the proofs, which are straightforward.

\begin{lemma}
\label{lemma:intersectionsprob} Suppose $k\geq 2$ and $\mathcal{U}\subset 2^{[k]}$. The probability that there exists $h\in E(\mathcal{H})$ such that $h\cap [k]=u$ for all $u\in\mathcal{U}$ is 
\[
\prod_{u\in\mathcal{U}} (1-(1-p)^{\binom{n-k}{d-|u|}}).
\]
Furthermore this probability is at most
\[
p^{|\mathcal{U}|}\prod_{u\in\mathcal{U}} \binom{n-k}{d-|u|} = \Theta_n(n^{(1+\delta)|\mathcal{U}|-\sum_{u\in\mathcal{U}} |u|}).
\]
\end{lemma}

\begin{lemma}
\label{lemma:projcoverprob}
Suppose $k\geq 2$ and $\mathcal{E}\subset\binom{[k]}{2}$. The probability that $\mathcal{E}\subset E(\text{Proj}(\mathcal{H}))$ is at most
\[
\sum_\mathcal{U}p^{|\mathcal{U}|}\prod_{u\in\mathcal{U}} \binom{n-k}{d-|u|},
\]
where the sum is over $\mathcal{U}\subset 2^{[k]}$ satisfying the following conditions: 
\begin{itemize}
\item For all $u\in\mathcal{U}$, $|u|\geq 2$ and $\binom{u}{2}\not\subset \bigcup_{u'\in\mathcal{U}\backslash\{u\}} \binom{u'}{2}$. 
\item $\mathcal{E}\subset\bigcup_{u\in\mathcal{U}}\binom{u}{2}$.
\end{itemize}
\end{lemma}

\begin{remark}
$\text{Proj}(\mathcal{H})$ is a $2$-uniform hypergraph so $E(\text{Proj}(\mathcal{H}))$ is a set of edges, which are hyperedges with size $2$.
\end{remark}

Next we discuss a technique from \cite{reconstruct} that involves using a relaxation technique to establish a convex optimization problem after applying \Cref{lemma:intersectionsprob,lemma:projcoverprob}. Suppose $k\geq 2$ and $\mathcal{E}\subset\binom{[k]}{2}$. The goal is to upper bound the probability that $\mathcal{E}\subset E(\text{Proj}(\mathcal{H}))$ using \Cref{lemma:projcoverprob}. Suppose $\mathcal{U}$ satisfies the conditions of \Cref{lemma:projcoverprob}; the lemma implies that an upper bound on $p^{|\mathcal{U}|}\prod_{u\in\mathcal{U}}\binom{n-k}{d-|u|}$ is an upper bound on the probability that $\mathcal{E}\subset E(\text{Proj}(\mathcal{H}))$ after scaling by some constant since the number of $\mathcal{U}$ is finite.

First observe that  $|\mathcal{E}|\leq \sum_{u\in\mathcal{U}, |u|\geq 2}\binom{|u|}{2}$. The relaxation technique is to replace $\binom{x}{2}$ for some real variable $x\geq 2$ with the real variable $y\geq 1$; that is, $y=\binom{x}{2}$ and $x=\frac{1+\sqrt{1+8y}}{2}$. If $y_u=\binom{x_u}{2}$ for $u\in\mathcal{U}$ such that $|u|\geq 2$ then $|\mathcal{E}|\leq \sum_{u\in\mathcal{U}, |u|\geq 2} y_u$. From \Cref{lemma:intersectionsprob},
\[
p^{|\mathcal{U}|}\prod_{u\in\mathcal{U}}\binom{n-k}{d-|u|}=\Theta_n(n^{(1+\delta)|\mathcal{U}|-\sum_{u\in\mathcal{U}} |u|})= \Theta_n(n^{(1+\delta)|\mathcal{U}|-\sum_{u\in\mathcal{U}} \frac{1+\sqrt{1+8y_u}}{2}}).
\]
Suppose $M=|\mathcal{U}|$; the conditions of \Cref{lemma:projcoverprob} imply that $1\leq M\leq |\mathcal{E}|$. Then maximizing the quantity $p^{|\mathcal{U}|}\prod_{u\in\mathcal{U}}\binom{n-k}{d-|u|}$  corresponds to maximizing $(1+\delta)M-\sum_{i=1}^M\frac{1+\sqrt{1+8y_i}}{2}$ given that $\sum_{i=1}^M y_i\geq |\mathcal{E}|$ and $y_i\geq 1$ for $1\leq i\leq M$, which is a convex optimization problem. Particularly, since the function $(1+\delta)M-\sum_{i=1}^M\frac{1+\sqrt{1+8y_i}}{2}$ is convex, it is maximized at a vertex of the set of inputs.

Some results of \cite{reconstruct} that are proved using the methods described in this subsection are \cite[Lemmas 35, 39, and 40]{reconstruct}. In this paper we prove results in \Cref{sec:proj}, \Cref{thm:fakeedge}, and \Cref{ratio} using these methods.

\begin{remark}
For the proof of \Cref{lemma:covarq} we also impose the constraint that $|u|\leq d$ for all $u\in\mathcal{U}$, which corresponds to $y_i\leq \binom{d}{2}$ for $1\leq i\leq M$, because the hyperedges of $\mathcal{H}$ have size $d$. It is necessary to impose this constraint because the value $k$ in \Cref{lemma:intersectionsprob} and \Cref{lemma:projcoverprob} may be greater than $d$ in the context of \Cref{lemma:covarq}.
\end{remark}

\subsection{Proof of \texorpdfstring{\Cref{thm:fakeedge}}{}}
\label{subsec:fakeedge}
First we lower bound $q$. Note that if $[d]\in E(\mathcal{H})$ then $[d]\in E(\mathcal{H}_c)$ so $q\geq p$. Suppose $[d]\notin E(\mathcal{H})$; this event occurs with probability $1-p$. Then, $[d]\in E(\mathcal{H}_c)$ if and only if for each edge $\{i, j\}$ for $1\leq i<j\leq d$, there exists $h\in\binom{[n]}{d}\backslash \{[d]\}$ such that $\{i,j\}\subset h$ and $h\in E(\mathcal{H})$. 

Suppose $1\leq i<j\leq d$. The probability there exists $h\in\binom{[n]}{d}\backslash \{[d]\}$ such that $\{i,j\}\subset h$ and $h\in E(\mathcal{H})$ is $1-(1-p)^{\binom{n-2}{d-2}-1}$. Using the Harris inequality gives that 
\begin{align*}
\Pr[[d]\in E(\mathcal{H}_c)|[d]\notin E(\mathcal{H})] & \geq (1-(1-p)^{\binom{n-2}{d-2}-1})^{\binom{d}{2}} \\ & \geq \left(\binom{n-2}{d-2}p - p -O_n((n^{d-2}p)^2))\right)^{\binom{d}{2}} \\
& \geq (1-o_n(1))\left(\binom{n-2}{d-2}p\right)^{\binom{d}{2}} \\
& = (1-o_n(1))\left(\frac{cn^{\delta-1}}{(d-2)!}\right)^{\binom{d}{2}}.
\end{align*}
Hence,
\begin{equation}
\label{eq:qlower}
\begin{split}
q=\Pr[[d]\in E(\mathcal{H}_c)] & \geq p + (1-p)(1-o_n(1))\left(\frac{cn^{\delta-1}}{(d-2)!}\right)^{\binom{d}{2}} \\ 
& = p + \left(\frac{cn^{(\delta-1)}}{(d-2)!}\right)^{\binom{d}{2}} + o_n(n^{\binom{d}{2}(\delta-1)}).
\end{split}
\end{equation}

Next we upper bound $q$ using the technique discussed in \Cref{subsec:projcover}. From \Cref{lemma:projcoverprob} with $[k]$ replaced by $[d]$ and $\mathcal{E}$ replaced by $\binom{[d]}{2}$,
\[
q=\Pr[[d]\in E(\mathcal{H}_c)]\leq\sum_\mathcal{U} p^{|\mathcal{U}|}\prod_{u\in\mathcal{U}} \binom{n-d}{d-|u|},
\]
where the sum is over $\mathcal{U}\subset 2^{[d]}$ satisfying the conditions of the lemma. For convenience, denote the set of such $\mathcal{U}$ by $\mathcal{P}$.

Suppose $\mathcal{U}\in\mathcal{P}$ and $M=|\mathcal{U}|$. It is clear that $1\leq M\leq \binom{d}{2}$. If $M=1$ ($\mathcal{U}=\{[d]\}$) the probability is $p$ and if $M=\binom{d}{2}$ ($\mathcal{U}=\binom{[d]}{2}$) then the probability is $(1+o_n(1))\left(\frac{cn^{\delta-1}}{(d-2)!}\right)^{\binom{d}{2}}$. Using the union bound gives that
\begin{equation}
\label{eq:qupper1}
q\leq p + (1+o_n(1))\left(\frac{cn^{\delta-1}}{(d-2)!}\right)^{\binom{d}{2}} + \sum_{\substack{2\leq M\leq \binom{d}{2}-1, \\ \mathcal{U}\in\mathcal{P},\, |\mathcal{U}|=M}} \Theta_n(n^{(1+\delta)M - \sum_{u\in\mathcal{U}} |u|}).
\end{equation}

Suppose $\mathcal{U}\in\mathcal{P}$ and $M=|\mathcal{U}|$; recall that $1\leq M\leq \binom{d}{2}$. Furthermore suppose $y_u=\binom{|u|}{2}$ for $u\in\mathcal{U}$. Then $y_u\geq 1$ for $u\in\mathcal{U}$ and $\sum_{u\in\mathcal{U}} y_u \geq \binom{d}{2}$. Furthermore
\begin{equation}
\label{eq:qupper2}
M(1+\delta)-\sum_{u\in\mathcal{U}} |u| = M(1+\delta) - \sum_{u\in\mathcal{U}} \frac{1+\sqrt{8y_u+1}}{2}.
\end{equation}

Suppose $1\leq M\leq \binom{d}{2}$. Let $\mathcal{R}_M$ be the set of $(y_i)_{1\leq i\leq M}$ such that $y_i\geq 1$ for $1\leq i\leq M$ and $\sum_{i=1}^M y_i\geq \binom{d}{2}$. Let
\[
f(y_1, \ldots, y_M) = M(1+\delta) - \sum_{i=1}^m \frac{1+\sqrt{8y_i+1}}{2}.
\]
An upper bound of \pref{eq:qupper2} is the maximal value of bound $f$ over $\mathcal{R}_M$. Since $f$ is convex, this maximal value occurs at the vertex $y_i=1$ for $1\leq i\leq M-1$ and $y_M = \binom{d}{2}-M+1$. The value of $f$ at this vertex is
\[
g(M) := M(1-\delta) + 2 - \frac{1+\sqrt{8(\binom{d}{2}-M+1) + 1}}{2}
\]
and $\max_{y\in\mathcal{R}_M} f(y)=g(M)$. Observe that $g$ is convex in $M$ over $[1, \binom{d}{2}]$. Hence, the maximum value of $g$ for $M\in[1, \binom{d}{2}]$ is $g(1)=-d+1+\delta$ or $g(\binom{d}{2})=\binom{d}{2}(\delta-1)$. Particularly, if $1<M<\binom{d}{2}$ then $g(M)<\max(-d+1+\delta, \binom{d}{2}(\delta-1))$. Using the fact that an upper bound of \pref{eq:qupper2} is $g(M)$ then gives that if $\mathcal{U}\in\mathcal{P}$ and $1<|\mathcal{U}|<\binom{d}{2}$,
\[
|\mathcal{U}|(1+\delta)-\sum_{u\in\mathcal{U}} |u| < \max(-d+1+\delta, \binom{d}{2}(\delta-1)).
\]
Using \pref{eq:qupper1} then gives that
\[
\Pr[[d]\in E(\mathcal{H}_c)] \leq p+(1+o_n(1))\left(\frac{cn^{\delta-1}}{(d-2)!}\right)^{\binom{d}{2}} + o_n(p + n^{\binom{d}{2}(\delta-1)}).
\]
Using this inequality and \pref{eq:qlower} completes the proof.

\begin{remark}
Observe that \Cref{thm:fakeedge} for $\delta>\frac{d-1}{d+1}$ is implied by \Cref{thm:evcalc} with $l=m=0$.
\end{remark}

\subsection{Results about combinatorial structures in \texorpdfstring{$\mathcal{H}_c$}{}}

The main goal of this section is to prove that $E(\mathcal{H}_c)$ is concentrated around its mean if $\delta>\frac{d-1}{d+1}$, see \Cref{corr:varnumclique}. Observe that \Cref{lemma:evq,lemma:covarq} have similar statements and proofs as \cite[Lemma 35 and Lemma 40]{reconstruct}, but some differences are that we require different bounds and only consider when $\delta>\frac{d-1}{d+1}$. Furthermore we use methods discussed in \Cref{subsec:projcover} in this section.

\begin{lemma}
\label{lemma:evq}
Suppose $\delta>\frac{d-1}{d+1}$. Suppose $m$ is an integer such that $0\leq m\leq d-1$. Assume that $\{K_i: 1\leq i\leq M\}$ is a set of subsets $S$ of $[d]$ such that $2\leq |S|\leq d$. Assume that
\[
\bigcup_{i=1}^M \binom{K_i}{2} \supset \binom{[d]}{2}\backslash\binom{[m]}{2}.
\]
Then
\[
(1+\delta)M - \sum_{i=1}^M |K_i| \leq \left(\binom{d}{2}-\binom{m}{2}\right)(\delta-1).
\]
Equality occurs if and only if the $K_i$ are distinct and $\{K_i: 1\leq i\leq M\}=\binom{[d]}{2}\backslash\binom{[m]}{2}$.
\end{lemma}

\begin{proof}
First we may assume that $\binom{K_i}{2}\cap \left(\binom{[d]}{2}\backslash\binom{[m]}{2}\right)\not\subset\bigcup_{j\in [M]\backslash\{i\}} \binom{K_j}{2}$ for $1\leq i\leq M$. We can assume this because if the condition is not true we can remove $K_i$ and increase $(1+\delta)M-\sum_{i=1}^M |K_i|$. This condition implies that the $K_i$ are distinct and $M\leq\binom{d}{2}-\binom{m}{2}$.

Suppose $y_i=\binom{|K_i|}{2}$ for $1\leq i\leq M$, then $y_i\geq 1$ for $1\leq i\leq m$ and $\sum_{i=1}^M y_i \geq \binom{d}{2}-\binom{m}{2}$. We have that 
\[
(1+\delta)M-\sum_{i=1}^M |K_i| = (1+\delta)M - \sum_{i=1}^M \frac{1+\sqrt{1+8y_i}}{2}.
\]
Let
\[
f(y_1, \ldots, y_M) = (1+\delta)M - \sum_{i=1}^M \frac{1+\sqrt{1+8y_i}}{2}.
\]

Since $f$ is convex, its maximal value occurs at a vertex. Suppose $(y_i)_{1\leq i\leq M}$ is the vertex such that $y_i=1$ for $1\leq i\leq M-1$ and $y_M = \binom{d}{2}-\binom{m}{2}-M+1$. Then
\[
f(y_1, \ldots, y_M) = (\delta-1)M + 2 - \frac{1+\sqrt{1+8(\binom{d}{2}-\binom{m}{2}-M+1)}}{2}.
\]
Since $f(y_1, \ldots, y_M)$ is convex with respect to $M$, the maximum value of $f(y_1, \ldots, y_M)$ for $1\leq M\leq \binom{d}{2}-\binom{m}{2}$ occurs when $M\in\{1, \binom{d}{2}-\binom{m}{2}\}$. 

Suppose $M=1$. Then
\[
f(y_1, \ldots, y_M) = 1+\delta - \frac{1+\sqrt{1+8(\binom{d}{2}-\binom{m}{2})}}{2}.
\]
Because the number of edges in $[d]$ but not $[m]$ is greater than $M$, we also must prove that equality does not hold. Because $\delta>\frac{d-1}{d+1}$ it suffices to prove that 
\[
1+\frac{d-1}{d+1} - \frac{1+\sqrt{1+8(\binom{d}{2}-\binom{m}{2})}}{2} \leq -\left(\binom{d}{2}-\binom{m}{2}\right)\frac{2}{d+1}.
\]
This can be proved using expansion.

Next suppose $M=\binom{d}{2}-\binom{m}{2}$. Then $y_M=1$ so
\[
f(y_1, \ldots, y_M) = \left(\binom{d}{2}-\binom{m}{2}\right)(\delta-1).
\]
Afterwards it is straightforward to verify the equality case.
\end{proof}

\begin{lemma}
\label{lemma:covarq}
Suppose $\delta>\frac{d-1}{d+1}$. Suppose $m$ and $k$ are integers such that $ m\in\{0,1,2,d-1\}$ and $m\leq k\leq d-1$. Assume that $\{K_i: 1\leq i\leq M\}$ is a set of subsets $S$ of $[1, 2d-k]$ such that $2\leq |S|\leq d$. Assume that
\[
\bigcup_{i=1}^M \binom{K_i}{2}\supset \left(\binom{[d]}{2}\bigcup \binom{[k]\cup \{i: d+1\leq i\leq 2d-k\}}{2}\right)\backslash \binom{[m]}{2}
\]
Then
\begin{equation}
\label{eq:covar_ineq}
(1 + \delta)M-\sum_{i=1}^M |K_i| \leq k-m + (d(d-1)-m(m-1))(\delta-1).
\end{equation}
Equality occurs if and only if $k=m$, the $K_i$ are distinct, and $\\ \{K_i: 1\leq i\leq M\}= \left(\binom{[d]}{2}\bigcup \binom{[k]\cup \{i: d+1\leq i\leq 2d-k\}}{2}\right)\backslash \binom{[m]}{2}$.
\end{lemma}

\begin{proof}
Let 
\[
\mathcal{E} = \left(\binom{[d]}{2}\cup \binom{[k]\bigcup \{i: d+1\leq i\leq 2d-k\}}{2}\right)\backslash \binom{[m]}{2}.
\]
Similarly to the proof of \Cref{lemma:evq}, assume the condition \condition{cond:minimal} that for $1\leq i\leq M$, $\binom{K_i}{2}\cap \mathcal{E}\not\subset\bigcup_{j\in [M]\backslash\{i\}} \binom{K_j}{2}$. Note that \refcondition{cond:minimal} implies that the $K_i$ are distinct and that
\[
M \leq |\mathcal{E}| = 2\binom{d}{2} - \binom{k}{2} - \binom{m}{2}.
\]
We must prove that the equality case of \pref{eq:covar_ineq} occurs if and only if $M=2\binom{d}{2}-\binom{k}{2}-\binom{m}{2}$. Furthermore, since each element of $[2d-k]$ must be a vertex of one of the $K_i$ and $2d-k>d$, $M\geq 2$.

\noindent\textbf{Case 1: $m\in \{0, 1, 2\}$, $m < d-1$}

Suppose $m\in \{0, 1, 2\}$ and $m<d-1$. (We do not consider when $m=2$ and $d=3$.)

Suppose $y_i=\binom{|K_i|}{2}$ for $1\leq i\leq M$. Then, $1\leq y_i\leq \binom{d}{2}$ for $1\leq i\leq M$ and
\[
\sum_{i=1}^M y_i \geq 2\binom{d}{2}-\binom{k}{2}-\binom{m}{2}.
\]
Furthermore, the left hand side of \pref{eq:covar_ineq} is
\[
(1+\delta)M - \sum_{i=1}^M \frac{1+\sqrt{1+8y_i}}{2}.
\]
Let $\mathcal{R}_M$ be the set of $(y_i)_{1\leq i\leq M}$ such that $2\leq y_i\leq \binom{d}{2}$ for $1\leq i\leq M$ and $\sum_{i=1}^M y_i \geq 2\binom{d}{2}-\binom{k}{2}-\binom{m}{2}$. Additionally, let
\[
f(y_1, \ldots, y_M) =(1+\delta)M - \sum_{i=1}^M \frac{1+\sqrt{1+8y_i}}{2}.
\]
Observe that $f$ is convex so the maximum value of $f$ over $\mathcal{R}_M$ occurs at a vertex of $\mathcal{R}_M$. Suppose $(y_i)_{1\leq i\leq M}$ is a vertex of $\mathcal{R}_M$ such that $y_i\in\{2, \binom{d}{2}\}$ for $1\leq i\leq M-1$.

Suppose $2\leq M \leq \binom{d}{2}-\binom{k}{2}-\binom{m}{2}$. If $y_i=1$ for $1\leq i\leq M-1$ then
\[
y_M \geq 2\binom{d}{2}-\binom{k}{2}-\binom{m}{2} - M + 1 > \binom{d}{2},
\]
which is a contradiction. Assume that one of the $y_i$, $1\leq i\leq M-1$ equals $\binom{d}{2}$; it is clearly not optimal if two distinct $y_i$ for $1\leq i\leq M-1$ equal $\binom{d}{2}$ since we will then have that $\sum_{i=1}^M y_i > 2\binom{d}{2}$, so we can increase the value of $f$ by decreasing some of the $y_i$. Without loss of generality, assume that $y_i=1$ for $1\leq i\leq M-2$, $y_{M-1}=\binom{d}{2}$, and 
\[
y_M = \binom{d}{2}-\binom{k}{2}-\binom{m}{2}-M + 2.
\]
Then,
\[
f(y_1, \ldots, y_M) = (\delta-1)M+2(\delta+1) +4 -d -\frac{1+\sqrt{1+8y_M}}{2}.
\]
Since $M<2\binom{d}{2}-\binom{k}{2}-\binom{m}{2}$, we must prove that equality does not occur. Therefore, we must prove that
\[
(\delta-1)M +4 -d -\frac{1+\sqrt{1+8y_M}}{2} < k-m+(d(d-1)-m(m-1))(\delta-1)
\]
 It suffices to prove that this inequality is true for $2\leq M\leq \binom{d}{2}-\binom{k}{2}-\binom{m}{2}+1$ (observe that we add $\binom{d}{2}-\binom{k}{2}-\binom{m}{2}+1$ as a value for $M$ to simplify calculations). The left hand side is convex with respect to $M$, so it suffices to prove that the inequality is true for $M\in\{2, \binom{d}{2}-\binom{k}{2}-\binom{m}{2}+1\}$.

First, suppose $M=2$. The required inequality is
\begin{align*}
& 2(\delta-1) + 4 -d -\frac{1+\sqrt{1+8(\binom{d}{2}-\binom{k}{2}-\binom{m}{2})}}{2} \\
& < k-m + (d(d-1)-m(m-1))(\delta-1).
\end{align*}
Since $\delta>\frac{d-1}{d+1}$ and $d(d-1)-m(m-1) > 2$, it suffices to prove that
\[
-\frac{4}{d+1} + 4 - d - \frac{1+\sqrt{1+8(\binom{d}{2}-\binom{k}{2}-\binom{m}{2})}}{2} \leq k - m - \frac{2}{d+1}(d(d-1)-m(m-1)).
\]
This inequality can be verified by expansion. Note that equality occurs if and only if $k=m=0$ or $k=m=1$.

Although \Cref{lemma:covarq} is true for $m\in \{0, 1, 2, d-1\}$, we only use the case $m=0$, see \Cref{remark:validm}. Therefore, for completeness, we include the expansion for the case $m=0$ and equivalently $m=1$. It suffices to prove that
\begin{align*}
& -\frac{4}{d+1} + 4 - d - \frac{1+\sqrt{1+4d^2-4d-4k^2 + 4k}}{2} \leq k - \frac{2d(d-1)}{d+1} \\
\Leftrightarrow & \frac{2d^2 - 2d - 4}{d+1} + 4 - d - k - \frac{1}{2} \leq \frac{\sqrt{1+4d^2-4d-4k^2 + 4k}}{2} \\
\Leftrightarrow & 2d - 2k - 1 \leq \sqrt{1+4d^2 - 4d - 4k^2 + 4k} \Leftrightarrow 8k^2 \leq 8dk,
\end{align*}
which follows from $0\leq k \leq d-1$.

Next, suppose $M=\binom{d}{2}-\binom{k}{2}-\binom{m}{2}+1$. Then, $y_M=1$ so the required inequality is
\[
(\delta-1)\left(\binom{d}{2}-\binom{k}{2}-\binom{m}{2}+1\right) + 2 -d < k-m + (d(d-1)-m(m-1))(\delta-1).
\]
This is equivalent to
\[
(1-\delta)\left(\binom{d}{2} - \binom{m}{2} + \binom{k}{2} - 1\right) + 2 < k-m+d.
\]
We need to prove this inequality for $\delta>\frac{d-1}{d+1}$, so it suffices to prove that
\[
\frac{2}{d+1}\left(\binom{d}{2} - \binom{m}{2} + \binom{k}{2} - 1\right) + 2\leq k-m+d.
\]
This inequality is equivalent to
\[
(k-m)(d+2-k-m)\geq 0,
\]
which is true since $m\leq k\leq d-1$ and $m\leq 2$.

Suppose $\binom{d}{2}-\binom{k}{2}-\binom{m}{2}+1 \leq M \leq 2\binom{d}{2}-\binom{k}{2}-\binom{m}{2}$. It is clearly not optimal for two of the $y_i$ for $1\leq i\leq M-1$ to equal $\binom{d}{2}$. Suppose one of the $y_i$ for $1\leq i\leq M-1$ equals $\binom{d}{2}$. Without loss of generality, suppose $y_i=1$ for $1\leq i\leq M-2$ and $y_{M-1}=\binom{d}{2}$. Then,
\[
y_M \geq \binom{d}{2} - \binom{k}{2} - \binom{m}{2} - M + 2.
\]
Since $\binom{d}{2} - \binom{k}{2} - \binom{m}{2} - M + 2 \leq 1$, 
\begin{align*}
f(y_1, \ldots, y_{M-2}, y_{M-1}, y_M) & \leq f\left(y_1, \ldots, y_{M-2},\binom{d}{2}, 1\right) \\
& \leq f\left(y_1, \ldots, y_{M-2}, 2\binom{d}{2} - \binom{k}{2} -\binom{m}{2} - M + 1, 1\right) \\
& = f\left(y_1, \ldots, y_{M-2}, 1, 2\binom{d}{2} - \binom{k}{2} -\binom{m}{2} - M + 1\right). 
\end{align*}
Then, we may assume that $y_i=1$ for $1\leq i\leq M-1$. We have that
\[
y_M = 2\binom{d}{2} - \binom{k}{2}-\binom{m}{2} - M + 1.
\]
Furthermore, 
\[
f(y_1, \ldots, y_M) = (\delta-1)M + 2 - \frac{1+\sqrt{1+8(2\binom{d}{2} - \binom{k}{2}-\binom{m}{2} - M + 1)}}{2}.
\]
Observe that $f$ convex with respect to $M$. Therefore, $f$ is maximized over $\binom{d}{2}-\binom{k}{2}-\binom{m}{2}+1\leq M \leq 2\binom{d}{2}-\binom{k}{2}-\binom{m}{2}$ when $M\in \{\binom{d}{2}-\binom{k}{2}-\binom{m}{2}+1, 2\binom{d}{2}-\binom{k}{2}-\binom{m}{2}\}$.

Suppose $M=\binom{d}{2}-\binom{k}{2}-\binom{m}{2}+1$. Then $y_M=\binom{d}{2}$ so this case is equivalent to the previous case we considered where $(y_1, \ldots, y_M) = (1, \ldots, 1, \binom{d}{2}, 1)$.

Next suppose $M=2\binom{d}{2}-\binom{k}{2}-\binom{m}{2}$. Then $y_M=1$ so the required inequality is 
\[
(\delta - 1)\left(2\binom{d}{2}-\binom{k}{2}-\binom{m}{2}\right) \leq k-m + (d(d-1)-m(m-1))(\delta-1).
\]
This is equivalent to
\[
(1-\delta)\left(\binom{k}{2}-\binom{m}{2}\right) \leq k-m.
\]
If $k=m$ then it is clear that equality occurs. Suppose $k>m$. We must prove that 
\[
(1-\delta)\left(\binom{k}{2}-\binom{m}{2}\right) < k-m.
\]
Since $\delta>\frac{d-1}{d+1}$, it suffices to prove that
\[
\frac{2}{d+1}\left(\binom{k}{2}-\binom{m}{2}\right) \leq k-m.
\]
This is equivalent to 
\[
\frac{1}{d+1}(k-m)(k+m-1) \leq k-m,
\]
which is true since $m\leq 2$ and $k\leq d-1$.

\noindent\textbf{Case 2: $m= d-1$}

Next suppose $m=d-1$. Since $k=m=d-1$, we must prove that
\[
(1+\delta)M - \sum_{i=1}^M|K_i| \leq 2(d-1)(\delta-1)
\]
and that equality occurs if and only if $\{K_i: 1\leq i\leq M\}=\mathcal{E}$. Note that
\[
\mathcal{E}=\{\{d, j\}: j\in [d-1]\}\cup\{\{d+1, j\}: j\in [d-1]\}\subset \bigcup_{i=1}^M \binom{K_i}{2}.
\]
Suppose
\[
\{K_i: 1\leq i\leq M\} = S_a\sqcup S_b\sqcup S_{ab},
\]
where $K_i\in S_a$ if $K_i\cap \{d, d+1\} = \{d\}$, $K_i\in S_b$ if $K_i\cap \{d, d+1\}=\{d+1\}$, and $K_i\in S_{ab}$ if $K_i\cap \{d, d+1\} = \{d, d+1\}$ for $1\leq i\leq M$.

Let $Z=\{j: j\in [d-1], \exists k\in S_{ab} \text{ such that } j\in k\}$. For all $j\in Z$, both edges in $\mathcal{E}$ that contain $j$ are covered by an element of $S_{ab}$. Suppose $k\in S_a$ contains an element $j$ of $Z$. If $k=\{d,j\}$ then \refcondition{cond:minimal} will be contradicted so $|k|\geq 3$. If we remove $j$ from $k$ then the left hand side of \pref{eq:covar_ineq} will decrease but all of the edges of $\mathcal{E}$ will remain covered. Hence, we can assume that no element of $S_a$ contains an element of $Z$. We can similarly assume that no element of $S_b$ contains an element of $Z$.

Furthermore, assume that $j\in [d-1]$ and $k_1, k_2\in S_a$ satisfy $k_1\not=k_2$ and $j\in k_1\cap k_2$. If $k_1=\{d,j\}$ then \refcondition{cond:minimal} will be contradicted so $|k_1|\geq 3$. If we remove $j$ from $k_1$ then the left hand side of \pref{eq:covar_ineq} will decrease but all of the edges of $\mathcal{E}$ will remain covered. Hence, we can assume that no element of $[d-1]$ is contained in two elements of $S_a$. We can similarly assume that no element of $[d-1]$ is contained in two elements of $S_b$ and that no element of $[d-1]$ is contained in two elements of $S_{ab}$. 

Suppose $k\in S_a$ and $|k|\geq 3$. Suppose $j\in [d-1]\cap k$. Suppose we remove $j$ from $k$ and add $\{d, j\}$ to $\{K_i: 1\leq i\leq M\}$. Then, the left hand side of \pref{eq:covar_ineq} will increase by $\delta$. Hence, we can assume that $|k|=2$ for all $k\in S_a$ and similarly that $|k|=2$ for all $k\in S_b$. 

Since each element of $S_a$ is $\{d,j\}$ for some $j\in [d-1]\backslash Z$, $|S_a|=d-1-|Z|$. Similarly, $|S_b|=d-1-|Z|$. Furthermore, the left hand side of \pref{eq:covar_ineq} is 
\begin{align*}
2(\delta-1)(d-1-|Z|) + \sum_{k\in S_{ab}} 1+\delta - |k| = (\delta-1)(2d-2-2|Z| + |S_{ab}|) - |Z|.
\end{align*}
The inequality 
\[
(\delta-1)(2d-2-2|Z| + |S_{ab}|) - |Z| \leq (\delta-1)(2d-2)
\]
is equivalent to
\[
-(1-\delta)|S_{ab}| \leq (2\delta - 1)|Z|.
\]
If $|Z|=|S_{ab}|=0$, then the inequality holds with equality. Suppose $|Z|>0$. Then, the inequality is strict because $\delta>\frac{d-1}{d+1} \geq \frac{1}{2}$. Hence, equality holds if and only if $|Z|=|S_{ab}|=0$. Furthermore, $|Z|=|S_{ab}|=0$ if and only if $\{K_i: 1\leq i\leq M\} = \mathcal{E}$.
\end{proof}

\begin{remark}
\label{remark:validm}
It may be possible to generalize the previous result to more values of $m$. We use the case $m=0$ and intermediate results to prove \Cref{corr:varnumclique}, which is used in \Cref{subsec:impossibility}.
\end{remark}

\begin{theorem}
\label{thm:evcalc}
Suppose $\delta>\frac{d-1}{d+1}$. Suppose $m$ is an integer such that $0\leq m\leq d-1$ and $l$ is an integer such that $l\geq m$. Let $X$ be the set of elements $h$ of $\binom{[n]}{d}$ such that $h\cap[l]=[m]$ and $\binom{h}{2}\backslash\binom{[m]}{2}\subset E(\pg)$, where $[0]$ is the empty set. Then 
\[
\E[|X|] = (1+o_n(1))\binom{n}{d-m}\left(\frac{cn^{\delta-1}}{(d-2)!}\right)^{\binom{d}{2}-\binom{m}{2}}.
\]
\end{theorem}

\begin{proof}
The proof of this theorem has the same structure as the proof of \Cref{thm:fakeedge}. Suppose $h\in\binom{[n]}{d}$ and $h\cap [l]=[m]$. The number of $h$ is $\binom{n-l}{d-m}$. Hence, it suffices to prove that
\[
\Pr[h\in X] = \Pr\left[\binom{h}{2}\backslash \binom{[m]}{2}\subset E(\pg)\right] = (1+o_n(1))\left(\frac{cn^{\delta-1}}{(d-2)!}\right)^{\binom{d}{2}-\binom{m}{2}}.
\]

Using the Harris inequality gives that
\begin{align*}
\Pr[h\in X] & \geq (1-(1-p)^{\binom{n-2}{d-2}})^{\binom{d}{2}-\binom{m}{2}} \\ & \geq \left(\binom{n-2}{d-2}p - p -O_n((n^{d-2}p)^2)\right)^{\binom{d}{2}-\binom{m}{2}} \\
& \geq (1-o_n(1))\left(\binom{n-2}{d-2}p\right)^{\binom{d}{2}-\binom{m}{2}} \\
& = (1-o_n(1))\left(\frac{cn^{\delta-1}}{(d-2)!}\right)^{\binom{d}{2}-\binom{m}{2}}.
\end{align*}

From \Cref{lemma:projcoverprob} with $[k]$ replaced by $h$ and $\mathcal{E}$ replaced by $\binom{h}{2}\backslash\binom{[m]}{2}$, 
\[
\Pr\left[\binom{h}{2}\backslash\binom{[m]}{2}\subset E(\mathcal{H}_c)\right] \leq \sum_\mathcal{U} p^{|\mathcal{U}|}\prod_{u\in\mathcal{U}}\binom{n-d}{d-|u|},
\]
where the sum is over $\mathcal{U}\subset 2^h$ satisfying the conditions of the lemma.

If $\mathcal{U}=\binom{h}{2}\backslash\binom{[m]}{2}$, then
\[
p^{|\mathcal{U}|}\prod_{u\in\mathcal{U}}\binom{n-d}{d-|u|}\leq (1+o_n(1))\left(\frac{cn^{\delta-1}}{(d-2)!}\right)^{\binom{d}{2}-\binom{m}{2}}.
\]
Otherwise 
\[
p^{|\mathcal{U}|}\prod_{u\in\mathcal{U}}\binom{n-d}{d-|u|} = \Theta_n(n^{(1+\delta)|\mathcal{U}| - \sum_{u\in\mathcal{U}} |u|}) = o_n\left(n^{\left(\binom{d}{2}-\binom{m}{2}\right)(\delta-1)}\right)
\]
by \Cref{lemma:evq}. Therefore
\[
\Pr[[d]\in X] \leq (1+o_n(1))\left(\frac{cn^{\delta-1}}{(d-2)!}\right)^{\binom{d}{2}-\binom{m}{2}},
\]
which finishes the proof.
\end{proof}

\begin{theorem}
\label{thm:lowvar}
Suppose $\delta>\frac{d-1}{d+1}$. Suppose $m$ is an integer such that $m\in\{0, 1, 2, d-1\}$ and $l$ is an integer such that $l\geq m$. Let $X$ be the set of elements $h$ of $\binom{[n]}{d}$ such that $h\cap [l]=[m]$ and $\binom{h}{2}\backslash\binom{[m]}{2}\subset E(\pg)$, where $[0]$ is the empty set. Then $\Var[|X|] = o_n(\E[|X|]^2)$.
\end{theorem}

\begin{proof}
Let $\mathcal{S}$ be the set of elements of $\binom{[n]}{d}$ that have intersection with $[l]$ equal to $[m]$. We have that
\[
\E[|X|^2] = \sum_{a,b\in\mathcal{S}} \Pr[a,b \in X].
\]

Suppose $m\leq k\leq d-1$. Suppose $a,b\in\mathcal{S}$ and $k=|a\cap b|$. From \Cref{lemma:projcoverprob} with $[k]$ replaced by $a\cup b$ and $\mathcal{E}$ replaced by $(a\cup b)\backslash\binom{[m]}{2}$, 
\begin{equation}
\label{eq:probab}
\Pr[a,b\in X]=\Pr\left[\left(\binom{a}{2}\cup \binom{b}{2}\right)\backslash\binom{[m]}{2}\subset E(\mathcal{H}_c)\right]\leq \sum_\mathcal{U} p^{|\mathcal{U}|}\prod_{u\in\mathcal{U}}\binom{n-2d + k}{d-|u|},
\end{equation}
where the sum is over $\mathcal{U}\subset 2^{a\cup b}$ satisfying the conditions of the lemma. For convenience, denote the set of such $\mathcal{U}$ by $\mathcal{P}$.

Suppose $k>m$. From \Cref{lemma:covarq}, for $\mathcal{U}\in \mathcal{P}$ we have that
\[
(1+\delta)|\mathcal{U}|-\sum_{u\in\mathcal{U}} |u| < k-m + (d(d-1)-m(m-1))(\delta-1),
\]
where the equality case of the lemma cannot occur. Using \pref{eq:probab} then gives that
\[
\Pr[a,b\in X] = o_n(n^{k-m+(d(d-1)-m(m-1))(\delta-1)}).
\]

Suppose $k=m$. From \Cref{lemma:covarq}, for $\mathcal{U}\in\mathcal{P}$ we have that
\[
(1+\delta)|\mathcal{U}|-\sum_{u\in\mathcal{U}} |u| \leq (d(d-1)-m(m-1))(\delta-1)
\]
with equality if and only if $\mathcal{U}=(\binom{a}{2}\cup\binom{b}{2})\backslash\binom{[m]}{2}$. Using \pref{eq:probab} then gives that
\[
\Pr[a,b\in X] \leq (1+o_n(1))\left(\frac{cn^{\delta-1}}{(d-2)!}\right)^{d(d-1)-m(m-1)}.
\]

We therefore have that
\begin{align*}
\E[|X|^2] = & \sum_{a,b\in\mathcal{S}} \Pr[a,b \in X] 
= \sum_{k=m}^d \sum_{a,b\in\mathcal{S},\,|a\cap b| = k} \Pr[a,b\in X] \\
\leq & \E[|X|] + \sum_{k=m+1}^{d-1} \binom{n-l}{d-m}\binom{d-m}{k-m}\binom{n-l-d+m}{d-k} o_n(n^{k-m+d(d-1)(\delta-1)}) \\ 
& + \binom{n-l}{d-m}\binom{n-l-d+m}{d-m}(1+o_n(1))\left(\frac{cn^{\delta-1}}{(d-2)!}\right)^{d(d-1)-m(m-1)} \\
= & \E[|X|]+(1+o_n(1))\binom{n}{d-m}^2\left(\frac{cn^{\delta-1}}{(d-2)!}\right)^{d(d-1)-m(m-1)}.
\end{align*}
Furthermore, from \Cref{thm:evcalc},
\[
\E[|X|] \geq (1-o_n(1))\binom{n}{d-m}\left(\frac{cn^{\delta-1}}{(d-2)!}\right)^{\binom{d}{2}-\binom{m}{2}}.
\]
It follows that
\[
\Var[|X|] = \E[|X|^2] - \E[|X|]^2 = o_n(\E[|X|]^2),
\]
which finishes the proof.
\end{proof}

\begin{corollary}
\label{corr:varnumclique}
Suppose $\delta>\frac{d-1}{d+1}$. Then $\Var[e(\mathcal{H}_c)]=o_n(\E[e(\mathcal{H}_c)]^2)$.
\end{corollary}

\begin{proof}
This follows from \Cref{thm:lowvar} with $l=m=0$.
\end{proof}

\begin{remark} We expect \Cref{corr:varnumclique} to be true for $\delta\leq\frac{d-1}{d+1}$ as well although we omit a rigorous proof. For $\delta<\frac{d-1}{d+1}$, the idea is that almost all hyperedges of $\mathcal{H}_c$ are hyperedges of $\mathcal{H}$, the number of which we know is concentrated.
\end{remark}

\section{Exact Recovery}
\label{sec:exactproof}

\subsection{Proof of \texorpdfstring{\Cref{thm:mainexact}}{}}

From \cite[Theorem 4]{reconstruct}, the probability of exact recovery is $1-o_n(1)$ if $d=3$ and $\delta<\frac{2d-4}{2d-1}$. In this section we address the remaining cases of \Cref{thm:mainexact}. First we address the lower bound of the exact recovery threshold.

\begin{corollary}
\label{corr:threshold4}
Suppose $d=4$. If $\delta<\frac{2d-4}{2d-1}$ then the optimal probability of exact recovery is $1-o_n(1)$.
\end{corollary}

\begin{proof}
Suppose $\delta<\frac{2d-4}{2d-1}$. From \Cref{lemma:ambiguous}, it suffices to prove that for all ambiguous graphs $G_a$, the probability that $\text{Cli}(G_a)$ is a 2-connected component of $\mathcal{H}_c$ is $o_n(1)$. Suppose $G_a$ is an ambiguous graph. Let $P$ be the set of hypergraphs $h$ such that $\text{Proj}(h)=G_a$. Then,
\begin{equation}
\label{preimage}
\Pr[\text{Cli}(G_a) \text{ is a two-connected component of $\mathcal{H}_c$}] \leq \sum_{h\in P} \Pr[h\subset \mathcal{H}].
\end{equation}
Suppose $h\in P$. Using \Cref{thm:ambiguous_threshold} gives that
\[
\delta <  \frac{2d-4}{2d-1}\leq d-1-\frac{1}{m(h)}.
\]
Then, \Cref{lemma:threshold} implies that $\Pr[h\subset \mathcal{H}]=o_n(1)$. Using \pref{preimage} finishes the proof.
\end{proof}

\begin{corollary}
\label{corr:threshold5}
Suppose $d\geq 5$. If $\delta<\frac{d-1}{d+1}$ then the optimal probability of exact recovery is $1-o_n(1)$.
\end{corollary}
\begin{proof}
Since $\frac{2d-4}{2d-1}\geq \frac{d-1}{d+1}$ when $d\geq 5$, we can prove this theorem using the same argument as the proof of \Cref{corr:threshold4}. The only difference is in the application of \Cref{thm:ambiguous_threshold}. We have that using the theorem gives that if $h\in P$, then
\[
\delta < \frac{d-1}{d+1} \leq \frac{2d-4}{2d-1} \leq d-1-\frac{1}{m(h)},
\]
so \Cref{lemma:threshold} implies that $\Pr[h\subset \mathcal{H}]=o_n(1)$.
\end{proof}

\begin{remark}
It suffices to the weaker result \Cref{lemma:min} rather than \Cref{thm:ambiguous_threshold} in the proof of \Cref{corr:threshold5}. Using the lemma gives that if $h\in P$, then
\[
\delta < \frac{d-1}{d+1} \leq d-1-\frac{1}{m(h)},
\]
and similarly \Cref{lemma:threshold} implies that $\Pr[h\subset \mathcal{H}]=o_n(1)$.
\end{remark}

Next we address the upper bound of the exact recovery threshold. Observe that it suffices to prove that if $\delta> \min(\frac{d-1}{d+1}, \frac{2d-4}{2d-1})$, then the probability of exact recovery is $o_n(1)$. From \Cref{thm:mainpartial}, the partial recovery loss is $o_n(1)$ if $\delta>\frac{d-1}{d+1}$, which implies that the probability of exact recovery is $o_n(1)$ in this regime. Proving the following result completes the proof of \Cref{thm:mainexact}.

\begin{theorem}
\label{thm:exactupper1}
Suppose $d\geq 3$. If $\delta=\frac{2d-4}{2d-1}$ then the probability of exact recovery is $1-\Omega_n(1)$ and if $\delta>\frac{2d-4}{2d-1}$ then the probability of exact recovery is $o_n(1)$.
\end{theorem}

\begin{remark}
The case where $\delta\geq\frac{2d-4}{2d-1}$ implies exact recovery having $1-\Omega_n(1)$ probability has been proved in \cite[Appendix A]{reconstruct} for $3\leq d\leq 5$ using two-connected components. Furthermore, when $d\geq 5$ the probability of exact recovery is $o_n(1)$ if $\delta>\frac{2d-4}{2d-1}\geq \frac{d-1}{d+1}$ from \Cref{thm:mainpartial}. Hence the main contribution of this result is proving that the probability of exact recovery is $o_n(1)$ if $\delta>\frac{2d-4}{2d-1}$ and $d=3,4$.
\end{remark}

\begin{proof}[Proof of \Cref{thm:exactupper1}]
Suppose $\delta\geq \frac{2d-4}{2d-1}$. It suffices to prove that $\mathcal{A}^*$ fails with probability $\Omega_n(1)$ and $1-o_n(1)$ if $\delta>\frac{2d-4}{2d-1}$.

We consider a graph that is $G_{a, d}$ repeated over many two-connected components. Suppose $m\geq 1$ and $h$ is a $d$-uniform hypergraph with vertex set $V$. Suppose $V=\bigsqcup_{i=1}^m V_i$. For $1\leq i\leq m$ suppose 
\[
V_i = \{v_j^i: 1\leq j\leq d+1\} \bigsqcup_{1\leq j\leq d-1} S_j^{1; i}\bigsqcup_{1\leq j\leq d-1} S_j^{2; i},
\]
where $|S_j^{1; i}|=|S_j^{2; i}|=d-2$ for $1\leq j\leq d-1$. Furthermore, suppose that for $1\leq i\leq m$ the hyperedges of the subgraph of $h$ induced by $V_i$ are $\{v_1^i, \ldots, v_d^i\}$, $\{v_d^i, v_j^i, S_j^{1; i}\}$ for $j\in [d-1]$, and $\{v_{d+1}^i, v_j^i, S_j^{2; i}\}$ for $j\in[d-1]$. Additionally assume that $h$ has no other hyperedges.

We have that
\[
\Pr[\mathcal{A}^*(\pg)\not=\mathcal{H}|h\subset\mathcal{H}, \mathcal{H} \text{ is not minimal}] = 1.
\]

For all $G\in\{0,1\}^{\binom{[n]}{2}}$, let $S(G)$ be the set of hypergraphs $H$ such that $\text{Proj}(H)=G$, $h\subset H$, and $H$ is minimal. Let $\mathcal{G}^*$ be the set of $G$ such that $|S(G)|\geq 1$. 

Suppose $G\in\mathcal{G}^*$. Suppose $H\in S(G)$. For $1\leq i\leq m$, let $H^i$ be the hypergraph obtained from $H$ after the hyperedge for $\{v_1^i,\ldots,v_d^i\}$ is removed and the hyperedge for $\{v_1^i,\ldots,v_{d-1}^i, v_{d+1}^i\}$ is added. Note that $H$ and $H^i$ are distinct elements of $S(G)$ for $1\leq i\leq m$ so $|S(G)|\geq m+1$.

Furthermore, for all $G\in\mathcal{G}^*$, let $P(G)$ be $\Pr[\mathcal{H}=H]$ for a hypergraph $H$ such that $\text{Proj}(H)=G$ and $H$ is minimal. We have that
\begin{align*}
\Pr[h\subset\mathcal{H}, \mathcal{H} \text{ is minimal}] & = \sum_{h\subset H, H\text{ is minimal}} \Pr[\mathcal{H} = H] = \sum_{G\in\mathcal{G}^*}\sum_{H\in S(G)} \Pr[\mathcal{H}=H] \\
& = \sum_{G\in\mathcal{G}^*} P(G)|S(G)|.
\end{align*}
Furthermore,
\[
\Pr[\mathcal{A}^*(\pg)=\mathcal{H}, h\subset\mathcal{H}, \mathcal{H} \text{ is minimal}] = \sum_{G\in\mathcal{G}^*} \sum_{H\in S(G)} \Pr[\mathcal{A}(\pg)=H, \mathcal{H}=H].
\]
Suppose $G\in\mathcal{G}^*$. We have that
\begin{align*}
\sum_{H\in S(G)} \Pr[\mathcal{A}^*(\pg)=H, \mathcal{H}=H] & = \sum_{H\in S(G)} \Pr[\mathcal{A}^*(\pg)=H |\mathcal{H}=H]\Pr[\mathcal{H}=H] \\
& = P(G)\sum_{\substack{\text{Proj}(H)=G, h\subset H, \\ H \text{ is minimal}}} \Pr[\mathcal{A}^*(\pg)=H |\pg=G] \\
& \leq P(G).
\end{align*}
Therefore,
\[
\Pr[\mathcal{A}^*(\pg)=\mathcal{H}, h\subset\mathcal{H}, \mathcal{H} \text{ is minimal}] \leq \sum_{G\in\mathcal{G}^*} P(G).
\]
Since $|S(G)|\geq m$ for all $G\in\mathcal{G}^*$,
\begin{align*}
& \Pr[\mathcal{A}^*(\pg)=\mathcal{H}, h\subset\mathcal{H}, \mathcal{H} \text{ is minimal}] \leq \sum_{G\in\mathcal{G}^*} P(G) \leq \frac{1}{m+1}\sum_{G\in\mathcal{G}^*} P(G)|S(G)| \\
& = \frac{1}{m+1}\Pr[h\subset\mathcal{H}, \mathcal{H} \text{ is minimal}].
\end{align*}
It follows that 
\[
\Pr[\mathcal{A}^*(\pg)\not=\mathcal{H}, h\subset\mathcal{H}, \mathcal{H} \text{ is minimal}] \geq \frac{m}{m+1}\Pr[h\subset\mathcal{H}, \mathcal{H} \text{ is minimal}].
\]

Thus,
\begin{align*}
\Pr[\mathcal{A}^*(\pg)\not= \mathcal{H}]  \geq & \Pr[\mathcal{A}^*(\pg)\not=\mathcal{H}, h\subset\mathcal{H}, \mathcal{H} \text{ is not minimal}] \\ & + \Pr[\mathcal{A}^*(\pg)\not=\mathcal{H}, h\subset\mathcal{H}, \mathcal{H} \text{ is minimal}] \\
\geq &\Pr[h\subset\mathcal{H}, \mathcal{H} \text{ is not minimal}] + \frac{m}{m+1}\Pr[h\subset\mathcal{H}, \mathcal{H} \text{ is minimal}] \\
\geq & \frac{m}{m+1}\Pr[h\subset\mathcal{H}].
\end{align*}
Because 
\[
-\frac{1}{m(h)}=-d+1+\frac{2d-4}{2d-1}\leq -d+1+\delta,
\]
$\Pr[h\subset\mathcal{H}] = \Omega(1)$ so $\Pr[\mathcal{A}^*(\pg)\not=\mathcal{H}]=\Omega(1)$. Assume that $\delta>\frac{2d-4}{2d-1}$. Then, $-\frac{1}{m(h)}<-d+1+\delta$ so $\Pr[h\subset\mathcal{H}]=1-o_n(1)$. It follows that for all $m\geq 1$,
\[
\Pr[\mathcal{A}^*(\pg)\not= \mathcal{H}] \geq \frac{m}{m+1}(1-o_n(1)).
\]
Therefore $\Pr[\mathcal{A}^*(\pg)\not= \mathcal{H}]=1-o_n(1)$.
\end{proof}

\subsection{Proof of \texorpdfstring{\Cref{thm:mainexactweighted}}{}}

Observe that the probability of weighted exact recovery being $o_n(1)$ if $\delta>\frac{d-1}{d+1}$ follows from the weighted partial recovery loss being $1-o_n(1)$ if $\delta>\frac{d-1}{d+1}$ by \Cref{thm:mainpartialweighted}. Thus it suffices to prove that the probability of weighted exact recovery is $1-o_n(1)$ if $\delta<\frac{d-1}{d+1}$. 

First note that the analog of \Cref{lemma:ambiguous} is true with projections replaced by weighted projections. Afterwards we can use \Cref{thm:weightedratio} and the same argument as the proof of \Cref{corr:threshold4}.

\section{Ambiguous graph results}
\label{sec:ambiguous}

The goal of this section is to prove combinatorial results that will eventually justify upper bounds of the thresholds for exact recovery. We apply these results in \Cref{sec:exactproof} to prove \Cref{thm:mainexact,thm:mainexactweighted}.

\subsection{Optimization result}
\label{subsec:optimization}

\begin{lemma}
\label{ratio}
Suppose $d\geq 3$ and $\gamma\in [\frac{d-1}{d+1}, \infty)$. Assume that $h$ is a $d$-uniform hypergraph. Assume that the set of edges of $h$ is $E_h\sqcup I$, where $E_h$ and $I$ are disjoint. Suppose $U$ is a set of vertices of $h$ such that each hyperedge in $E_h$ is a subset of $U$. Suppose $\mathcal{P}$ is a set of edges such that for all $\{a, b\}\in\mathcal{P}$, $\{a,b\}\subset U$ and there exists $i\in I$ such that $\{a, b\}\subset i$. If
\[
\frac{(d-1)|E_h|-|U|+|\mathcal{P}|}{|E_h| + |\mathcal{P}|} \geq \gamma,
\]
then $d-1-\frac{1}{m(h)}\geq \min(\gamma, 1)$.
\end{lemma}

\begin{proof}
Assume that
\[
\frac{(d-1)|E_h|-|U|+|\mathcal{P}|}{|E_h| + |\mathcal{P}|} \geq \gamma.
\]
For all $i\in I$, let $x_i=|i\cap U|$. Let $Z$ be the set of $i\in I$ such that $x_i\leq 1$. Let $V$ be the set of vertices in $U$ or some hyperedge in $I\backslash Z$; that is, $V=U\bigcup_{i\in I\backslash Z} i$. Suppose $h'$ is the subgraph of $h$ that is induced by $V$.

Start with the vertex set $U$. The number of vertices is $|U|$. After adding the hyperedges $i\in I\backslash Z$, the number of vertices $|V|$ satisfies 
\[
|V|\leq |U| + \sum_{i\in I\backslash Z} |i\backslash U| =|U| + \sum_{i\in I\backslash Z} (d-x_i).
\]
Therefore,
\[
\alpha(h') \geq \frac{|E_h| + |I\backslash Z|}{|V|} \geq \frac{|E_h| + |I\backslash Z|}{|U|+\sum_{i\in I\backslash Z} (d-x_i)}.
\]
Note that
\begin{align*}
d-1-\frac{1}{m(h)}  & \geq d-1-\frac{1}{\alpha(h')} \\
& \geq d-1-\frac{|U|+\sum_{i\in I\backslash Z} (d-x_i)}{|E_h|+|I\backslash Z|} \\ & = \frac{(d-1)|E_h| -|U| +(d-1)|I\backslash Z| -\sum_{i\in I\backslash Z} (d-x_i)}{|E_h|+|I\backslash Z|} \\
& = \frac{(d-1)|E_h| -|U| - |I\backslash Z| +\sum_{i\in I\backslash Z} x_i}{|E_h|+|I\backslash Z|}.
\end{align*}
Hence, it suffices to prove that
\[
\frac{(d-1)|E_h| -|U| - |I\backslash Z| +\sum_{i\in I\backslash Z} x_i}{|E_h|+|I\backslash Z|} \geq \min(\gamma, 1).
\]

Next we use the technique from \Cref{subsec:projcover} to finish the proof. Observe that
\[
\sum_{i\in I\backslash Z} \binom{x_i}{2} \geq |\mathcal{P}|
\]
from the definition of $\mathcal{P}$. For all $i\in I\backslash Z$, let $y_i=\binom{x_i}{2}$. Note that
\[
x_i = \frac{1+\sqrt{1+8y_i}}{2}
\]
for all $i\in I\backslash Z$. Because $x_i\geq 2$ for all $i\in I\backslash Z$, $1\leq y_i \leq \binom{d}{2}$ for all $i\in I\backslash Z$. We have that
\begin{align*}
& \frac{(d-1)|E_h| -|U| - |I\backslash Z| +\sum_{i\in I\backslash Z} x_i}{|E_h|+|I\backslash Z|} \\ & \geq \min_{\substack{1\leq y_i\leq\binom{d}{2},\,i\in I\backslash Z, \\ \sum_{i\in I\backslash Z} y_i \geq |\mathcal{P}|}} \frac{(d-1)|E_h| -|U| - |I\backslash Z| +\sum_{i\in I\backslash Z} \frac{1+\sqrt{1+8y_i}}{2}}{|E_h|+|I\backslash Z|}.
\end{align*}
Hence, it suffices to prove that
\[
\min_{\substack{1\leq y_i\leq\binom{d}{2},\,i\in I\backslash Z, \\ \sum_{i\in I\backslash Z} y_i \geq |\mathcal{P}|}} \frac{(d-1)|E_h| -|U| - |I\backslash Z| +\sum_{i\in I\backslash Z} \frac{1+\sqrt{1+8y_i}}{2}}{|E_h|+|I\backslash Z|} \geq \min(\gamma, 1).
\]

Assume that $M=|I\backslash Z|$ and replace $I\backslash Z$ with $\{1, \ldots, M\}$, for simplicity. Furthermore, let $\mathcal{R}_M=\{(y_i)_{1\leq i\leq M}: 1\leq y_i\leq \binom{d}{2}, 1\leq i\leq M, \sum_{i=1}^M y_i\geq |\mathcal{P}|\}$ for $M\geq 1$.

\noindent\textbf{Case 1:} $M = 0$
\\ 
\indent If $M=0$, then $|\mathcal{P}|=0$, so
\begin{align*}
\frac{(d-1)|E_h| -|U| - M+\sum_{i=1}^M \frac{1+\sqrt{1+8y_i}}{2}}{|E_h|+M} & = \frac{(d-1)|E_h| -|U|}{|E_h|} \\ & = \frac{(d-1)|E_h|-|U|+|\mathcal{P}|}{|E_h| + |\mathcal{P}|} \geq \gamma.
\end{align*}

\noindent\textbf{Case 2:} $1\leq M\leq |\mathcal{P}|$
\\ 
\indent Suppose $1\leq M\leq |\mathcal{P}|$. Observe that 
\[
\frac{(d-1)|E_h| -|U| - M+\sum_{i=1}^M \frac{1+\sqrt{1+8y_i}}{2}}{|E_h|+M} 
\]
is concave in $(y_i)_{1\leq i\leq M}$, so the function is minimized over $\mathcal{R}_M$ at a vertex. 

At the vertex, $M-1$ of the values must be elements of $\{1, \binom{d}{2}\}$. Without loss of generality, assume that $y_j\in\{1,\binom{d}{2}\}$ for $1\leq j\leq M-1$. Assume that $A$ of these values equal $\binom{d}{2}$ and $M-A-1$ equal $1$. Then, since $\sum_{i=1}^M y_i \geq |\mathcal{P}|$,
\begin{equation}
\label{eq:linearcombo}
y_M - 1 + A\left(\binom{d}{2}-1\right) \geq |\mathcal{P}| - M
\Rightarrow (d-2)A + \frac{2(y_M-1)}{d+1} \geq \frac{2(\mathcal{P}-M)}{d+1}.
\end{equation}

Furthermore,
\begin{align*}
& \frac{(d-1)|E_h| -|U| - M+\sum_{i=1}^M \frac{1+\sqrt{1+8y_i}}{2}}{|E_h|+M}  \\
& = \frac{(d-1)|E_h| -|U| -M + Ad + 2M-2A-2 + \frac{1+\sqrt{1+8y_M}}{2}}{|E_h|+M} \\
& = \frac{(d-1)|E_h| -|U| + |\mathcal{P}| - (|\mathcal{P}|-M) + Ad-2A+\frac{-3+\sqrt{1+8y_M}}{2}}{|E_h|+|\mathcal{P}| - (|\mathcal{P}|-M)}.
\end{align*}
Let $X=(d-1)|E_h|-|U|+|\mathcal{P}|$, $Y=|E_h|+|\mathcal{P}|$, $W = |\mathcal{P}|-M$, and $\Delta = Ad-2A+\frac{-3+\sqrt{1+8y_M}}{2}$. Then,
\[
\frac{(d-1)|E_h| -|U| - M+\sum_{i=1}^M \frac{1+\sqrt{1+8y_i}}{2}}{|E_h|+M} = \frac{X-W+\Delta}{Y-W},
\]
and we know that $\frac{X}{Y} \geq \gamma \geq \frac{d-1}{d+1}$. Observe that
\[
\frac{(d-1)|E_h| -|U| - M+\sum_{i=1}^M \frac{1+\sqrt{1+8y_i}}{2}}{|E_h|+M}- \frac{X}{Y} = \frac{W(X-Y)+\Delta Y}{(Y-W)Y} = \frac{W(\frac{X-Y}{Y}) + \Delta}{Y-W}.
\]
Because $\frac{X-Y}{Y} \geq \gamma - 1 \geq -\frac{2}{d+1}$,
\[
\frac{(d-1)|E_h| -|U| - M+\sum_{i=1}^M \frac{1+\sqrt{1+8y_i}}{2}}{|E_h|+M}- \frac{X}{Y} \geq \frac{\Delta - \frac{2W}{d+1}}{Y-W}.
\]
The goal is to prove that $\frac{(d-1)|E_h| -|U| - M+\sum_{i=1}^M \frac{1+\sqrt{1+8y_i}}{2}}{|E_h|+M}\geq \frac{X}{Y}$. For this, it suffices to prove that $\Delta\geq \frac{2W}{d+1}$.

Since $\Delta= (d-2)A + \frac{-3+\sqrt{1+8y_M}}{2}$, in order to prove that $\Delta\geq \frac{2W}{d+1}$, using \pref{eq:linearcombo} gives that it suffices to prove that
\[
\frac{-3+\sqrt{1+8y_M}}{2} \geq \frac{2(y_M-1)}{d+1},
\]
where $1\leq y_M\leq \binom{d}{2}$. We have that $f(x) = \frac{-3+\sqrt{1+8x}}{2} - \frac{2(x-1)}{d+1}$ is concave, so $f(x)$ is minimized over the interval $[1, \binom{d}{2}]$ at its endpoints. Since $f(1)=f(\binom{d}{2}) = 0$, $f(x) \geq 0$ over $[1, \binom{d}{2}]$, which shows that $\Delta \geq \frac{2W}{d+1}$. Thus,
\[
\frac{(d-1)|E_h| -|U| - M+\sum_{i=1}^M \frac{1+\sqrt{1+8y_i}}{2}}{|E_h|+M} \geq \frac{X}{Y} \geq \gamma.
\]

\noindent\textbf{Case 3:} $M>|\mathcal{P}|$
\\ 
\indent Next, suppose $M>|\mathcal{P}|$. Then, 
\[
\frac{(d-1)|E_h| -|U| - M+\sum_{i=1}^M \frac{1+\sqrt{1+8y_i}}{2}}{|E_h|+M} 
\]
is minimized over $\mathcal{R}_M$ at $y_i=1$ for $1\leq i\leq M$. Therefore, over $\mathcal{R}_M$ we have that
\begin{align*}
\frac{(d-1)|E_h| -|U| - M+\sum_{i=1}^M \frac{1+\sqrt{1+8y_i}}{2}}{|E_h|+M} & \geq \frac{(d-1)|E_h|-|U|+M}{|E_h| + M} \\
& = \frac{(d-1)|E_h|-|U|+|\mathcal{P}| + (M-|\mathcal{P}|)}{|E_h| + |\mathcal{P}| + (M-|\mathcal{P}|)} \\
& \geq \min(\gamma, 1).
\end{align*}
We are done.
\end{proof}

\subsection{Lower bounds of \texorpdfstring{$m(h)$}{} for a nonminimal unweighted preimage \texorpdfstring{$h$}{}}

Suppose $G\in\mathcal{G}$ has covers $h$ and $g$ (i.e. $\text{Proj}(h)=\text{Proj}(g)=G$) such that $h\not=g$ and $e(h)\geq e(g)$. Note that $G$ is not necessarily ambiguous (despite the title of this section) and $g$ is not necessarily minimal. In this section we derive lower bounds for $d-1-\frac{1}{m(h)}$ for $d\geq 4$.

Let $E_h$ be the set of hyperedges in $h$ but not $g$ and $E_g$ be the set of hyperedges in $g$ but not $h$. Furthermore, let $I$ be the set of hyperedges in both $h$ and $g$. 

Let $\mathcal{E}_h$ be the set of edges of $\text{Proj}(E_h)$ and $\mathcal{E}_g$ be the set of edges of $\text{Proj}(E_g)$. Furthermore, let $\mathcal{E}$ be the set of edges of $\text{Proj}(I)$.

\begin{lemma}
\label{subset}
The set $\mathcal{E}_h$ is a subset of $\mathcal{E}\cup\mathcal{E}_g$ and the set $\mathcal{E}_g$ is a subset of $\mathcal{E}\cup\mathcal{E}_g$.
\end{lemma}

\begin{proof}
Suppose $\{i,j\}\in \mathcal{E}_h$. Then, $\{i,j\}$ is an edge of $\text{Proj}(h)=G$. Hence, $\{i,j\}$ is an edge of $\text{Proj}(g)=G$, which implies that $\{i,j\}\in \mathcal{E}\cup \mathcal{E}_g$.
\end{proof}

Let $\mathcal{P}$ be the symmetric difference of $\mathcal{E}_h$ and $\mathcal{E}_g$. From \Cref{subset}, $\mathcal{P}\subset\mathcal{E}$. Let $U$ be the set of vertices in some hyperedge in $E_h\cup E_g$. Observe that for all $\{a, b\}\in\mathcal{P}$, $\{a,b\}\subset U$ and because $\mathcal{P}\subset\mathcal{E}$, there exists $i\in I$ such that $\{a,b\}\subset i$. Hence, $E_h$, $I$, $U$, and $\mathcal{P}$ satisfy the conditions of \Cref{ratio}.

Let $V_h$ be the set of vertices that are in $E_h$ but not $E_g$, $V_g$ be the set of vertices that are in $E_g$ but not $E_h$, and $V_I$ be the set of vertices that are in both $E_g$ and $E_h$. Observe that
\[
U = V_h\sqcup V_g \sqcup V_I.
\]

Suppose $v\in U$. Let $d_h(v)$ be the number of elements of $E_h$ that contain $v$ and $d_h^*(v)$ be the number of $u\in U$ such that $\{u, v\}\in \mathcal{E}_h\backslash\mathcal{E}_g$. Similarly, let $d_g(v)$ be the number of elements of $E_g$ that contain $v$ and $d_g^*(v)$ be the number of $u\in U$ such that $\{u,v\}\in \mathcal{E}_g\backslash\mathcal{E}_h$.

Suppose $v\in V_I$. Let $k(v)$ be the largest positive integer $k$ such that there exists $i_h\in E_h$ and $i_g\in E_g$ such that $v\in i_h, i_g$ and $|i_h\backslash i_g| = |i_g\backslash i_h| = k$. Assume that $i_h\in E_h$ and $i_g\in E_g$ satisfy $v\in i_h, i_g$ and $|i_h\backslash i_g| = k(v)$. We have that $i_h\not=i_g$ so $ i_h\backslash i_g$ and $i_g\backslash i_h$ are nonempty. Let $i_h(v)=i_h$ and $i_g(v)=i_g$. If there are multiple choices for $(i_h, i_g)$, we can select one choice randomly. Let $n_h(v)$ be the number of $w\in i_h\backslash i_g$ such that $\{v, w\}\in \mathcal{E}_h\backslash\mathcal{E}_g$ and $n_g(v)$ be the number of $w\in i_g\backslash i_h$ such that $\{v, w\}\in \mathcal{E}_g\backslash\mathcal{E}_h$.

The following lemma is implied by \Cref{thm:general}. However, we include its proof since its contents motivate later methods.

\begin{lemma}
\label{lemma:min} 
Suppose $d\geq 5$. Then $d-1-\frac{1}{m(h)}\geq \frac{d-1}{d+1}$.
\end{lemma}

\begin{proof}
From \Cref{ratio}, it suffices to prove that
\[
\frac{(d-1)|E_h|-|U|+|\mathcal{P}|}{|E_h| + |\mathcal{P}|} \geq \frac{d-1}{d+1}.
\]
Thus, it suffices to prove that
\begin{equation}
\label{ineq}
d(d-1)|E_h| + 2|\mathcal{P}| \geq (d+1)|U|.
\end{equation} 

We have that
\[
d|E_h| = \sum_{v\in U} d_h(v) = \sum_{v\in V_h\cup V_I} d_h(v) \text{ and } d|E_g| = \sum_{v\in V_g\cup V_I} d_g(v).
\]
Furthermore,
\[
2|\mathcal{P}| = \sum_{v\in U} d^*_h(v) + d^*_g(v),
\]
so
\begin{equation}
\label{sum2}
2|\mathcal{P}| = \sum_{v\in V_h} d^*_h(v) + \sum_{v\in V_g} d^*_g(v) + \sum_{v\in V_I} (d^*_h(v) + d^*_g(v)).
\end{equation}

Suppose $v\in V_I$. For simplicty, let $i_h=i_h(v)$ and $i_g=i_g(v)$.

Assume that $n_h(v)=0$. Suppose $w\in i_h\backslash i_g$. Then, there exists $i\in E_g$ such that $\{u, w\}\in i$ because $\{u,w\}\in\mathcal{E}_h$ and $\{u,w\}\notin \mathcal{E}_h\backslash\mathcal{E}_g$. Because $w\notin i_g$, $i\not=i_g$. As $v\in i, i_g$, $d_g(v)\geq 2$. 

Assume that $n_h(v)>0$. Then, there exists $w\in i_h\backslash i_g$ such that $\{u, w\}\in \mathcal{E}_h\backslash\mathcal{E}_g$, so $d_h^*(v)\geq 1$. Observe that $d_g(v)\geq 1$ because $v\in i_g$.

Hence,
\[
d_g(v)\geq 1 + \1\{n_h(v) = 0\} \text{ and } d_h^*(v) \geq 1-\1\{n_h(v)=0\}.
\]
We similarly have that
\[
d_h(v)\geq 1 + \1\{n_g(v) = 0\} \text{ and } d_g^*(v) \geq 1-\1\{n_g(v)=0\}.
\]
Let
\[
N_h = \sum_{v\in V_I} \1\{n_h(v)=0\} \text{ and } N_g = \sum_{v\in V_I} \1\{n_g(v) = 0\}.
\]

We have that $d_h(v)\geq 1$ for all $v\in V_h$ and $d_g(v)\geq 1$ for all $v\in V_g$. Thus, 
\begin{equation}
\label{ineqh}
d|E_h| = \sum_{v\in V_h\cup V_I} d_h(v) \geq |V_h|+\sum_{v\in V_I} (1+\1\{n_g(v)=0\}) \geq |V_h|+|V_I|+N_g
\end{equation}
and similarly,
\begin{equation}
\label{ineqg}
d|E_g| \geq |V_g|+|V_I|+N_h.
\end{equation}
Adding \pref{ineqh} and \pref{ineqg} gives that
\begin{align*}
d(|E_h|+|E_g|) & \geq |V_h| + |V_g| + \sum_{v\in V_I} (2+\1\{n_h(v)=0\} + \1\{n_g(v)=0\}) \\
& = |V_h|+|V_g|+2|V_I| + N_h + N_g \\
& = |U| + |V_I| + N_h + N_g.
\end{align*}
Since $|E_h|\geq|E_g|$ because $g$ is minimal,
\begin{equation}
\label{ineqhg}
d|E_h| \geq \frac{|U|+|V_I|+N_h+N_g}{2}.
\end{equation}

Suppose $v\in V_h$. Suppose $v\in i$ for $i\in h$. We have that for all $w\in i$ such that $v\not=w$, $\{v, w\}\in \mathcal{E}_h\backslash\mathcal{E}_g$ because $v\in V_h$. Therefore, $d_h^*(v) \geq d-1$. Similarly, if $v\in V_g$, $d_g^*(v)\geq d-1$. Hence, \pref{sum2} gives that
\begin{equation}
\label{ineqp}
\begin{split}
2|\mathcal{P}| & \geq (d-1)|V_h| + (d-1)|V_g| + \sum_{v\in V_I} (2-\1\{n_h(v) = 0\} - \1\{n_g(v) = 0\}) \\
& = (d-1)(|V_h|+|V_g|) + 2|V_I| - N_h-N_g.
\end{split}
\end{equation}
We have that
\begin{align*}
d(d-1)|E_h| + 2|\mathcal{P}| & \geq \frac{(d-1)|U|}{2} + \frac{(d+3)|V_I|}{2}+(d-1)(|V_h|+|V_g|)+\frac{d-3}{2}(N_h+N_g) \\
& = \frac{(d-1)|U|}{2} + \frac{(d+3)|U|}{2} + \frac{d-5}{2}(|V_h|+|V_g|) + \frac{d-3}{2}(N_h+N_g) \\
& = (d+1)|U| + \frac{d-5}{2}(|V_h|+|V_g|) + \frac{d-3}{2}(N_h+N_g) \\
& \geq (d+1)|U|.
\end{align*}
This proves that \pref{ineq} is true, which finishes the proof.
\end{proof}

\begin{theorem}
\label{thm:min4}
Suppose $d=4$. Then, $d-1-\frac{1}{m(h)}\geq \frac{2d-4}{2d-1}=\frac{4}{7}$.
\end{theorem}

\begin{proof}
For the sake of contradiction, that $d-1-\frac{1}{m(h)}<\frac{4}{7} < \frac{3}{5}$, where $\frac{d-1}{d+1}=\frac{3}{5}$. From \Cref{ratio},
\[
\frac{(d-1)|E_h|-|U|+|\mathcal{P}|}{|E_h|+|\mathcal{P}|}<\frac{3}{5}.
\]
This is equivalent to
\begin{equation}
\label{inequ}
12|E_h|+2|\mathcal{P}|<5|U|.
\end{equation}

For the sake of contradiction, assume that $|V_h|=0$. Using \pref{ineqhg} gives that
\[
4|E_h|\geq \frac{|U|+|V_I|+N_h+N_g}{2}
\]
and using \pref{ineqp} gives that
\[
2|\mathcal{P}| \geq 3|V_g|+2|V_I|-N_h-N_g = 2|U|+|V_g|-N_h-N_g.
\]
Hence,
\begin{align*}
12|E_h|+2|\mathcal{P}| &\geq \frac{12}{8}(|U|+|V_I|+N_h+N_g)+2|U|+|V_g|-N_h-N_g \\
& = \frac{7}{2}|U|+(|V_I|+|V_g|) + \frac{1}{2}(|V_I| + N_h+N_g) \\  
& = \frac{9}{2}|U| + \frac{1}{2}(|V_I|+N_h+N_g).
\end{align*}
Therefore, \pref{inequ} implies that
\[
\frac{9}{2}|U| + \frac{1}{2}(|V_I|+N_h+N_g)<5|U|,
\]
so
\begin{equation}
\label{ineq3}
|V_I|+N_h+N_g<|U|.
\end{equation}

Additionally, using \pref{ineqg} gives that,
\[
4|E_h|\geq 4|E_g| \geq |V_g|+|V_I|+N_h=|U|+N_h.
\]
We therefore have that
\begin{align*}
12|E_h|+2|\mathcal{P}| &\geq 3(|U|+N_h)+2|U|+|V_g|-N_h-N_g \\
& =  5|U|+|V_g| + 2N_h-N_g.
\end{align*}
Thus, \pref{inequ} gives that
\[
5|U|+|V_g| + 2N_h-N_g<5|U| \Rightarrow 2N_h+|V_g| < N_g.
\]
Substituting this in \pref{ineq3} implies that 
\[
|V_I| + 3N_h + |V_g| < |U|,
\]
which is a contradiction to $|U|=|V_I|+|V_g|$. Thus, $|V_h|>0$.

Next we prove that $m(h)\geq \frac{1}{d-1-\frac{2d-4}{2d-1}}=\frac{7}{17}$. Suppose $v\in V_h$. Assume that $i\in h$ and $v\in i$. Suppose $i=\{v, u_1, u_2, u_3\}$. Observe that because $v\in V_h$, $\{v, u_1\}, \{v, u_2\}, \{v, u_3\} \in \mathcal{E}_h\backslash\mathcal{E}_g$. Hence, each of these edges is contained in an element of $I$ by \Cref{subset}. Let 
\[
\mathcal{I} = \{\{v, u_1\}, \{v, u_2\}, \{v, u_3\}\}.
\]

\noindent\textbf{Step 1: Covering the edges of $\mathcal{I}$}

We cannot cover the edges in $\mathcal{I}$ with one element of $I$. Suppose we cover the edges with $a, b\in I$, where $\{v, u_1, u_2\} \subset a$ and $\{v, u_3\}\subset b$ without loss of generality. Then, if $\kappa$ is the subgraph of $h$ induced by the vertices of $a$, $b$, and $i$, $e(\kappa)\geq 3$ and $v(\kappa)\leq 7$ so
\[
\frac{e(\kappa)}{v(\kappa)} \geq \frac{3}{7} > \frac{7}{17}.
\]
Hence, $m(h)>\frac{7}{17}$. 

Suppose we cover the edges in $\mathcal{I}$ with $a,b,c\in I$, where $\{v, u_1\}\subset a$, $\{v, u_2\}\subset b$, and $\{v, u_3\}\subset c$. Let $\mathfrak{f}(a)=\{v, u_1\}$, $\mathfrak{f}(b)=\{v, u_2\}$, $\mathfrak{f}(c)=\{v, u_3\}$, and $\mathfrak{f}(i)=\{v, u_1, u_2, u_3\}$. For any two distinct elements $x,y\in\{a,b,c,i\}$, $\mathfrak{f}(x)\cap \mathfrak{f}(y) \subset x\cap y$ since $\mathfrak{f}(x)\subset x$ and $\mathfrak{f}(y)\subset y$. Suppose there exists two distinct elements $x,y\in \{a,b,c,i\}$ such that $\mathfrak{f}(x)\cap \mathfrak{f}(y)$ is a strict subset of $x\cap y$. Then, if $\kappa$ is the subgraph of $h$ induced by the vertices of $a$, $b$, $c$, and $i$, then $e(\kappa)\geq 4$ and $v(\kappa)\leq 9$ so
\[
\frac{e(\kappa)}{v(\kappa)} \geq \frac{4}{9} > \frac{7}{17}.
\]
Assume that for any two distinct elements $x,y\in\{a,b,c,i\}$, $\mathfrak{f}(x)\cap \mathfrak{f}(y) = x\cap y$. Then, if $\kappa$ is the subgraph of $h$ induced by the vertices of $a$, $b$, $c$, and $i$, $e(\kappa)\geq 4$ and $v(\kappa)=10$.

\noindent\textbf{Step 2: Covering $\{u_1, u_2\}$, $\{u_2, u_3\}$, and $\{u_3, u_1\}$}

We have that the edges $\{u_1, u_2\}$, $\{u_2, u_3\}$, and $\{u_3, u_1\}$ must be covered by elements of $E_g\cup I$ by \Cref{subset}. 

Assume that $d\in (E_h\cup I)\backslash\{a,b,c,i\}$ and $|d\cap \{u_1, u_2, u_3\}|\geq 2$. If $\kappa$ is the subgraph of $h$ induced by the vertices of $a$, $b$, $c$, $d$, and $i$, then
\[
\frac{e(\kappa)}{v(\kappa)}\geq \frac{5}{12}>\frac{7}{17}.
\]

Next, assume the condition \condition{cond:intersection} that there does not exist $d\in (E_h\cup I) \backslash \{a,b,c,i\}$ such that $|d\cap \{u_1, u_2, u_3\}|\geq 2$. In particular, this implies that the edges $\{u_1, u_2\}$, $\{u_2, u_3\}$, and $\{u_3, u_1\}$ must be covered by elements of $E_g$.

Let $\mathcal{V}$ be the set of vertices of $a$, $b$, $c$, and $i$; observe that $|\mathcal{V}|=10$. 

\noindent\textbf{Step 2.1: Covering $\{u_1, u_2\}$}

Suppose $d\in E_g$ and $\{u_1, u_2\}\subset d$. The two cases are $d\cap \{u_1, u_2, u_3\}$ equals $\{u_1, u_2\}$ or $\{u_1, u_2, u_3\}$.

\noindent \textbf{Step 2.1.1: $d\cap \{u_1, u_2, u_3\} = \{u_1, u_2, u_3\}$}

Assume that $d\cap \{u_1, u_2, u_3\}=\{u_1, u_2, u_3\}$. Suppose $d=\{w, u_1, u_2, u_3\}$. By \Cref{subset}, all edges in the set
\[
\mathcal{S}=\{\{w, u_1\}, \{w, u_2\}, \{w, u_3\}\}
\]
must be covered by an element of $E_h\cap I$.

Assume that $w\in \mathcal{V}$. For the sake of contradiction, assume that all edges in $\mathcal{S}$ are covered by an element of $\{a,b,c,i\}$. Then, $\{w, u_1\}$ is covered by some element of $\{a,b,c,i\}$. Thus, $w\in a\cup i$ since $u_1\notin b, c$. Since $w\notin i$, $w\in a$. Similarly, $w\in b, c$. Hence, $w\in a\cap b\cap c = \{v\}$, which is a contradiction to $d\not=i$. Suppose $e\in E_h\cup I$ is not an element of $\{a,b,c,i\}$ and covers some element of $\mathcal{S}$. If $\kappa$ is the subgraph of $h$ induced by the vertices of $a$, $b$, $c$, $e$, and $i$, $e(\kappa)\geq 5$ and $v(\kappa)\leq 12$ so
\[
\frac{e(\kappa)}{v(\kappa)}\geq \frac{5}{12}>\frac{7}{17}.
\]

Assume that $w\notin\mathcal{V}$. Then, no element of $\mathcal{S}$ is covered by an element of $\{a,b,c,i\}$. By \refcondition{cond:intersection}, each of the elements of $\mathcal{S}$ must be covered by a distinct element of $(E_h\cup I)\backslash \{a,b,c,i\}$, otherwise two elements of $\{u_1, u_2, u_3\}$ will be contained in a single element of $(E_h\cup I)\backslash \{a,b,c,i\}$. Hence, this is the only case that we consider.

Suppose the elements of $\mathcal{S}$ can be covered by three elements of $(E_h\cup I)\backslash \{a,b,c,i\}$ but not less than three. Then, there exists $e,f,j\in E_h\cup I$ such that $\{w, u_1\}\subset e$, $\{w, u_2\}\subset f$, and $\{w, u_3\}\subset j$. If $\kappa$ is the subgraph of $h$ induced by the vertices $a$, $b$, $c$, $e$, $f$, $j$, and $i$, $e(\kappa)\geq 7$ and $v(\kappa)\leq 17$ so
\[
\frac{e(\kappa)}{v(\kappa)} \geq \frac{7}{17}.
\]

\noindent\textbf{Step 2.1.2: $d\cap\{u_1, u_2, u_3\} = \{u_1, u_2\}$}

Suppose $d\cap \{u_1, u_2, u_3\} = \{u_1, u_2\}$. Suppose $d=\{u_1, u_2, w_1, w_2\}$. Observe that the five edges in the set
\[
\mathcal{S} = \{\{u_1, w_1\}, \{u_1, w_2\}, \{u_2, w_1\}, \{u_2, w_2\}, \{w_1, w_2\}\}
\]
must be covered by elements of $E_h\cup I$ by \Cref{subset}.

Assume that $w_1, w_2\in\mathcal{V}$. For the sake of contradiction, assume that all edges in $\mathcal{S}$ are covered by some element of $\{a, b, c, i\}$. Then, $\{w_1, u_1\}$ must be covered by $a$ or $i$, so $w_1\in a\cup i$. Also, $\{w_1, u_2\}$ must be covered by $b$ or $i$, so $w_1\in b\cup i$. Observe that $(a\cup i)\cap (b\cup i) = i$, so $w_1\in i$. Similarly, $w_2\in i$. This is a contradiction to $d\not= i$. Thus, some edge in $\mathcal{S}$ is not covered by some element of $\{a,b,c,i\}$. This edge must be covered by $e\in E_h\cup I$. Note that because $d\subset\mathcal{V}$ and $|d\cap e| \geq 2$, $|e\backslash \mathcal{V}|\leq 2$. Then, if $\kappa$ is the subgraph of $h$ induced by the vertices of $a$, $b$, $c$, $e$, and $i$, $e(\kappa)\geq 5$ and $v(\kappa)\leq 12$ so
\[
\frac{e(\kappa)}{v(\kappa)} \geq \frac{5}{12} > \frac{7}{17}.
\]

Assume that $|\{w_1, w_2\}\cap\mathcal{V}| = 1$. Without loss of generality, assume that $w_1\in\mathcal{V}$ and $w_2\notin\mathcal{V}$. We have that the edges $\{w_2, w_1\}$, $\{w_2, u_1\}$, and $\{w_2, u_2\}$ in $\mathcal{S}$ are not covered by elements of $\{a,b,c,i\}$. Let 
\[
\mathcal{Q}=\{\{w_2, w_1\},\{w_2, u_1\}, \{w_2, u_2\}\}.
\]
The cases that we must consider are that the elements of $\mathcal{Q}$ are covered by two or three elements of $E_h\cup I$. By \refcondition{cond:intersection}, they cannot be covered by one element of $E_h\cup I$.

Suppose the elements of $\mathcal{Q}$ can be covered by two elements of $E_h\cup I$. There exists $e,f\in E_h\cup I$ such that $w_2\in e, f$ and $\{w_1, u_1, u_2\}\subset e\cup f$. Suppose $e,f\in E_h\cup I$ such that $\{w_2, w_1, u_1\} \subset e$ and $\{w_2, u_2\} \subset f$, without loss of generality; note that the vertices $\{w_1, u_1, u_2\}$ can be considered to be equivalent for the purposes of this computation. If $\kappa$ is the subgraph of $h$ induced by the vertices of $a$, $b$, $c$, $e$, $f$, and $i$, $e(\kappa)\geq 6$ and $v(\kappa)\leq 14$ so
\[
\frac{e(\kappa)}{v(\kappa)} \geq \frac{6}{14} > \frac{7}{17}.
\]

Suppose the elements of $\mathcal{Q}$ can be covered by three elements of $E_h\cup I$ but not less than three elements. Then, there exists $e,f, j\in E_h\cup I$ such that $\{w_2, u_1\} \subset e$ and $\{w_2, u_2\} \subset f$, and $\{w_2, w_1\}\subset j$. Assume that $e$, $f$, and $j$ satisfy this condition. If $\kappa$ is the subgraph of $h$ induced by the vertices of $a$, $b$, $c$, $e$, $f$, $j$, and $i$, $e(\kappa)\geq 7$ and $v(\kappa)\leq 17$ so
\[
\frac{e(\kappa)}{v(\kappa)} \geq \frac{7}{17}.
\]

Next, assume that $w_1, w_2\notin\mathcal{V}$. Then, none of the elements of $\mathcal{S}$ are covered by elements of $\{a, b, c, i\}$. The cases we consider are when the elements of $\mathcal{S}$ are covered by at least $2$ and at most $5$ elements of $E_h\cup I$. Furthermore, recall that there does not exist an element of $(E_h\cup I)\backslash \{a, b,c,i\}$ that contains $\{u_1, u_2\}$ by \refcondition{cond:intersection}.

Suppose the elements of $\mathcal{S}$ can be covered by two elements of $E_h\cup I$. There exists $e,f\in E_h\cup I$ such that $\{w_1, w_2, u_1\}\subset e$ and $\{w_1, w_2, u_2\}\subset f$. Assume that $e$ and $f$ satisfy this condition. If $\kappa$ is the subgraph of $h$ induced by the vertices of $a$, $b$, $c$, $e$, $f$, and $i$, then $e(\kappa)\geq 6$ and $v(\kappa)\leq 14$ so
\[
\frac{e(\kappa)}{v(\kappa)}\geq \frac{6}{14} > \frac{7}{17}.
\]

Suppose the elements of $\mathcal{S}$ can be covered by three elements of $E_h\cup I$ but not less than three elements. Then, there exists $e,f,j\in E_h\cup I$ such that either $\{w_1, w_2, u_1\}\subset e$, $\{w_1, u_2\}\subset f$, and $\{w_2, u_2\}\subset j$ or $\{w_1, w_2, u_2\}\subset e$, $\{w_1, u_1\}\subset f$, and $\{w_2, u_1\}\subset j$. Without loss of generality, suppose $e,f,j\in E_h\cup I$ satisfy the condition that $\{w_1, w_2, u_1\}\subset e$, $\{w_1, u_2\}\subset f$, and $\{w_2, u_2\}\subset j$. If $\kappa$ is the subgraph of $h$ induced by the vertices of $a$, $b$, $c$, $e$, $f$, $j$, and $i$, then $e(\kappa)\geq 7$ and $v(\kappa)\leq 17$ so
\[
\frac{e(\kappa)}{v(\kappa)}\geq\frac{7}{17}.
\]

Next, suppose the elements of $\mathcal{S}$ can be covered by four elements of $E_h\cup I$ but not less than four elements. Suppose $e,f,j,k\in E_h\cup I$ cover the elements of $\mathcal{S}$. Since $|\mathcal{S}|=5$, at least one of $e,f,j,k$ must contain three elements of $\{u_1, u_2, w_1, w_2\}$. Without loss of generality, suppose $e$ satisfies this condition. Since $e$ cannot contain $\{u_1, u_2\}$, suppose $\{w_1, w_2, u_1\}\subset e$, without loss of generality. The uncovered edges of $\mathcal{S}$ are $\{w_1, u_2\}$ and $\{w_2, u_2\}$. There exists two elements of $\{f, j, k\}$ that covers both of these edges, so the elements of $\mathcal{S}$ can be covered by at most three elements, which is a contradiction.

Suppose the elements of $\mathcal{S}$ are covered by five elements of $E_h\cup I$ but not less than five elements. Suppose $e,f,j,k,l\in E_h\cup I$ and $\{u_1, w_1\}\subset e$, $\{u_1, w_2\}\subset f$, $\{u_2, w_1\}\subset j$, $\{u_2, w_2\}\subset k$, and $\{w_1, w_2\}\subset l$. Let $\mathfrak{f}(e)=\{u_1, w_1\}$, $\mathfrak{f}(f)=\{u_1, w_2\}$, $\mathfrak{f}(j)=\{u_2, w_1\}$, $\mathfrak{f}(k)=\{u_2, w_2\}$, and $\mathfrak{f}(l)=\{w_1, w_2\}$. For any two distinct elements $x,y\in\{a,b,c,e,f,j,k,l,i\}$, $\mathfrak{f}(x)\cap \mathfrak{f}(y) \subset x\cap y$ since $\mathfrak{f}(x)\subset x$ and $\mathfrak{f}(y)\subset y$. Suppose there exists two distinct elements $x,y\in \{a,b,c,e,f,j,k,l,i\}$ such that $\mathfrak{f}(x)\cap \mathfrak{f}(y)$ is a strict subset of $x\cap y$. Then, if $\kappa$ is the subgraph of $h$ induced by the vertices of $a$, $b$, $c$, $e$, $f$, $j$, $k$, $l$, and $i$, then $e(\kappa)\geq 9$ and $v(\kappa)\leq 21$ so
\[
\frac{e(\kappa)}{v(\kappa)} \geq \frac{9}{21} > \frac{7}{17}.
\]
Assume that for any two distinct elements $x,y\in\{a,b,c,e,f,j,k,l,i\}$, $\mathfrak{f}(x)\cap \mathfrak{f}(y) = x\cap y$. Then, if $\kappa$ is the subgraph of $h$ induced by the vertices of $a$, $b$, $c$, $e$, $f$, $j$, $k$, $l$, and $i$, $e(\kappa)\geq 9$ and $v(\kappa)=22$.

\noindent\textbf{Step 2.2: Covering $\{u_3, u_1\}$}

There exists $d'\in E_g$ such that $\{u_3, u_1\}\subset d'$. Assume that $d'\cap \{u_1, u_2, u_3\} = \{u_3, u_1\}$; we have already considered the case $d'\cap\{u_1, u_2, u_3\}=\{u_1, u_2, u_3\}$ in the case $d\cap \{u_1, u_2, u_3\}=\{u_1, u_2, u_3\}$. Suppose $d'=\{u_3, u_1, w_1', w_2'\}$. By symmetry, we have that the only case we must consider is if $w_1',w_2'\notin\mathcal{V}$ and the edges in the set
\[
\{\{u_3, w_1'\}, \{u_3, w_2'\}, \{u_1, w_1'\}, \{u_1, w_2'\}, \{w_1', w_2'\}\}
\]
are covered by five elements $e'$, $f'$, $j'$, $k'$, and $l'$ of $E_h\cap I$. Suppose $\{u_3, w_1'\}\subset e'$, $\{u_3, w_2'\}\subset f'$, $\{u_1, w_1'\}\subset j'$, $\{u_1, w_2'\}\subset k'$, and $\{w_1', w_2'\}\subset l'$. Let $\mathfrak{f}(e')=\{u_3, w_1'\}$, $\mathfrak{f}(f')=\{u_3, w_2'\}$, $\mathfrak{f}(j')=\{u_1, w_1'\}$, $\mathfrak{f}(k')=\{u_1, w_2'\}$, and $\mathfrak{f}(l')=\{w_1', w_2'\}$. By symmetry, we may further assume that for any two distinct elements $x,y\in\{a,b,c,e',f',j',k',l',i\}$, $\mathfrak{f}(x)\cap \mathfrak{f}(y) = x\cap y$.

Let $\mathcal{V}'$ be the set of vertices of elements of $\{a,b,c,e,f,j,k,l,i\}$; observe that $|\mathcal{V}'|=22$.

Suppose $w_1', w_2'\in \mathcal{V}'\backslash\mathcal{V}$. Note that the edges $\{w_1', u_3\}$ and $\{w_2', u_3\}$ are not covered by any element of $\{e, f, j, k, l\}$ since no element of $\{e,f,j,k,l\}$ contains $u_3$. Because, $w_1', w_2'\notin\mathcal{V}$, $\{w_1', u_3\}$ and $\{w_2', u_3\}$ are not covered by any element of $\{a,b, c, i\}$. Since $\{w_1', u_3\}$ is covered by $e'$ and $\{w_2', u_3\}$ is covered by $f'$, we have that $e', f'\notin \{a,b,c,e,f,j,k,l,i\}$. Then, if $\kappa$ is the subgraph of $h$ induced by the vertices of $a$, $b$, $c$, $e$, $e'$, $f$, $f'$, $j$, $k$, $l$, and $i$, $e(\kappa)\geq 11$ and $v(\kappa)\leq 26$ so
\[
\frac{e(\kappa)}{v(\kappa)}\geq \frac{11}{26}>\frac{7}{17}.
\]

Assume that $|\{w_1', w_2'\}\cap\mathcal{V}'\backslash\mathcal{V}| = 1$. Without loss of generality, assume that $w_1'\in \mathcal{V}'\backslash\mathcal{V}$ and $w_2'\notin \mathcal{V}'$. Using the argument from the previous case gives that $\{w_1', u_3\}$ is not covered by any element of $\{a,b,c,e,f,j,k,l,i\}$, so $e'\notin \{a,b,c,e,f,j,k,l,i\}$. Because $w_2'\notin \mathcal{V}'$, any element of $\{e',f',j',k',l'\}$ that contains $w_2'$ is not an element of $\{a,b,c,e,f,j,k,l,i\}$. Thus, $f', k', l'\notin \{a,b,c,e,f,j,k,l,i\}$. Then, if $\kappa$ is the subgraph of $h$ induced by the vertices of $a$, $b$, $c$, $e$, $e'$, $f$, $f'$, $j$, $k$, $k'$, $l$, $l'$, and $i$, $e(\kappa)\geq 13$ and $v(\kappa)\leq 31$ so
\[
\frac{e(\kappa)}{v(\kappa)}\geq \frac{13}{31}>\frac{7}{17}.
\]

Assume that $w_1', w_2'\notin \mathcal{V'}$. Then, $e',f', j', k', l'\notin \{a,b,c,e,f,j,k,l,i\}$. Thus, if $\kappa$ is the subgraph of $h$ induced by the vertices of $a$, $b$, $c$, $e$, $e'$, $f$, $f'$, $j$, $j'$, $k$, $k'$, $l$, $l'$, and $i$, $e(\kappa)\geq 14$ and $v(\kappa)\leq 34$ so
\[
\frac{e(\kappa)}{v(\kappa)}\geq \frac{14}{34}=\frac{7}{17}.
\] 
This proves that $m(h)\geq \frac{7}{17}$.
\end{proof}

\begin{theorem}
\label{thm:general}
Suppose $d\geq 5$. Then, $d-1-\frac{1}{m(h)}\geq \frac{2d-4}{2d-1}$.
\end{theorem}
\begin{proof}
For the sake of contradiction, assume that $d-1-\frac{1}{m(h)}<\frac{2d-4}{2d-1}$. Let $\gamma=\frac{2d-4}{2d-1}$. We must have that $d-1-\frac{1}{m(h)}<\gamma$. Because $\gamma \geq \frac{d-1}{d+1}$, \Cref{ratio} gives that
\[
\frac{(d-1)|E_h|-|U|+|\mathcal{P}|}{|E_h|+|\mathcal{P}|}<\gamma.
\]
This is equivalent to
\[
(d-1-\gamma)|E_h| + (1-\gamma)|\mathcal{P}| < |U|.
\]
Observe that
\[
|E_h| \geq \frac{1}{2d}(|V_h| + |V_g| + \sum_{v\in V_I} (d_h(v)+d_g(v)))
\]
and from \pref{sum2},
\[
|\mathcal{P}| \geq \frac{1}{2}((d-1)(|V_h| + |V_g|) + \sum_{v\in V_I} (d_h^*(v)+d_g^*(v))).
\]
Hence,
\begin{align*}
&\frac{d-1-\gamma}{2d}(|V_h| + |V_g| + \sum_{v\in V_I} (d_h(v)+d_g(v))) \\
& + \frac{1-\gamma}{2}((d-1)(|V_h| + |V_g|) + \sum_{v\in V_I} (d_h^*(v)+d_g^*(v))) < |U|.
\end{align*}
For $v\in V_I$, let
\[
f(v) = \frac{d-1-\gamma}{2d}(d_h(v)+d_g(v)) + \frac{1-\gamma}{2}(d_h^*(v)+d_g^*(v)).
\]
We have that
\begin{equation}
\label{firstineq}
\left(\frac{d-1-\gamma}{2d} + \frac{(d-1)(1-\gamma)}{2}\right)(|V_h|+|V_g|) +\sum_{v\in V_I} f(v) < |U|.
\end{equation}

\begin{claim}
\label{ineqdeg}
Suppose $v\in V_I$. If $n_h(v)<k(v)$ then $d_g(v)>1$ and if $n_g(v)<k(v)$ then $d_h(v)>1$.
\end{claim}

\begin{proof}
Suppose $v\in V_I$, $i_g=i_g(v)$, and $i_h=i_h(v)$. Assume that $n_h(v)<k(v)$; the case $n_g(v)<k(v)$ follows similarly. Then, there exists $w\in i_h\backslash i_g$ such that $\{v, w\}\in \mathcal{E}_h$ and $\{v, w\}\notin \mathcal{E}_h\backslash \mathcal{E}_g$. Assume that $w$ satisfies this condition. Then, there exists $i\in E_g$ such that $i\not=i_g$ and $\{v, w\}\subset i$, so $d_g(v) \geq 2$.
\end{proof}

Let $\mathcal{V}$ be the set of $v\in V_I$ such that:
\begin{enumerate}
\item $d_h(v)=d_g(v)=1$.
\item $k(v) = n_h(v) = n_g(v) = 1$.
\end{enumerate}
Observe that if $v\in\mathcal{V}$, $i_h(v)$ and $i_g(v)$ are deterministic since $d_h(v)=1$ and $d_g(v)=1$, respectively. Furthermore, the formulation of $\mathcal{V}$ corresponds to the ambiguous graph $G_{a,d}$ from \cite{reconstruct}. For the sake of contradiction, assume the condition \condition{cond:noambiguous} that there does not exist $i_h\in E_h$ and $i_g\in E_g$ such that $|i_h\backslash i_g|=1$ and $i_h\cap i_g\subset \mathcal{V}$.

Suppose $v\in \mathcal{V}$. Since $d_h(v)=d_g(v)=1$, $d_h^*(v)\geq n_h(v) = 1$, and $d_g^*(v)\geq n_g(v)=1$,
\begin{equation}
\label{fval}
f(v) \geq \frac{d-1-\gamma}{d} + (1-\gamma).
\end{equation}
Let $C = \frac{d-1-\gamma}{d} + (1-\gamma)$. 

Suppose $v\in V_I\backslash \mathcal{V}$. If $n_h(v)<k(v)$ or $n_g(v)<k(v)$, then using \Cref{ineqdeg} gives that
\[
f(v) \geq C + \frac{d-1-\gamma}{2d} - \frac{1-\gamma}{2}.
\]
If $k(v)=n_h(v)=n_g(v)>1$, then,
\[
f(v) \geq C + (1-\gamma).
\]
If $d_h(v)>1$ or $d_g(v)>1$ and $k(v)=n_h(v)=n_g(v)=1$, then
\[
f(v) \geq C + \frac{d-1-\gamma}{2d}.
\]

Suppose $\beta\in\mathbb{R}$ such that $\frac{d-1-\gamma}{d} \geq (1+2\beta)(1-\gamma)$ and $0\leq\beta\leq 1$. Then, since $\frac{d-1-\gamma}{2d} - \frac{1-\gamma}{2}\geq \beta(1-\gamma)$ and $\beta\leq 1$, 
\begin{equation}
\label{ineqf}
f(v) \geq C+\beta(1-\gamma)=\frac{d-1-\gamma}{d}+(1+\beta)(1-\gamma).
\end{equation}
Note that $\frac{d-1-\gamma}{d} \geq (1+2\beta)(1-\gamma)$ if and only if
\begin{equation}
\label{betaineq}
\beta \leq \frac{1}{2}\left(\frac{d-1-\gamma}{d(1-\gamma)}-1\right) = \frac{1}{2}\left(\frac{2d}{3}-\frac{8}{3}+\frac{5}{3d}\right).
\end{equation}

Let $\mathcal{A}$ be the set of $v\in V_I$ such that $d_h(v)+d_g(v)\geq 3$ and there exists $i_h\in E_h$ and $i_g\in E_g$ such that $v\in i_h\cap i_g$ and $|i_h\backslash i_g|=1$. Let $\mathcal{S}$ be the set of $v\in \mathcal{A}$ such that $d_h(v)+d_g(v)=3$ and $\mathcal{T}$ be the set of $v\in\mathcal{A}$ such that $d_h(v)+d_g(v)>3$.

Let $\mathcal{V}^*$ be the set of $v\in \mathcal{V}$ such that $i_h(v)\cap i_g(v)\cap \mathcal{S}$ is nonempty. Furthermore, let $\mathcal{V}^{**}$ be the set of $v\in \mathcal{V}$ such that $i_h(v)\cap i_g(v)\cap \mathcal{T}$ is nonempty. 

Suppose $v\in \mathcal{V}$. By \refcondition{cond:noambiguous}, there exists $u\in i_h(v)\cap i_g(v)$ such that $u\notin \mathcal{V}$. Assume that $u$ satisfies this condition. We have that there exists $i\in E_h\cup E_g\backslash\{i_h(v), i_g(v)\}$ such that $u\in i$ because $u\notin\mathcal{V}$, so $d_g(u)+d_h(u)\geq 3$ and $u\in\mathcal{A}$. This implies that
$\mathcal{V}^*\cup \mathcal{V}^{**} = \mathcal{V}$.

\begin{claim}
\label{ineqS}
$(d-1)(\sum_{v\in \mathcal{S}} d_h(v) + d_g(v)) \geq 3|\mathcal{V}^*|$.
\end{claim}
\begin{proof}
Suppose $v\in\mathcal{S}$. Without loss of generality, assume that $d_h(v)=2$ and $d_g(v)=1$. Suppose $i_h^1, i_h^2\in E_h$, $i_h^1\not=i_h^2$, $i_g\in E_g$, $|i_h^1\backslash i_g|=1$, and $v\in i_h^1\cap i_h^2\cap i_g$. Let 
\[
s(v)=((i_h^1\cap i_g)\cup (i_h^2\cap i_g))\cap\mathcal{V}^*.
\]
Note that $s(v)\subset i_g\backslash\{v\}$ so $|s(v)|\leq d-1$. Hence,
\[
(d-1)(d_h(v)+d_g(v)) = 3(d-1) \geq 3|s(v)|.
\]
This implies that
\[
(d-1)\sum_{v\in\mathcal{S}} (d_h(v) + d_g(v)) \geq 3\sum_{v\in\mathcal{S}} |s(v)|.
\]

Suppose $u\in\mathcal{V}^*$ and suppose $v\in i_h(u)\cap i_g(u)\cap \mathcal{S}$; note that $i_h(u)\cap i_g(u)\cap \mathcal{S}$ is nonempty by the definition of $\mathcal{V}^*$. We have that $u\in s(v)$. Hence, $\mathcal{V}^*\subset \bigcup_{v\in\mathcal{S}} s(v)$ so $|\mathcal{V}^*|\leq \sum_{v\in\mathcal{S}}|s(v)|$. This finishes the proof.
\end{proof}

\begin{claim}
\label{ineqT}
$(d-1)(\sum_{v\in\mathcal{T}} d_h(v) + d_g(v)) \geq 2|\mathcal{V}^{**}|$.
\end{claim}
\begin{proof}
Suppose $v\in\mathcal{T}$. Define
\[
t(v)=\left(\bigcup_{\substack{i_h\in E_h, i_g\in E_g,\\v\in i_h\cap i_g, \\
|i_h\backslash i_g|=1}} (i_h\cap i_g)\right)\cap \mathcal{V}^{**}.
\]
Suppose $i_h\in E_h$ and $v\in i_h$. Note that $i_h\cap t(v)\subset i_h\backslash\{v\}$ so $|i_h\cap t(v)|\leq d-1$. We have that
\[
|t(v)| \leq \sum_{i_h\in E_h: v\in i_h} |i_h\cap t(v)| \leq (d-1)d_h(v).
\]
Similarly, $|t(v)| \leq (d-1)d_g(v)$. Hence,
\[
(d-1)(d_h(v)+d_g(v))\geq 2|t(v)|.
\]
We then have that
\[
(d-1)\sum_{v\in\mathcal{T}}(d_h(v)+d_g(v))\geq 2\sum_{v\in\mathcal{T}} |t(v)|.
\]

Suppose $u\in \mathcal{V}^{**}$. We have that $i_h(u)\cap i_g(u)\cap \mathcal{T}$ is nonempty by the definition of $\mathcal{V}^{**}$. If $v\in i_h(u)\cap i_g(u)\cap \mathcal{T}$ then $u\in t(v)$. Hence, $\mathcal{V}^{**}\subset\bigcup_{v\in\mathcal{T}} t(v)$, so $|\mathcal{V}^{**}|\leq \sum_{v\in\mathcal{T}} |t(v)|$. We are done.
\end{proof}

Observe that \pref{firstineq} implies that
\[
(-\frac{d-1-\gamma}{2d} + \frac{(d-1)(1-\gamma)}{2})(|V_h|+|V_g|) + \sum_{v\in V_I} (f(v) - \frac{d-1-\gamma}{d})<\frac{1+\gamma}{d}|U|.
\]
Note that $-\frac{d-1-\gamma}{2d} + \frac{(d-1)(1-\gamma)}{2}\geq \frac{1+\gamma}{d}$, so
\[
\frac{1+\gamma}{d}(|V_h|+|V_g|) + \sum_{v\in V_I} (f(v) - \frac{d-1-\gamma}{d}) < \frac{1+\gamma}{d}|U|.
\]
This implies that
\begin{equation}
\label{mainineq}
\sum_{v\in V_I} (f(v) - \frac{d-1-\gamma}{d}) < \frac{1+\gamma}{d}|V_I|.
\end{equation}

If $v\in\mathcal{V}$, then \pref{fval} gives that
\[
f(v)-\frac{d-1-\gamma}{d} \geq 1-\gamma.
\]
Hence, using \pref{ineqf} gives that
\begin{equation}
\label{sum}
\begin{split}
&\sum_{v\in V_I} (f(v)-\frac{d-1-\gamma}{d}) \\ & \geq (1-\gamma)|\mathcal{V}| + \sum_{v\in\mathcal{S}\cup\mathcal{T}} (f(v)-\frac{d-1-\gamma}{d}) + \sum_{v\in V_I\backslash (\mathcal{V}\cup\mathcal{S}\cup\mathcal{T})} (f(v)-\frac{d-1-\gamma}{d}) \\ &\geq (1-\gamma)|\mathcal{V}| + \sum_{v\in \mathcal{S}\cup\mathcal{T}} (f(v)-\frac{d-1-\gamma}{d}) + (|V_I|-|\mathcal{V}|-|\mathcal{S}\cup \mathcal{T}|)(1+\beta)(1-\gamma) \\
& = (1-\gamma)|\mathcal{V}| + \sum_{v\in \mathcal{S}\cup\mathcal{T}} (f(v)-\frac{d-1-\gamma}{d}-(1+\beta)(1-\gamma)) + (|V_I|-|\mathcal{V}|)(1+\beta)(1-\gamma).
\end{split}
\end{equation}

Suppose $v\in\mathcal{S}$. Then,
\[
\frac{d-1-\gamma}{2d}(d_h(v)+d_g(v)) \geq 3\frac{d-1-\gamma}{2d}.
\]
Without loss of generality, assume that $d_h(v)=2$ and $d_g(v)=1$. Suppose $i_h\in E_h$, $i_g\in E_g$, and $v\in i_h\cap i_g$. Observe that for all $w\in i_h\backslash i_g$, $\{w,v\}\in \mathcal{E}_h\backslash\mathcal{E}_g$ since $d_g(v)=1$. Since $|i_h\backslash i_g|\geq 1$, $d_h^*(v)\geq 1$. Similarly, if $d_h(v)=1$ and $d_g(v)=2$, $d_g^*(v)\geq 1$. It follows that 
\[
d_h^*(v)+d_g^*(v)\geq 1.
\]
Hence,
\[
\sum_{v\in \mathcal{S}} (f(v) - \frac{d-1-\gamma}{d}-(1+\beta)(1-\gamma)) \geq |\mathcal{S}|\cdot (\frac{d-1-\gamma}{2d} - (\frac{1}{2}+\beta)(1-\gamma)).
\]
Furthermore, using \Cref{ineqS} gives that
\begin{align*}
& \sum_{v\in \mathcal{S}} (f(v) - \frac{d-1-\gamma}{d}-(1+\beta)(1-\gamma)) \\
& \geq \frac{d-1-\gamma}{2d}\sum_{v\in\mathcal{S}} (d_h(v)+d_g(v)) - (\frac{d-1-\gamma}{d}+(\frac{1}{2}+\beta)(1-\gamma))|\mathcal{S}| \\
& \geq \frac{d-1-\gamma}{2d}\cdot \frac{3|\mathcal{V}^*|}{d-1} - (\frac{d-1-\gamma}{d}+(\frac{1}{2}+\beta)(1-\gamma))|\mathcal{S}|.
\end{align*}
Next, we consider a linear combination of these two inequalities. If 
\[
c^* =\frac{\frac{d-1-\gamma}{2d}-(\frac{1}{2}+\beta)(1-\gamma)}{\left(\frac{d-1-\gamma}{2d}-(\frac{1}{2}+\beta)(1-\gamma)\right) + \left(\frac{d-1-\gamma}{d}+(\frac{1}{2}+\beta)(1-\gamma)\right)} =\frac{\frac{d-1-\gamma}{2d}-(\frac{1}{2}+\beta)(1-\gamma)}{3\frac{d-1-\gamma}{2d}},
\]
then
\begin{align*}
&\sum_{v\in \mathcal{S}} (f(v) - \frac{d-1-\gamma}{d}-(1+\beta)(1-\gamma)) 
\\
& \geq (1-c^*)|\mathcal{S}|\cdot (\frac{d-1-\gamma}{2d} - (\frac{1}{2}+\beta)(1-\gamma)) \\ 
& + c^*(\frac{d-1-\gamma}{2d}\cdot \frac{3|\mathcal{V}^*|}{d-1} - (\frac{d-1-\gamma}{d}+(\frac{1}{2}+\beta)(1-\gamma))|\mathcal{S}|) \\
&= (\frac{d-1-\gamma}{2d(d-1)}-\frac{(\frac{1}{2}+\beta)(1-\gamma)}{d-1})|\mathcal{V}^*|.
\end{align*}
Observe that $c^*\geq 0$ because $\frac{d-1-\gamma}{d} \geq (1+2\beta)(1-\gamma)$.

Suppose $v\in\mathcal{T}$. Then,
\[
f(v) \geq 2\frac{d-1-\gamma}{d}. 
\]
Hence,
\[
\sum_{v\in \mathcal{T}} (f(v) - \frac{d-1-\gamma}{d}-(1+\beta)(1-\gamma)) \geq |\mathcal{T}|\cdot (\frac{d-1-\gamma}{d} - (1+\beta)(1-\gamma)).
\]
Furthermore, using \Cref{ineqT} gives that
\begin{align*}
& \sum_{v\in \mathcal{T}} (f(v) - \frac{d-1-\gamma}{d}-(1+\beta)(1-\gamma)) \\ & \geq \frac{d-1-\gamma}{2d}\sum_{v\in\mathcal{T}} (d_h(v)+d_g(v)) - (\frac{d-1-\gamma}{d}+(1+\beta)(1-\gamma))|\mathcal{T}| \\
& \geq \frac{d-1-\gamma}{d}\cdot \frac{|\mathcal{V}^{**}|}{d-1} - (\frac{d-1-\gamma}{d}+(1+\beta)(1-\gamma))|\mathcal{T}|.
\end{align*}
Next, we consider a linear combination of these two inequalities. If
\[
c^{**} =\frac{\frac{d-1-\gamma}{d}-(1+\beta)(1-\gamma)}{\left(\frac{d-1-\gamma}{d}-(1+\beta)(1-\gamma)\right) + \left(\frac{d-1-\gamma}{d}+(1+\beta)(1-\gamma)\right)} =\frac{\frac{d-1-\gamma}{d}-(1+\beta)(1-\gamma)}{2\frac{d-1-\gamma}{d}},
\]
then
\begin{align*}
&\sum_{v\in \mathcal{S}} (f(v) - \frac{d-1-\gamma}{d}-(1+\beta)(1-\gamma)) \geq (1-c^{**})|\mathcal{T}|\cdot (\frac{d-1-\gamma}{d} - (1+\beta)(1-\gamma)) \\ 
&+ c^{**}(\frac{d-1-\gamma}{d}\cdot \frac{|\mathcal{V}^{**}|}{d-1} - (\frac{d-1-\gamma}{d}+(1+\beta)(1-\gamma))|\mathcal{T}|) \\
&= (\frac{d-1-\gamma}{2d(d-1)}-\frac{(1+\beta)(1-\gamma)}{2(d-1)})|\mathcal{V}^{**}|.
\end{align*}
Similarly, $c^{**}\geq 0$ because $\frac{d-1-\gamma}{d} \geq (1+\beta)(1-\gamma)$.

Let $\beta^*=\frac{1}{2}+\beta$. Because $\frac{d-1-\gamma}{d} \geq (1+2\beta)(1-\gamma)$, $\frac{d-1-\gamma}{d} \geq 2\beta^*(1-\gamma)$. Furthermore, $\frac{1}{2}+\beta\geq \frac{1+\beta}{2}$ since $\beta\geq 0$. It follows that 
\begin{align*}
\sum_{v\in \mathcal{S}\cup\mathcal{T}} (f(v)-\frac{d-1-\gamma}{d}-(1+\beta)(1-\gamma)) \geq & (\frac{d-1-\gamma}{2d(d-1)}-(\frac{1}{2}+\beta)\frac{1-\gamma}{d-1})|\mathcal{V}^*|\\ & +(\frac{d-1-\gamma}{2d(d-1)}-\frac{(1+\beta)(1-\gamma)}{2(d-1)})|\mathcal{V}^{**}| \\
\geq & (\frac{d-1-\gamma}{2d(d-1)}-\beta^*\frac{1-\gamma}{d-1})(|\mathcal{V}^*|+|\mathcal{V}^{**}|) \\
\geq & (\frac{d-1-\gamma}{2d(d-1)}-\beta^*\frac{1-\gamma}{d-1})|\mathcal{V}|.
\end{align*}
 Using this inequality and \pref{sum} gives that
\[
\sum_{v\in V_I} (f(v)-\frac{d-1-\gamma}{d}) 
\geq (1-\gamma + \frac{d-1-\gamma}{2d(d-1)}-\beta^*\frac{1-\gamma}{d-1})|\mathcal{V}| + (|V_I|-|\mathcal{V}|)(1+\beta)(1-\gamma).
\]
Afterwards, \pref{mainineq} implies that
\[
(1-\gamma + \frac{d-1-\gamma}{2d(d-1)}-\beta^*\frac{1-\gamma}{d-1})|\mathcal{V}| + (|V_I|-|\mathcal{V}|)(1+\beta)(1-\gamma) < \frac{1+\gamma}{d}|V_I|.
\]

From \pref{betaineq}, we can set $\beta$ to be $\frac{1}{3}$. Let $\beta=\frac{1}{3}$ and $\beta^*=\frac{1}{2}+\beta=\frac{5}{6}$. Furthermore, let
\[
\ell(x)=(1-\gamma + \frac{d-1-\gamma}{2d(d-1)}-\beta^*\frac{1-\gamma}{d-1})x + (|V_I|-x)(1+\beta)(1-\gamma).
\]
We have that $\ell(|\mathcal{V}|)< \frac{1+\gamma}{d}|V_I|$. Note that $\ell$ is a linear function and $0\leq |\mathcal{V}|\leq |V_I|$. Then, $\ell(|\mathcal{V}|)<\frac{1+\gamma}{d}|V_I|$ implies that $\min(\ell(0), \ell(|V_I|))<\frac{1+\gamma}{d}|V_I|$. 

First, observe that
\[
\ell(0)- \frac{1+\gamma}{d}|V_I|= ((1+\beta)(1-\gamma)-\frac{1+\gamma}{d})|V_I|=\frac{1}{2d-1}(3(1+\beta)-\frac{4d-5}{d})|V_I|.
\]
Since
\[
3(1+\beta)-\frac{4d-5}{d} = \frac{5}{d} > 0,
\]
$\ell(0)\geq\frac{1+\gamma}{d}|V_I|$.

Furthermore,
\[
\ell(|V_I|) = (1-\gamma + \frac{d-1-\gamma}{2d(d-1)}-\beta^*\frac{1-\gamma}{d-1})|V_I|.
\]
We have that
\[
1-\gamma + \frac{d-1-\gamma}{2d(d-1)}-\beta^*\frac{1-\gamma}{d-1} - \frac{1+\gamma}{d} = \frac{7d-5-6\beta^*d}{2d(d-1)(2d-1)}.
\]
Since $\beta^* = \frac{5}{6}$,
\[
7d-5-6d\beta^* = 2d - 5 > 0.
\]
Thus, $\ell(|V_I|) \geq \frac{1+\gamma}{d}|V_I|$. This is a contradiction to $\min(\ell(0), \ell(|V_I|))<\frac{1+\gamma}{d}|V_I|$.

Therefore, there exists $i_h\in E_h$ and $i_g\in E_g$ such that $|i_h\backslash i_g|=1$ and $i_h\cap i_g\subset\mathcal{V}$. Assume that $i_h$ and $i_g$ satisfy this condition. Suppose $i_h = (i_h\cap i_g)\cup \{v_h\}$ and $i_g=(i_h\cap i_g)\cup \{v_g\}$. Because $i_h\cap i_g\subset \mathcal{V}$, $d_h(v)=d_g(v)=1$ for all $v\in i_h\cap i_g$, which implies that $\{v_h, v\}\in\mathcal{E}_h\backslash \mathcal{E}_g$ and $\{v_g, v\}\in\mathcal{E}_g\backslash\mathcal{E}_h$ for all $v\in i_h\cap i_g$. 

Let $E_h'=\{i_h\}$, $E_g'=\{i_g\}$, $U'=i_h\cup i_g$, and $\mathcal{P}'=\{\{v_h, v\}: v\in i_h\cap i_g\}\cup \{\{v_g, v\}: v\in i_h\cap i_g\}$. Let $I'$ be the set of $i\in I$ such that there exists $e\in\mathcal{P}'$ such that $e\subset i'$. Because $\mathcal{P}'\subset \mathcal{P}$, each edge in $\mathcal{P}'$ is contained in some element of $I'$. Let $V'$ be the set of vertices in $i_h$, $i_g$, or some hyperedge in $I'$. Let $h'$ be the subgraph of $h$ induced by $V'$. Furthermore, let $h''$ be the graph with vertex set $V'$ and edge set $E_h'\cup I'$. We have that
\[
d-1-\frac{1}{m(h)} \geq d-1-\frac{1}{\alpha(h')} \geq d-1-\frac{1}{\alpha(h'')}
\]
so it suffices to prove that
\[
d-1-\frac{1}{\alpha(h'')}\geq \frac{2d-4}{2d-1}.
\]
Using \Cref{ratio} gives that to prove that $d-1-\frac{1}{\alpha(h'')} \geq \frac{2d-4}{2d-1}$, it suffices to prove that 
\[
\frac{(d-1)|E_h'|-|U'| + |\mathcal{P}'|}{|E_h'|+|\mathcal{P}'|} \geq \frac{2d-4}{2d-1}.
\]
Observe that
\[
\frac{(d-1)|E_h'|-|U'| + |\mathcal{P}'|}{|E_h'|+|\mathcal{P}'|} = \frac{d-1 - (d+1) + 2(d-1)}{1 + 2(d-1)} = \frac{2d-4}{2d-1},
\]
which finishes the proof.
\end{proof}

\subsection{Lower bounds of \texorpdfstring{$m(h)$}{} for a nonminimal weighted preimage \texorpdfstring{$h$}{}}
\label{subsec:weightedambiguous}

For the weighted random model, we consider when two hypergraphs have the same weighted projections, which would imply that they have the same projection. First continue using the notation of the previous section. Additionally, add the condition that $\text{Proj}_W(h)=\text{Proj}_W(g)$. Furthermore, observe that $e(h)=e(g)$ rather than $e(h)\geq e(g)$, which follows from $\text{Proj}_W(h)=\text{Proj}_W(g)$. The goal is to prove the following result:

\begin{theorem}
\label{thm:weightedratio}
Suppose $d\geq 3$. Then $d-1-\frac{1}{m(h)} \geq \frac{d}{2} - 1 \geq \frac{d-1}{d+1}$.
\end{theorem}

First, observe that we can remove all hyperedges in both $h$ and $g$; afterwards, the condition $\text{Proj}_W(h)= \text{Proj}_W(g)$ will still be satisfied, and $d-1-\frac{1}{m(h)}$ will be decreased. Hence, assume that $E(h)$ and $E(g)$ are disjoint. Since $h\not=g$, both $E(h)$ and $E(g)$ are nonempty.

\begin{lemma}
\label{lemma:vertexdegree}
Suppose $v\in V$. Then, $d_h(v)\not=1$.
\end{lemma}

\begin{proof}
For the sake of contradiction, assume that $d_h(v) = 1$. Suppose $i$ is the hyperedge of $h$ that contains $v$. Then, the weight of $\{v, u\}$ for all $u\in i\backslash\{v\}$ in $\text{Proj}_W(h)$ equals one. Since $\text{Proj}_W(g)=\text{Proj}_W(h)$, each edge $\{v, u\}$ for $u\in i\backslash\{v\}$ is contained in a hyperedge $i_u$ of $g$. If all of the $i_u$ are equal, then they must all equal $i$, which is a contradiction to $E(h)$ and $E(g)$ being disjoint. Therefore, some two of the $i_u$ are distinct, which implies that $d_g(v)\geq 2$. However, $\text{Proj}_W(h)=\text{Proj}_W(g)$ implies that $d_g(v)=d_h(v)=1$, which is a contradiction.
\end{proof}

\begin{proof}[Proof of \Cref{thm:weightedratio}]
Let $V'$ be the set of $v\in V$ such that $d_h(v) \geq 1$. Let $h'$ be the subgraph of $h$ induced by $V'$. Using \Cref{lemma:vertexdegree} implies that
\[
de(h') = de(h) = \sum_{v\in V'} d_h(v) \geq 2v(h'), 
\]
so $m(h) \geq \alpha(h') \geq \frac{2}{d} \Rightarrow d-1-\frac{1}{m(h)} \geq \frac{d}{2} - 1$. It is straightforward to verify that $\frac{d}{2}-1\geq \frac{d-1}{d+1}$.
\end{proof}

Observe that the threshold in \Cref{thm:weightedratio} is greater than the threshold in \Cref{thm:ambiguous_threshold}. The reason for this is that the graph $G_{a,d}$ is not ambiguous under the weighted projection. In fact, we show that the threshold $\frac{d}{2}-1$ in \Cref{thm:weightedratio} is achievable. 

Define the hypergraph $H=(V,E)$ as follows:
\begin{itemize}
\item $V=S_1\sqcup S_2\sqcup \{w_1, w_2, w_3, w_4\}$, where $|S_1|=|S_2|=d-2$.
\item $E$ consists of $S_1\cup \{w_1, w_2\}$, $S_1\cup \{w_3, w_4\}$, $S_2\cup \{w_2, w_3\}$, and $S_2\cup \{w_4, w_1\}$.
\end{itemize}
Then, define $G_{a,d}^w:= \text{Proj}(H)$.

Suppose $H'$ has the same vertex set as $H$ and edge set $S_2\cup \{w_1, w_2\}$, $S_2\cup \{w_3, w_4\}$, $S_1\cup \{w_2, w_3\}$, and $S_1\cup \{w_4, w_1\}$. Then, $H$ and $H'$ are two distinct minimal preimages of $G_{a,d}^w$, so $G_{a,d}^w$ is a weighted-ambiguous graph. Since each vertex of $H$ is contained in two hyperedges, we also observe that $H$ achieves the lower bound in \Cref{thm:weightedratio}.

\section{Entropy of the projected graph}
\label{sec:entropy}

For simplicity, assume that the logarithms in entropy are base $e$. In this section we mainly analyze the entropy of $\text{Proj}(\mathcal{H})$ and hence the conditional entropy $H(\mathcal{H}|\text{Proj}(\mathcal{H}))=H(\mathcal{H})-H(\text{Proj}(\mathcal{H}))$.

\begin{lemma}
\label{approx}
$e^{-p}\geq 1-p$ for all $p\geq 0$ and $e^{-2p} \leq 1-p$ if $p$ is sufficiently small.
\end{lemma}

\begin{lemma}
\label{lemma:condentropy}
Suppose $\delta>\frac{d-1}{d+1}$. Then there exists $C>0$ such that $C$ does not depend on $c$ and
$H(\mathcal{H}|\pg)\geq (C+o_n(1))H(\mathcal{H})$.
\end{lemma}

\begin{proof}
Since $H(\mathcal{H}|\pg)=H(\mathcal{H})-H(\pg)$, it suffices to prove that $H(\mathcal{H})-H(\pg)=\Omega_n(H(\mathcal{H}))$ if $\delta>\frac{d-1}{d+1}$.

First observe that
\begin{equation}
\label{entropyexacth}
H(\mathcal{H}) = -\binom{n}{d}H_B(p)= c\frac{d-1-\delta}{d!}\log(n)n^{1+\delta} + o_n(\log(n)n^{1+\delta}).
\end{equation}

For all $i,j\in[n]$ such that $i\not=j$, let $X_{\{i,j\}}=\1\{\{i,j\}\in E(\pg)\}$. We have that
\begin{equation}
\label{unionbound}
H(\pg) = H(X_{\{i, j\}}: 1\leq i<j\leq n) \leq \sum_{1\leq i<j\leq n} H(X_{\{i, j\}})=\binom{n}{2}H_B(1-(1-p)^{\binom{n-2}{d-2}}).
\end{equation}
Bernoulli's inequality implies that
\[
1-(1-p)^{\binom{n-2}{d-2}} \leq \binom{n-2}{d-2}p
\]
and for sufficiently large $n$ we have that $\binom{n-2}{d-2}p\leq \frac{1}{e}$. Hence, for sufficiently large $n$ we have that
\[
-(1 - (1-p)^{\binom{n-2}{d-2}})\log(1 - (1-p)^{\binom{n-2}{d-2}}) \leq -\binom{n-2}{d-2}p\log(\binom{n-2}{d-2}p)
\]
since $-x\log(x)$ increases as $x$ increases over $[0, \frac{1}{e}]$. Then,
\begin{align*}
H_B(1-(1-p)^{\binom{n-2}{d-2}}) & \leq -\binom{n-2}{d-2}p\log(\binom{n-2}{d-2}p)- (1-p)^{\binom{n-2}{d-2}}\log((1-p)^{\binom{n-2}{d-2}}) \\
& = (1-\delta)\binom{n-2}{d-2}p\log(n) + o_n(\log(n)n^{\delta-1}).
\end{align*}
Using this inequality and \pref{unionbound} gives that
\[
H(\pg) \leq  c\frac{1-\delta}{2(d-2)!}\log(n)n^{1+\delta} + o_n(\log(n)n^{\delta+1}).
\]
Using this inequality and \pref{entropyexacth} then gives that it suffices to prove that
\[
\frac{1-\delta}{2(d-2)!} < \frac{d-1-\delta}{d!};
\]
afterwards we can let $C=\frac{\frac{d-1-\delta}{d!}-\frac{1-\delta}{2(d-2)!}}{\frac{d-1-\delta}{d!}}$, which does not depend on $c$. This inequality is true because $\delta>\frac{d-1}{d+1}$, which finishes the proof.
\end{proof}

Firstly observe that $H(\mathcal{H}|\pg)=H(\mathcal{H})-H(\pg)$ and
\[
H(\pg)\leq H(\mathcal{H}) = \Theta_n(n^{1+\delta}\log(n)).
\]
We analyze $H(\mathcal{H}|\pg)$ and $H(\pg)$ in the next result.

\begin{theorem}
\label{thm:entropycompare}
$H(\pg)=\Theta_n(n^{1+\delta}\log(n))$. Furthermore 
\[
H(\mathcal{H}|\pg) = 
\begin{cases}
o_n(n^{1+\delta}) & \text{if } \delta<\frac{d-1}{d+1}, \\
O_n(n^{1+\delta}) & \text{if } \delta=\frac{d-1}{d+1}, \\
\Theta_n(n^{1+\delta}\log(n)) & \text{if } \delta>\frac{d-1}{d+1}.
\end{cases}
\]
\end{theorem}

\begin{proof}
We have that 
\begin{align*}
H(\pg)=I(\mathcal{H}; \pg) &\geq \binom{n}{d}I(\1\{[d]\in E(\mathcal{H})\}; \pg) \\
& = \binom{n}{d}\left(H_B(p) - \sum_{G\in\mathcal{G}} \Pr[\pg=G]H(\1\{[d]\in E(\mathcal{H})\}|\pg=G)\right) \\
& = \binom{n}{d}\left(H_B(p) - \sum_{G\in\mathcal{G}, \binom{[d]}{2}\subset E(G)} \Pr[\pg=G]H(\1\{[d]\in E(\mathcal{H})\}|\pg=G)\right).
\end{align*}
Observe that $q = \sum_{G\in\mathcal{G}, \binom{[d]}{2}\subset E(G)} \Pr[\pg=G]$. Then Jensen's inequality implies that
\[
\sum_{G\in\mathcal{G}, \binom{[d]}{2}\subset E(G)} \Pr[\pg=G]H(\1\{[d]\in E(\mathcal{H})\}|\pg=G) \leq qH_B\left(\frac{p}{q}\right).
\]
Therefore
\begin{equation}
\label{eq:entropycompare}
\begin{split}
H(\pg) & \geq \binom{n}{d}\left(H_B(p)-qH_B\left(\frac{p}{q}\right)\right) \\ 
& = \binom{n}{d}\left(-(1-p)\ln(1-p)-p\ln(q)+(q-p)\ln(1-\frac{p}{q})\right).
\end{split}
\end{equation}

Assume that $\delta<\frac{d-1}{d+1}$. Then \Cref{corr:asyqp} gives that $q=(1+o_n(1))p$ so $qH_B(\frac{p}{q})=o_n(p)$ and \pref{eq:entropycompare} gives that $H(\pg)=H(\mathcal{H})-o_n(n^{1+\delta})$

Next assume that $\delta=\frac{d-1}{d+1}$. We have that $q=(1+\Theta_n(1))p$ from \Cref{corr:asyqp} so $qH_B(\frac{p}{q})=O_n(p)$ and \pref{eq:entropycompare} gives that $H(\pg)=H(\mathcal{H})-O_n(n^{1+\delta})$.

Suppose $\delta>\frac{d-1}{d+1}$. Then $q=\omega_n(p)$ from \Cref{corr:asyqp} so if $n$ is sufficiently large then \Cref{approx} gives that
\[
(q-p)\ln(1-\frac{p}{q})\geq -2(q-p)\frac{p}{q} = \Omega_n(p)
\]
and using \pref{eq:entropycompare} gives that $H(\pg)=\Omega_n(n^{1+\delta}\log(n))$. Furthermore \Cref{lemma:condentropy} gives that $\\ H(\mathcal{H}|\pg)=\Theta_n(H(\mathcal{H}))$.
\end{proof}

\begin{corollary}
\label{corr:entropydiff}
The conditional entropy $H(\mathcal{H}|\pg)$ is $o_n(H(\mathcal{H}))$ if $\delta\leq\frac{d-1}{d+1}$.
\end{corollary}

\begin{proof}
This follows from \Cref{thm:entropycompare}.
\end{proof}

\begin{lemma}
\label{lemma:entropyorder}
If $n$ is sufficiently large then 
\[
H(\pg)\geq \left(1-\frac{\binom{n}{d-2}p}{1-\binom{n}{d-2}p}\right)\sum_{i=2}^{n-d+2} (i-1)H_B(1-(1-p)^{\binom{n-i}{d-2}})=\Omega(n^{1+\delta}\log(n)).
\]
\end{lemma}

\begin{proof}
Assume that $n$ is sufficiently large throughout the proof. Suppose the set of vertices in $\mathcal{H}$ is $[n]$. We construct a sequence of edges $S$. Initialize $S$ as the empty sequence. From $i=2$ to $i=n$, add the edges $\{i, j\}$ to $S$ from $j=1$ to $j=i-1$. Note that $S$ contains each edge between two vertices of $[n]$ exactly once. It follows that
\[
H(\pg) = \sum_{a=1}^{\binom{n}{2}} H(X_{S_a}|X_{S_b}, 1\leq b\leq a-1).
\]

Suppose $1\leq a\leq \binom{n}{2}$. Assume that $S_a = \{i, j\}$ where $1\leq j<i\leq n$. Observe that the random variables $\1\{h\in\mathcal{H}\}$ for $h\in \binom{([n]\backslash [i-1])\cup\{j\}}{d}$ are independent of the random variables $X_{S_b}$ for $1\leq b\leq a-1$. This is because no hyperedge in $\binom{([n]\backslash [i-1])\cup\{j\}}{d}$ contains $S_b$ as an edge for $1\leq b\leq a-1$. The probability that a hyperedge in $\binom{([n]\backslash [i-1])\cup\{j\}}{d}$  that contains $S_a$ is present in $\mathcal{H}$ is
\[
1 - (1-p)^{\binom{n-i}{d-2}}.
\]
Suppose $x_b\in \{0,1\}$ for $1\leq b\leq a-1$. Hence,
\[
\Pr[X_{S_a}=1|X_{S_b}=x_b, 1\leq b\leq a-1] \geq 1 - (1-p)^{\binom{n-i}{d-2}}.
\]
Note that $1 - (1-p)^{\binom{n-i}{d-2}}\leq \binom{n-i}{d-2}p\leq \frac{1}{2}$ if $n$ is sufficiently large. Thus, if 
\[
\Pr[X_{S_a}=1|X_{S_b}=x_b, 1\leq b\leq a-1] \leq (1-p)^{\binom{n-i}{d-2}}
\]
as well, then
\begin{equation}
\label{eq:binaryentropy}
\begin{split}
H(X_{S_a}|X_{S_b}=x_b, 1\leq b\leq a-1) & = H_B(\Pr[X_{S_a}=1|X_{S_b}=x_b, 1\leq b\leq a-1]) \\ &\geq H_B(1-(1-p)^{\binom{n-i}{d-2}}).
\end{split}
\end{equation}
We have that
\[
\E_{X_{S_b}, 1\leq b\leq a-1}[\Pr[X_{S_a}=1|X_{S_b}=x_b, 1\leq b\leq a-1]] = \Pr[X_{S_a}=1] \leq \binom{n}{d-2}p.
\]
Therefore, Markov's inequality implies that
\[
\Pr_{X_{S_b}, 1\leq b\leq a-1}[\Pr[X_{S_a}=1|X_{S_b}=x_b, 1\leq b\leq a-1]>(1-p)^{\binom{n-i}{d-2}}] \leq \frac{\binom{n}{d-2}p}{(1-p)^{\binom{n-i}{d-2}}}. 
\]
Observe that $\frac{\binom{n}{d-2}p}{(1-p)^{\binom{n-i}{d-2}}} < \frac{\binom{n}{d-2}p}{1-\binom{n}{d-2}p}$. Using \pref{eq:binaryentropy} then gives that
\begin{align*}
& H(X_{S_a}|X_{S_b}, 1\leq b\leq a-1) = \E_{X_{S_b}, 1\leq b\leq a-1}[H(X_{S_a}|X_{S_b}=x_b, 1\leq b\leq a-1)] \\
& \geq (1-\Pr_{X_{S_b}, 1\leq b\leq a-1}[\Pr[X_{S_a}=1|X_{S_b}=x_b, 1\leq b\leq a-1] \leq (1-p)^{\binom{n-i}{d-2}}]) \\ 
& \cdot H_B(1-(1-p)^{\binom{n-i}{d-2}}) \\
& \geq (1-\frac{\binom{n}{d-2}p}{1-\binom{n}{d-2}p})H_B(1-(1-p)^{\binom{n-i}{d-2}}).
\end{align*}
Summing this inequality from $a=1$ to $a=\binom{n}{2}$ gives 
\[
H(\pg)\geq (1-\frac{\binom{n}{d-2}p}{1-\binom{n}{d-2}p})\sum_{i=2}^{n-d+2} (i-1)H_B(1-(1-p)^{\binom{n-i}{d-2}}).
\]
Observe that we use the fact that if $i>n-d+2$ then $H_B(1-(1-p)^{\binom{n-i}{d-2}})=0$. It is straightforward to check that this lower bound on $H(\pg)$ is $\Omega_n(n^{1+\delta}\log(n))$.
\end{proof}

\subsection{Exact recovery in a general setting}

The following result is relevant to the impossibility of exact recovery in a more general setting. Although we do not include all of the details, it is applicable in the problem we consider in this paper.

\begin{theorem}
\label{thm:prob_entropy_general}
Suppose $\mathcal{X}$ is a finite set. Suppose $\mathcal{S}\subset\mathbb{Z}_{\geq 1}$ contains infinitely many elements. Suppose $n\in\mathcal{S}$. Suppose $f_n:\mathcal{X}^n\rightarrow \mathcal{Y}_n$ and $\mathcal{A}_n: \mathcal{Y}_n\rightarrow\mathcal{X}^n$ are functions such that for all $y\in\mathcal{Y}_n$, $f_n(\mathcal{A}_n(y))=y$. Suppose $X^n\in\mathcal{X}^n$ is a random variable such that $X_i$, $i\geq 1$ are independent and identically distributed with distribution $p_n$. Assume that $x^-\in\mathcal{X}$ and $\lim_{n\rightarrow\infty} p_n(x)=0$ and $\lim_{n\rightarrow\infty}np_n(x)=\infty$ for all $x\in\mathcal{X}^-:=\mathcal{X}\backslash\{x^-\}$, where the limits are over $\mathcal{S}$. Also assume that $\sup_{y\in\mathcal{Y}_n} \Pr[f_n(X^n)=y] \leq \frac{1}{e}$ if $n$ is sufficiently large. The probability that $\mathcal{A}_n(f_n(X^n))$ is not $X^n$ is at least $1-\frac{H(f_n(X^n))}{H(X^n)}+o_n(1)$, where the asymptotic term $o_n(1)$ does not depend on $f_n$.
\end{theorem}

\begin{proof}
Suppose $n\in\mathcal{S}$. Let $\mathcal{U}_n$ be the set of $x^n\in\mathcal{X}^n$ such that there does not exist $y\in\mathcal{Y}_n$ such that $\mathcal{A}_n(y)=x^n$. Observe that $\Pr[\mathcal{A}_n(X^n)\not=X^n] =\Pr[X^n\in\mathcal{U}_n]$.

First observe that
\begin{equation}
\label{eq:decomp}
\begin{split}
H(X^n) = & -\sum_{y\in\mathcal{Y}_n} \Pr[X^n=\mathcal{A}_n(y)]\log(\Pr[X^n=\mathcal{A}_n(y)]) \\ &- \sum_{x^n\in\mathcal{U}_n}\Pr[X^n=x^n]\log(\Pr[X^n=x^n]).
\end{split}
\end{equation}
Suppose $x^n\in\mathcal{U}_n$. Then,
\[
\log(\Pr[X^n=x^n]) = (n-\sum_{x\in\mathcal{X}^-} N_x(x^n))\log(p_n(x^-))+\sum_{x\in\mathcal{X}^-}N_x(x^n)\log(p_n(x)).
\]
Using \Cref{approx} and the fact that $\lim_{n\rightarrow\infty} p_n(x)=0$ for $x\in\mathcal{X}^-$ gives that
\[
0\geq (n-\sum_{x\in\mathcal{X}^-} N_x(x^n))\log(p_n(x^-)) \geq -2n\sum_{x\in\mathcal{X}^-} p_n(x)
\]
if $n$ is sufficiently large. Observe that
\[
H(X^n)\geq -n\sum_{x\in\mathcal{X}^-} p_n(x)\log(p_n(x))
\]
and $-\log(p_n(x))=\Omega_n(1)$ since $p_n(x)=o_n(1)$ for all $x\in\mathcal{X}^-$. Hence 
\begin{equation}
\label{eq:remainder}
(n-\sum_{x\in\mathcal{X}^-} N_x(x^n))\log(p_n(x^-)) = o_n(H(X^n))
\end{equation}
so
\[
\log(\Pr[X^n=x^n]) = \sum_{x\in\mathcal{X}^-} N_x(x^n)\log(p_n(x)) + o_n(H(X^n)).
\]
This implies that
\begin{equation}
\label{eq:entropyapprox}
\begin{split}
&- \sum_{x^n\in \mathcal{U}_n}\Pr[X^n=x^n]\log(\Pr[X^n=x^n]) \\
& = o_n(H(X^n)) - \sum_{x^n\in\mathcal{U}_n} \Pr[X^n=x^n]\sum_{x\in\mathcal{X}^-}N_x(x^n)\log(p_n(x)).
\end{split}
\end{equation}

Let $\mathcal{W}_n$ be the set of $x^n\in\mathcal{X}^n$ such that  
\[
N_x(x^n) \in [np_n(x)-(np_n(x))^{\frac{3}{4}}, np_n(x)+(np_n(x))^{\frac{3}{4}}]
\]
for all $x\in\mathcal{X}^-$. For all $x\in\mathcal{X}^-$ the random variable variable $|N_x(X^n)|$ has mean $np_n(x)$ and variance $np_n(x)(1-p_n(x))$. Because $p_n(x)=o_n(1)$ for all $x\in\mathcal{X}^-$, Chebyshev's inequality implies that $X^n \in \mathcal{W}_n$ with probability $1-o_n(1)$. Therefore, $\Pr[X^n\in \mathcal{U}_n\backslash\mathcal{W}_n]=o_n(1)$. Observe that
\begin{align*}
& \sum_{x^n\in\mathcal{U}_n} \Pr[X^n=x^n]\sum_{x\in\mathcal{X}^-}N_x(x^n)\log(p_n(x)) \\ 
& = \sum_{x^n\in\mathcal{U}_n\cap \mathcal{W}_n} \Pr[X^n=x^n]\sum_{x\in\mathcal{X}^-}N_x(x^n)\log(p_n(x)) \\
& + \sum_{x^n\in\mathcal{U}_n\backslash\mathcal{W}_n} \Pr[X^n=x^n]\sum_{x\in\mathcal{X}^-}N_x(x^n)\log(p_n(x)) \\
& = \Pr[X^n\in \mathcal{U}_n]\left(n\sum_{x\in\mathcal{X}^-} p_n(x)\log(p_n(x))\right) \\ & +\sum_{x^n\in\mathcal{U}_n\backslash\mathcal{W}_n} \Pr[X^n=x^n]\sum_{x\in\mathcal{X}^-}N_x(x^n)\log(p_n(x)) + o_n(H(X^n)).
\end{align*}
Suppose $x\in\mathcal{X}^-$. We have that
\begin{align*}
\sum_{x^n\in\mathcal{U}_n\backslash\mathcal{W}_n} \Pr[X^n=x^n]N_x(x^n)  & \leq (\sum_{x^n\in\mathcal{U}_n\backslash\mathcal{W}_n} \text{Pr}[X^n=x^n])^{\frac{1}{2}}(\sum_{x^n\in\mathcal{U}_n\backslash\mathcal{W}_n} \text{Pr}[X^n=x^n]N_x(x^n)^2)^{\frac{1}{2}} \\
& \leq (\sum_{x^n\in\mathcal{U}_n\backslash\mathcal{W}_n} \text{Pr}[X^n=x^n])^{\frac{1}{2}}(n^2p_n(x)^2 + np_n(x)(1-p_n(x)))^{\frac{1}{2}} \\
& = o_n(np_n(x)).
\end{align*}
Hence
\[
\sum_{x^n\in\mathcal{U}_n\backslash\mathcal{W}_n} \Pr[X^n=x^n]\sum_{x\in\mathcal{X}^-}N_x(x^n)\log(p_n(x)) = o_n(H(X^n)).
\]
Therefore using \pref{eq:remainder} with $x^n=(x^-)^n$ gives that
\begin{align*}
& \sum_{x^n\in\mathcal{U}_n} \Pr[X^n=x^n]\sum_{x\in\mathcal{X}^-}N_x(x^n)\log(p_n(x)) \\ & = \Pr[X^n\in \mathcal{U}_n]\left(n\sum_{x\in\mathcal{X}^-} p_n(x)\log(p_n(x))\right) + o_n(H(X^n)) \\
& = \Pr[X^n\in\mathcal{U}_n]H(X^n) + o_n(H(X^n)).
\end{align*}
Then, using \pref{eq:decomp} gives that
\begin{equation}
\label{eq:entropyx}
H(X^n) =  -\sum_{y\in\mathcal{Y}_n} \Pr[X^n=\mathcal{A}_n(y)]\log(\Pr[X^n=\mathcal{A}_n(y)]) + \Pr[X^n\in\mathcal{U}_n]H(X^n) + o_n(H(X^n)).
\end{equation}

For all $y\in\mathcal{Y}_n$ we have that $\Pr[f_n(X^n)=y]\leq \frac{1}{e}$ if $n$ is sufficiently large and
\[
\Pr[f_n(X^n)=y] \geq \Pr[X^n=\mathcal{A}_n(y)].
\]
Afterwards using the fact that the function $-x\log(x)$ increases as $x$ increases over $[0, \frac{1}{e}]$ gives that
\begin{equation}
\label{eq:entropyf}
H(f_n(X^n))\geq -\sum_{y\in\mathcal{Y}_n} \Pr[X^n=\mathcal{A}(y)]\log(\Pr[X^n=\mathcal{A}(y)])
\end{equation}
if $n$ is sufficiently large. Afterwards using \pref{eq:entropyx} and \pref{eq:entropyf} gives that
\[
H(X^n)-H(f_n(X^n))\leq \Pr[X^n\in\mathcal{U}_n]H(X^n) + o_n(H(X^n)).
\]
This finishes the proof.
\end{proof}

\end{appendix}

\end{document}